\documentclass[12pt]{amsart}  

\usepackage[latin1]{inputenc}
\usepackage{amsmath} 
\usepackage{amsfonts}
\usepackage{amssymb}
\usepackage{stmaryrd}
\usepackage{latexsym} 
\usepackage{graphicx}
\usepackage{subfigure}
\usepackage{color}
\usepackage{hyperref}
\usepackage{verbatim}
\usepackage[all]{xy}
\usepackage{graphics}
\usepackage{pdfsync}
\usepackage{xcolor}
\usepackage{tikz}
\usepackage{bm}
\usetikzlibrary{arrows.meta}
\usetikzlibrary{shapes.geometric}

\theoremstyle{definition}
\newtheorem{theorem}{Theorem}[section]
\newtheorem{proposition}[theorem]{Proposition}

\newtheorem{lemma}[theorem]{Lemma}
\newtheorem{corollary}[theorem]{Corollary}

\newtheorem{definition}[theorem]{Definition}
\newtheorem{example}[theorem]{Example}

\newtheorem{definition/theorem}[theorem]{Definition/Theorem}

\theoremstyle{remark}
\newtheorem{remark}[theorem]{Remark}

\numberwithin{equation}{section}
\newlength\cellsize 
\savebox2{%
\begin{picture}(15,15)
\put(0,0){\line(1,0){15}}
\put(0,0){\line(0,1){15}}
\put(15,0){\line(0,1){15}}
\put(0,15){\line(1,0){15}}
\end{picture}}
\newcommand\cellify[1]{\def\thearg{#1}\def\nothing{}%
\ifx\thearg\nothing
\vrule width0pt height\cellsize depth0pt\else
\hbox to 0pt{\usebox2\hss}\fi%
\vbox to 15\unitlength{
\vss
\hbox to 15\unitlength{\hss$#1$\hss}
\vss}}
\newcommand\tableau[1]{\vtop{\let\\=\cr
\setlength\baselineskip{-16000pt}
\setlength\lineskiplimit{16000pt}
\setlength\lineskip{0pt}
\halign{&\cellify{##}\cr#1\crcr}}}
\savebox3{%
\begin{picture}(15,15)
\put(0,0){\line(1,0){15}}
\put(0,0){\line(0,1){15}}
\put(15,0){\line(0,1){15}}
\put(0,15){\line(1,0){15}}
\end{picture}}
\newcommand\expath[1]{%
\hbox to 0pt{\usebox3\hss}%
\vbox to 15\unitlength{
\vss
\hbox to 15\unitlength{\hss$#1$\hss}
\vss}}
\newcommand\bas[1]{\omit \vbox to \cellsize{ \vss \hbox to \cellsize{\hss$#1$\hss} \vss}}

\begin{document}

\title[A horizontal-strip LLT polynomial is determined by its weighted graph]{A horizontal-strip LLT polynomial is determined by its weighted graph}

\author{Foster Tom}

\thanks{
The author was supported in part by the Natural Sciences and Engineering Research Council of Canada.}

\begin{abstract}
We prove that two horizontal-strip LLT polynomials are equal if the associated weighted graphs defined by the author in a previous paper are isomorphic. This provides a sufficient condition for equality of horizontal-strip LLT polynomials and yields a well-defined LLT polynomial indexed by a weighted graph. We use this to prove some new relations between LLT polynomials and we explore a connection with extended chromatic symmetric functions.
\end{abstract}

\maketitle
\section{Introduction}\label{section:intro}

LLT polynomials have been studied extensively in algebraic combinatorics and representation theory. Horizontal-strip LLT polynomials generalize the Hall--Littlewood polynomials, which are the Frobenius series of cohomology rings of certain subsets of the flag variety \cite{cdm}. The Shuffle Theorem \cite{shuffle} of Carlsson and Mellit describes an LLT expansion of the Frobenius series of the space of diagonal harmonics. The Extended Delta Theorem of Blasiak, Haiman, Morse, Pun, and Seelinger \cite{extendeddelta} generalizes this to an LLT expansion of an infinite series of $\text{GL}_m$ characters using an action of the Schiffmann algebra on the space of symmetric functions. LLT polynomials appear positively in an expansion of Macdonald polynomials \cite{combmacdonald}, which implies that Macdonald polynomials are Schur-positive. LLT polynomials also arise in the representation theory of the quantum affine algebra \cite{lltoriginal} and of regular semisimple Hessenberg varieties via their connection to chromatic quasisymmetric functions \cite{hessenberg,chromquasihessenberg}. \\

If $\bm\lambda$ is a sequence of single cells, then the unicellular LLT polynomial $G_{\bm\lambda}(\bm x;q)$ can be expressed as a sum over arbitrary colourings of a unit interval graph $\Gamma(\bm\lambda)$ associated to $\bm\lambda$. Huh, Nam, and Yoo \cite{lltunicellschur} proved a combinatorial Schur expansion of $G_{\bm\lambda}(\bm x;q)$ whenever $\Gamma(\bm\lambda)$ is a ``melting lollipop'' and Alexandersson conjectured \cite{lltunicellepos} and then proved with Sulzgruber \cite{lltcombe} a combinatorial elementary symmetric function expansion of $G_{\bm\lambda}(\bm x;q+1)$ in terms of acyclic orientations of $\Gamma(\bm\lambda)$. An equality of unicellular LLT polynomials $G_{\bm\lambda}(\bm x;q)=G_{\bm\mu}(\bm x;q)$ is equivalent \cite{shuffle} to an equality of the corresponding chromatic quasisymmetric functions $X_{\Gamma(\bm\lambda)}(\bm x;q)=X_{\Gamma(\bm\mu)}(\bm x;q)$ introduced by Shareshian and Wachs \cite{chromposquasi}. Therefore, LLT polynomials are intimately connected to longstanding conjectures about equalities of chromatic symmetric functions, which is an area of active research \cite{propcatchrom,tuttechromsymeq,chromtreesnote}.\\

If $\bm\lambda$ is a sequence of rows, then the horizontal-strip LLT polynomial $G_{\bm\lambda}(\bm x;q)$ can be expressed as a sum over certain colourings of a unit interval graph $\tilde\Gamma(\bm\lambda)$ with some decorated edges and Alexandersson and Sulzgruber's result \cite{lltcombe} generalizes to this setting. In \cite{caterpillarllt} the author defined an alternative generalization of $\Gamma(\bm\lambda)$ to a weighted interval graph $\Pi(\bm\lambda)$. That paper gives a combinatorial Schur expansion of $G_{\bm\lambda}(\bm x;q)$ whenever $\Pi(\bm\lambda)$ is triangle-free and shows that the largest power of $q$ in $G_{\bm\lambda}(\bm x;q)$ is the total edge weight of $\Pi(\bm\lambda)$. \\

The main result of this paper, Theorem \ref{thm:main}, states that the horizontal-strip LLT polynomial $G_{\bm\lambda}(\bm x;q)$ is determined by the weighted graph $\Pi(\bm\lambda)$. In other words, if $\Pi(\bm\lambda)\cong\Pi(\bm\mu)$, then $G_{\bm\lambda}(\bm x;q)=G_{\bm\mu}(\bm x;q)$. In particular, this implies that if $\Pi$ is a weighted graph arising from this construction, then there corresponds a well-defined LLT polynomial $G_\Pi(\bm x;q)$. In Section \ref{section:maintheorem}, we prove Theorem \ref{thm:main} modulo a technical result, Lemma \ref{lem:key}, whose proof we postpone to Section \ref{section:keylemma}. In Section \ref{section:chromatic}, we explore a connection between $G_\Pi(\bm x;q)$ and the extended chromatic symmetric functions $X_\Pi(\bm x)$ associated to weighted graphs that were defined by Crew and Spirkl \cite{extendedchromsym} and whose relations were considered in \cite{extendedchromsymeq}.

\section{Background}\label{section:background}

A \emph{composition} $\alpha$ is a finite sequence of positive integers $\alpha=\alpha_1\cdots\alpha_\ell$. We denote by $\ell(\alpha)$ the \emph{length} of $\alpha$ and by convention, we set $\alpha_i=0$ if $i>\ell$. A \emph{partition} $\sigma$ is a composition that is weakly decreasing, that is $\sigma_1\geq\cdots\geq\sigma_\ell$. We also define the integer $n(\sigma)=\sum_i(i-1)\sigma_i$. If $\sigma$ and $\tau$ are partitions with $\sigma_i\geq\tau_i$ for every $i$, then the corresponding \emph{skew diagram} is \begin{equation}\lambda=\sigma/\tau=\{(i,j)\in\mathbb N\times\mathbb N: \ i\geq 1, \ \tau_i+1\leq j\leq \sigma_i\}. \end{equation} 
If $\tau$ is the empty partition, we write $\sigma$ instead of $\sigma/\emptyset$. The elements of $\lambda$ are called \emph{cells} and the \emph{content} of a cell $u=(i,j)\in\lambda$ is the integer $c(u)=j-i$. We primarily work with \emph{rows} and we assume the contents are nonnegative, so that they are skew diagrams of the form \begin{equation} R=a/b=\{(1,j)\in\mathbb N\times\mathbb N: \ b+1\leq j\leq a\}\end{equation}
for some $a\geq b\geq 0$. We denote by $l(R)=b$ and $r(R)=a-1$ the smallest and largest contents of cells of $R$. Note that $l(R)$ is the content of the leftmost cell of $R$, not the length of $R$, which is $|R|=r(R)-l(R)+1$. We also denote by $R^+=(a+1)/(b+1)$ and $R^-=(a-1)/(b-1)$ the rows obtained by shifting $R$ right or left respectively by one cell. A \emph{semistandard Young tableau (SSYT)} of shape $\lambda$ is a function $T:\lambda\to\{1,2,3,\ldots\}$ that satisfies $T_{i,j}\leq T_{i,j+1}\text{ and }T_{i,j}<T_{i+1,j}$, where we write $T_{i,j}$ instead of $T((i,j))$. A \emph{multiskew partition} is a sequence of skew diagrams $\bm\lambda=(\lambda^{(1)},\ldots,\lambda^{(n)})$. We say that $\bm\lambda$ is \emph{unicellular} if each $\lambda^{(i)}$ is a single cell and in keeping with the terminology of Alexandersson and Sulzgruber \cite{lltcombe}, we say that $\bm\lambda$ is a \emph{horizontal-strip} if each $\lambda^{(i)}$ is a row. If $\bm\lambda$ is a horizontal-strip, we denote by $\lambda$ the partition determined by the row lengths of $\bm\lambda$ and we define $n(\bm\lambda)=n(\lambda)$. We denote by \begin{equation}\text{SSYT}_{\bm\lambda}=\{\bm T=(T^{(1)},\ldots,T^{(n)}): \ T^{(i)}\in\text{SSYT}_{\lambda^{(i)}}\}\end{equation}
the set of sequences of SSYTs of shape $\bm\lambda$. Two entries $T^{(i)}(u)$ and $T^{(j)}(v)$ of $\bm T$ with $i<j$ form an \emph{inversion} if either
\begin{itemize}
\item $c(u)=c(v)$ and $T^{(i)}(u)>T^{(j)}(v)$, or
\item $c(u)=c(v)+1$ and $T^{(i)}(u)<T^{(j)}(v)$.
\end{itemize}
We denote by $\text{inv}(\bm T)$ the number of inversions of $\bm T$ and we denote by $\bm x^{\bm T}$ the monomial $x_1^{\#\text{ of }1\text{'s}}x_2^{\#\text{ of }2\text{'s}}x_3^{\#\text{ of }3\text{'s}}\cdots$. Now we define the \emph{LLT polynomial}
\begin{equation}G_{\bm\lambda}(\bm x;q)=\sum_{\bm T\in\text{SSYT}_{\bm\lambda}}q^{\text{inv}(\bm T)}\bm x^{\bm T}.\end{equation}

\begin{example}
The multiskew partition $\bm\lambda=(4/0,5/4,8/5,6/1)$, two SSYTs $\bm S$ and $\bm T$ of shape $\bm\lambda$ with their inversions marked by dashed lines, and the corresponding monomials of the LLT polynomial $G_{\bm\lambda}(\bm x;q)$ are given below. Because $\bm\lambda$ is a horizontal-strip, we draw it so that cells of the same content are aligned vertically. 

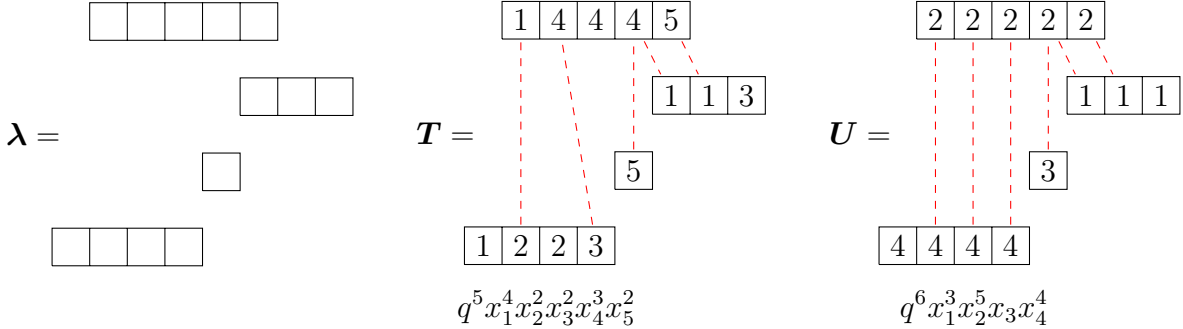
\begin{align*}
\begin{tikzpicture}
\draw (-0.25,0.25) node (0) {$\bm \lambda=$};
\draw (0,-1.5) -- (2,-1.5) -- (2,-1) -- (0,-1) -- (0,-1.5) (0.5,-1.5) -- (0.5,-1) (1,-1.5) -- (1,-1) (1.5,-1.5) -- (1.5,-1) (2,-0.5) -- (2.5,-0.5) -- (2.5,0) -- (2,0) -- (2,-0.5) (2.5,0.5) -- (4,0.5) -- (4,1) -- (2.5,1) -- (2.5,0.5) (3,0.5) -- (3,1) (3.5,0.5) -- (3.5,1) (0.5,1.5) -- (3,1.5) -- (3,2) -- (0.5,2) -- (0.5,1.5) (1,1.5) -- (1,2) (1.5,1.5) -- (1.5,2) (2,1.5) -- (2,2) (2.5,1.5) -- (2.5,2);
\end{tikzpicture}\hspace{20pt}
\begin{tikzpicture}
\draw (-0.25,0.25) node (0) {$\bm T=$};
\draw (0.25,-1.25) node (1) {$1$} (0.75,-1.25) node (2) {$2$} (1.25,-1.25) node (3) {$2$} (1.75,-1.25) node (4) {$3$} (2.25,-0.25) node (5) {$5$} (2.75,0.75) node (6) {$1$} (3.25,0.75) node (7) {$1$} (3.75,0.75) node (8) {$3$} (0.75,1.75) node (9) {$1$} (1.25,1.75) node (10) {$4$} (1.75,1.75) node (11) {$4$} (2.25,1.75) node (12) {$4$} (2.75,1.75) node (13) {$5$};
\draw (0,-1.5) -- (2,-1.5) -- (2,-1) -- (0,-1) -- (0,-1.5) (0.5,-1.5) -- (0.5,-1) (1,-1.5) -- (1,-1) (1.5,-1.5) -- (1.5,-1) (2,-0.5) -- (2.5,-0.5) -- (2.5,0) -- (2,0) -- (2,-0.5) (2.5,0.5) -- (4,0.5) -- (4,1) -- (2.5,1) -- (2.5,0.5) (3,0.5) -- (3,1) (3.5,0.5) -- (3.5,1) (0.5,1.5) -- (3,1.5) -- (3,2) -- (0.5,2) -- (0.5,1.5) (1,1.5) -- (1,2) (1.5,1.5) -- (1.5,2) (2,1.5) -- (2,2) (2.5,1.5) -- (2.5,2);
\draw [color=red, dashed] (2) -- (9) (4) -- (10) (5) -- (12) (12) -- (6) (13) -- (7);
\end{tikzpicture}\hspace{20pt}
\begin{tikzpicture}
\draw (-0.25,0.25) node (0) {$\bm U=$};
\draw (0.25,-1.25) node (1) {$4$} (0.75,-1.25) node (2) {$4$} (1.25,-1.25) node (3) {$4$} (1.75,-1.25) node (4) {$4$} (2.25,-0.25) node (5) {$3$} (2.75,0.75) node (6) {$1$} (3.25,0.75) node (7) {$1$} (3.75,0.75) node (8) {$1$} (0.75,1.75) node (9) {$2$} (1.25,1.75) node (10) {$2$} (1.75,1.75) node (11) {$2$} (2.25,1.75) node (12) {$2$} (2.75,1.75) node (13) {$2$};
\draw (0,-1.5) -- (2,-1.5) -- (2,-1) -- (0,-1) -- (0,-1.5) (0.5,-1.5) -- (0.5,-1) (1,-1.5) -- (1,-1) (1.5,-1.5) -- (1.5,-1) (2,-0.5) -- (2.5,-0.5) -- (2.5,0) -- (2,0) -- (2,-0.5) (2.5,0.5) -- (4,0.5) -- (4,1) -- (2.5,1) -- (2.5,0.5) (3,0.5) -- (3,1) (3.5,0.5) -- (3.5,1) (0.5,1.5) -- (3,1.5) -- (3,2) -- (0.5,2) -- (0.5,1.5) (1,1.5) -- (1,2) (1.5,1.5) -- (1.5,2) (2,1.5) -- (2,2) (2.5,1.5) -- (2.5,2);
\draw [color=red, dashed] (2) -- (9) (3) -- (10) (4) -- (11) (5) -- (12) (12) -- (6) (13) -- (7);
\end{tikzpicture}\\ q^5x_1^4x_2^2x_3^2x_4^3x_5^2 \hspace{100pt} q^6x_1^3x_2^5x_3x_4^4\hspace{50pt}\end{align*}

The LLT polynomial $G_{(4/0,5/4,8/5,6/1)}(\bm x;q)$ can be expanded in the Schur function basis as 
\begin{align*}G_{\bm\lambda}(\bm x;q)&=q^6s_{5 4 3 1} + q^6s_{5 4 4} + q^6s_{5 5 2 1} + 2q^6s_{5 5 3} + q^6s_{6 3 3 1} + 2q^6s_{6 4 2 1} + (3q^6+q^5)s_{6 4 3} \\&+ 2q^6s_{6 5 1 1} + (4q^6+q^5)s_{6 5 2} + (2q^6+q^5)s_{6 6 1} + (q^6+q^5)s_{7 3 2 1} + (q^6+2q^5)s_{7 3 3} \\&+ (q^6+2q^5)s_{7 4 1 1} + (2q^6+5q^5)s_{7 4 2} + (2q^6+6q^5)s_{7 5 1} + 4q^5s_{7 6} + q^5s_{8 2 2 1} \\&+ (2q^5+q^4)s_{8 3 1 1} + (4q^5+2q^4)s_{8 3 2} + (5q^5+5q^4)s_{8 4 1} + (3q^5+4q^4)s_{8 5} + 2q^4s_{9 2 1 1} \\&+ 3q^4s_{9 2 2} + (7q^4+2q^3)s_{9 3 1} + (5q^4+3q^3)s_{9 4} + q^3s_{(10)1 1 1} + 6q^3s_{(10)2 1} \\&+ (6q^3+q^2)s_{(10) 3} + 3q^2s_{(11) 1 1} + 5q^2s_{(11)2} + 3qs_{(12) 1} + s_{(13)}\end{align*}

\end{example}

LLT polynomials are symmetric functions \cite[Theorem 6.1]{lltoriginal} and moreover are \emph{Schur-positive}, \cite[Theorem 3.1.3]{combdiag} \cite{lrkl} \cite[Corollary 6.9]{lltpos}, meaning that they can be expressed as a linear combination of Schur functions where the coefficients are polynomials in $q$ with positive coefficients. Finding a combinatorial formula for this Schur expansion is a major open problem of active research \cite{lltunicellepos,lltchrom,lltcombe,lltmpath,lltunicellschur,caterpillarllt}. \\

When $\bm\lambda=(\lambda^{(1)},\ldots,\lambda^{(n)})$ is unicellular, we can associate to $\bm\lambda$ a labelled graph $\Gamma(\bm\lambda)$ with vertices $v_1,\ldots,v_n$ as follows. We label the cells of $\bm\lambda$ as $1,2,3,\ldots$ in \emph{reverse content reading order}, meaning in order of decreasing content and from top to bottom along constant content lines. Then vertices $v_i$ and $v_j$ are joined by an edge if it is possible for entries in cells $i$ and $j$ to form an inversion. This approach was effective in finding combinatorial formulas for unicellular LLT polynomials \cite{lltunicellepos,lltchrom,lltunicellschur}. Graphs arising from this construction are called \emph{unit interval graphs} and they have several equivalent characterizations \cite{uigraphs}. 

\begin{example} \label{ex:gamma} The unicellular multiskew partitions $\bm\lambda=(2/1,1/0,1/0,2/1,2/1)$ and $\bm\mu=(1/0,2/1,2/1,1/0,2/1)$ are given below with their cells labelled in reverse content reading order, along with the associated labelled graphs $\Gamma(\bm\lambda)$ and $\Gamma(\bm\mu)$. 
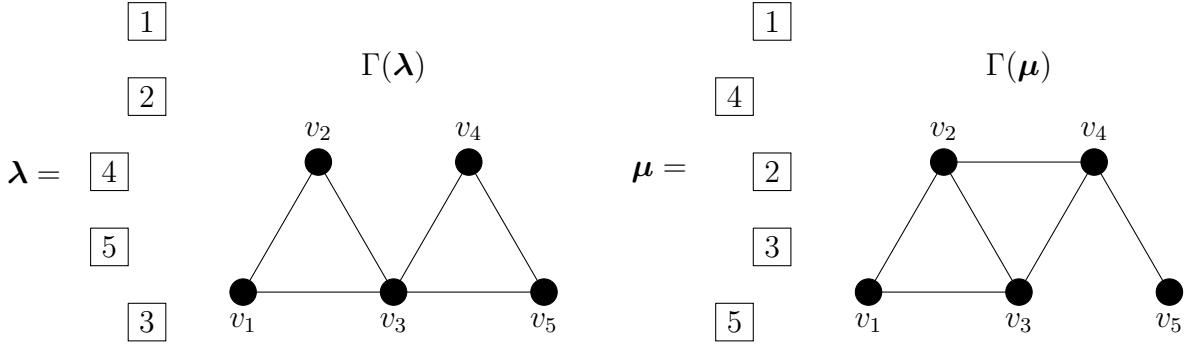
$$
\begin{tikzpicture}
\draw (-0.75,0.75) node {$\bm\lambda=$};
\draw (0.25,-0.25) node (1) {5} (0.25,0.75) node (2) {4} (0.75,-1.25) node (3) {3} (0.75,1.75) node (4) {2} (0.75,2.75) node (5) {1};
\draw (0,0) -- (0.5,0) -- (0.5,-0.5) -- (0,-0.5) -- (0,0) (0,0.5) -- (0.5,0.5) -- (0.5,1) -- (0,1) -- (0,0.5) (0.5,-1) -- (0.5,-1.5) -- (1,-1.5) -- (1,-1) -- (0.5,-1) (0.5,1.5) -- (0.5,2) -- (1,2) -- (1,1.5) -- (0.5,1.5) (0.5,2.5) -- (0.5,3) -- (1,3) -- (1,2.5) -- (0.5,2.5);
\end{tikzpicture} \hspace{20pt}
\begin{tikzpicture}
\draw (2,3) node {$\Gamma(\bm\lambda)$};
\filldraw (0,0) circle (5pt) node[align=center, below] (1){};
\filldraw (1,1.73) circle (5pt) node[align=center, above] (2){};
\filldraw (2,0) circle (5pt) node[align=center, below] (3){};
\filldraw (3,1.73) circle (5pt) node[align=center, above] (4){};
\filldraw (4,0) circle (5pt) node[align=center,below] (5){};
\node [below] at (1) {$v_1$};
\node [above] at (2) {$v_2$};
\node [below] at (3) {$v_3$};
\node [above] at (4) {$v_4$};
\node [below] at (5) {$v_5$};
\draw (0,0) -- (1,1.73) (1,1.73) -- (2,0) (0,0) -- (2,0) (2,0) -- (3,1.73) (2,0) -- (4,0) (3,1.73) -- (4,0);
\end{tikzpicture}\hspace{20pt}
\begin{tikzpicture}
\draw (-0.75,0.75) node {$\bm\mu=$};
\draw (0.75,-0.25) node (3) {3} (0.75,0.75) node (4) {2} (0.25,-1.25) node (1) {5} (0.25,1.75) node (2) {4} (0.75,2.75) node (5) {1};
\draw (0.5,0) -- (1,0) -- (1,-0.5) -- (0.5,-0.5) -- (0.5,0) (0.5,0.5) -- (1,0.5) -- (1,1) -- (0.5,1) -- (0.5,0.5) (0,-1) -- (0,-1.5) -- (0.5,-1.5) -- (0.5,-1) -- (0,-1) (0,1.5) -- (0,2) -- (0.5,2) -- (0.5,1.5) -- (0,1.5) (0.5,2.5) -- (0.5,3) -- (1,3) -- (1,2.5) -- (0.5,2.5);
\end{tikzpicture} \hspace{20pt}
\begin{tikzpicture}
\draw (2,3) node {$\Gamma(\bm\mu)$};
\filldraw (0,0) circle (5pt) node[align=center, below] (1){};
\filldraw (1,1.73) circle (5pt) node[align=center, above] (2){};
\filldraw (2,0) circle (5pt) node[align=center, below] (3){};
\filldraw (3,1.73) circle (5pt) node[align=center, above] (4){};
\filldraw (4,0) circle (5pt) node[align=center,below] (5){};
\node [below] at (1) {$v_1$};
\node [above] at (2) {$v_2$};
\node [below] at (3) {$v_3$};
\node [above] at (4) {$v_4$};
\node [below] at (5) {$v_5$};
\draw (0,0) -- (1,1.73) (0,0) -- (2,0) (1,1.73) -- (2,0) (1,1.73) -- (3,1.73) (2,0) -- (3,1.73) (3,1.73) -- (4,0);
\end{tikzpicture}
$$
\end{example}

When $\bm\lambda$ is unicellular, a tableau $\bm T\in\text{SSYT}_{\bm\lambda}$ is precisely a (possibly not proper) colouring $\kappa:\Gamma(\bm\lambda)\to\mathbb N$ and an inversion of $\bm T$ is an \emph{ascent} of $\kappa$, namely an edge $(v_i,v_j)\in E(\Gamma(\bm\lambda))$ with $i<j$ and $\kappa(v_i)<\kappa(v_j)$. Therefore, we can express $G_{\bm\lambda}(\bm x;q)$ in terms of $\Gamma(\bm\lambda)$ as 
\begin{equation} \label{eq:lltunicellular} G_{\bm\lambda}(\bm x;q)=\sum_{\substack{\kappa:\Gamma(\bm\lambda)\to\mathbb N\\\kappa\text{ arbitrary}}}q^{\text{asc}(\kappa)}\bm x^\kappa.\end{equation}
If we restrict the sum \eqref{eq:lltunicellular} to \emph{proper colourings}, meaning that $\kappa(v_i)\neq\kappa(v_j)$ whenever $(v_i,v_j)\in E(\Gamma(\bm\lambda))$, then we have the \emph{chromatic quasisymmetric function} $X_{\Gamma(\bm\lambda)}(\bm x;q)$. Unicellular LLT polynomials and chromatic quasisymmetric functions are related by a change of variables, namely the plethystic relationship \cite[Proposition 3.4]{shuffle} \begin{equation} \label{eq:unichrom} X_{\Gamma(\bm\lambda)}(\bm x;q)=(q-1)^{-|\Gamma(\bm\lambda)|}G_{\bm\lambda}[\bm x(q-1);q],\end{equation} 
which in particular implies that if $\bm\lambda$ and $\bm\mu$ are unicellular, then \begin{equation} \label{eq:equivuni} X_{\Gamma(\bm\lambda)}(\bm x;q)=X_{\Gamma(\bm\mu)}(\bm x;q)\text{ if and only if }G_{\bm\lambda}(\bm x;q)=G_{\bm\mu}(\bm x;q).\end{equation}
For example, the unicellular multiskew partitions $\bm\lambda$ and $\bm\mu$ from Example \ref{ex:gamma} have the same LLT polynomials and their labelled unit interval graphs have the same chromatic quasisymmetric functions. Because of this relationship, unicellular LLT polynomials may be a fruitful approach to the Stanley--Stembridge conjecture that chromatic symmetric functions of unit interval graphs are $e$-positive, or its refinement, the Shareshian--Wachs conjecture, which asks for a combinatorial elementary function expansion of the chromatic quasisymmetric function.\\

When $\bm\lambda=(R_1,\ldots,R_n)$ is a horizontal-strip, one generalization of the above construction is to include a set $S$ of special edges that record when cells of $\bm\lambda$ are in the same row. Then we can express the LLT polynomial in terms of this decorated graph $\tilde\Gamma(\bm\lambda)$ as \begin{equation}G_{\bm\lambda}(\bm x;q)=\sum_{\substack{\kappa:\tilde\Gamma(\bm\lambda)\to\mathbb N\\\kappa(v_i)\leq\kappa(v_j)\text{ if }i<j, \ (v_i,v_j)\in S}}q^{\text{asc}(\kappa)}\bm x^\kappa.\end{equation}
This approach was successfully employed by Alexandersson and Sulzgruber \cite{lltcombe}. 

\begin{example} The horizontal-strip $\bm\lambda=(4/0,5/4,8/5,6/1)$ and its associated decorated labelled graph $\tilde \Gamma(\bm\lambda)$ are given below.  
$$
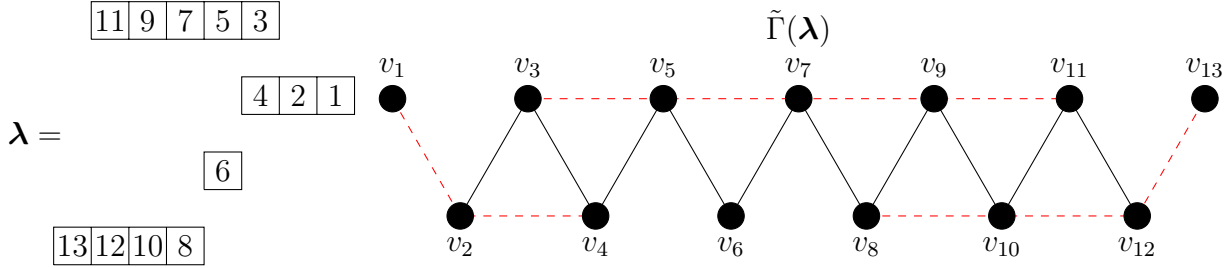
\begin{tikzpicture}
\draw (-0.25,0.25) node {$\bm\lambda=$};
\draw (0.25,-1.25) node (1) {13} (0.75,-1.25) node (2) {12} (0.75,1.75) node (3) {11} (1.25,-1.25) node (4) {10} (1.25,1.75) node (5) {9} (1.75,-1.25) node (6) {8} (1.75,1.75) node (7) {7} (2.25,-0.25) node (8) {6} (2.25,1.75) node (9) {5} (2.75,0.75) node (10) {4} (2.75,1.75) node (11) {3} (3.25,0.75) node (12) {2} (3.75,0.75) node (13) {1};
\draw (0,-1.5) -- (2,-1.5) -- (2,-1) -- (0,-1) -- (0,-1.5) (0.5,-1.5) -- (0.5,-1) (1,-1.5) -- (1,-1) (1.5,-1.5) -- (1.5,-1) (2,-0.5) -- (2.5,-0.5) -- (2.5,0) -- (2,0) -- (2,-0.5) (2.5,0.5) -- (4,0.5) -- (4,1) -- (2.5,1) -- (2.5,0.5) (3,0.5) -- (3,1) (3.5,0.5) -- (3.5,1) (0.5,1.5) -- (3,1.5) -- (3,2) -- (0.5,2) -- (0.5,1.5) (1,1.5) -- (1,2) (1.5,1.5) -- (1.5,2) (2,1.5) -- (2,2) (2.5,1.5) -- (2.5,2);
\end{tikzpicture}\hspace{5pt}
\begin{tikzpicture}
\draw (6.4,2.5) node {$\tilde\Gamma(\bm\lambda)$};
\draw [color=red, dashed] (1,1.56) -- (1.9,0) -- (3.7,0) (2.8,1.56) -- (4.6,1.56) -- (6.4,1.56) -- (8.2,1.56)-- (10,1.56) (7.3,0) -- (9.1,0) -- (10.9,0) -- (11.8,1.56);
\filldraw (1,1.56) circle (5pt) node[align=center, above] (1){};
\filldraw (1.9,0) circle (5pt) node[align=center, below] (2){};
\filldraw (2.8,1.56) circle (5pt) node[align=center, above] (3){};
\filldraw (3.7,0) circle (5pt) node[align=center,below] (4){};
\filldraw (4.6,1.56) circle (5pt) node[align=center,above] (5){};
\filldraw (5.5,0) circle (5pt) node[align=center,below] (6){};
\filldraw (6.4,1.56) circle (5pt) node[align=center,above] (7){};
\filldraw (7.3,0) circle (5pt) node[align=center,below] (8){};
\filldraw (8.2,1.56) circle (5pt) node[align=center,above] (9){};
\filldraw (9.1,0) circle (5pt) node[align=center,below] (10){};
\filldraw (10,1.56) circle (5pt) node[align=center,above] (11){};
\filldraw (10.9,0) circle (5pt) node[align=center,below] (12){};
\filldraw (11.8,1.56) circle (5pt) node[align=center,above] (13){};
\node [above] at (1) {$v_1$};
\node [below] at (2) {$v_2$};
\node [above] at (3) {$v_3$};
\node [below] at (4) {$v_4$};
\node [above] at (5) {$v_5$};
\node [below] at (6) {$v_6$};
\node [above] at (7) {$v_7$};
\node [below] at (8) {$v_8$};
\node [above] at (9) {$v_9$};
\node [below] at (10) {$v_{10}$};
\node [above] at (11) {$v_{11}$};
\node [below] at (12) {$v_{12}$};
\node [above] at (13) {$v_{13}$};
\draw (1.9,0) -- (2.8,1.56) -- (3.7,0) -- (4.6,1.56) -- (5.5,0) -- (6.4,1.56) -- (7.3,0) -- (8.2,1.56) -- (9.1,0) -- (10,1.56) -- (10.9,0);
\end{tikzpicture}
$$
\end{example}

We will be interested in the following \emph{weighted graph} $\Pi(\bm\lambda)$ that the author defined as an  alternative generalization of $\Gamma(\bm\lambda)$ to the horizontal-strip case.

\begin{definition} \cite[Definition 3.1, Definition 3.2]{caterpillarllt}
Let $\bm\lambda=(R_1,\ldots,R_n)$ be a horizontal-strip. For $1\leq i<j\leq n$ we define the integers \begin{equation}M(R_i,R_j)=\begin{cases}|R_i\cap R_j| & \text{ if }l(R_i)\leq l(R_j),\\ |R_i\cap R_j^+| & \text{ if }l(R_i)>l(R_j).\end{cases}\end{equation}
We then define a weighted graph $\Pi(\bm\lambda)$ whose vertices are the rows of $\bm\lambda$. The weight of a row $R_i$ is the number of cells $|R_i|$ and rows $R_i$ and $R_j$ with $i<j$ are joined by an edge of weight $M(R_i,R_j)$. By convention, we omit edges of weight zero.
\end{definition}

\begin{remark} In \cite[Definition 3.2]{caterpillarllt} the author also defined an ordering of the vertices of $\Pi(\bm\lambda)$ based on the content reading order of the rightmost cells in the rows of $\bm\lambda$. However, because our main Theorem states that the corresponding LLT polynomial does not depend on this labelling, we will not be interested in it. Also note that if $\bm\lambda$ is unicellular, then the graph $\Pi(\bm\lambda)$ is simply $\Gamma(\bm\lambda)$ with all vertex and edge weights one.
\end{remark}

\begin{example} \label{ex:wgraph} The horizontal-strip $\bm\lambda=(4/0,5/4,8/5,6/1)$ and the weighted graph $\Pi(\bm\lambda)$ are given below. We have $M(R_1,R_4)=3$, $M(R_2,R_4)=1$, and $M(R_3,R_4)=2$. We have also drawn the horizontal-strip $\bm\mu=(5/4,9/5,7/2,3/0)$, whose weighted graph $\Pi(\bm\mu)$ is isomorphic to $\Pi(\bm\lambda)$.

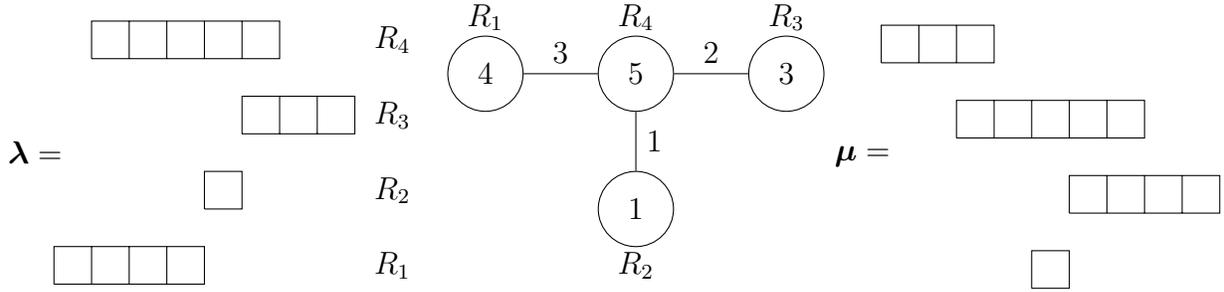
$$
\begin{tikzpicture}
\draw (-0.25,0.25) node (0) {$\bm \lambda=$};
\draw (4.5,-1.25) node (1) {$R_1$} (4.5,-0.25) node (2) {$R_2$} (4.5,0.75) node (3) {$R_3$} (4.5,1.75) node (4) {$R_4$};
\draw (0,-1.5) -- (2,-1.5) -- (2,-1) -- (0,-1) -- (0,-1.5) (0.5,-1.5) -- (0.5,-1) (1,-1.5) -- (1,-1) (1.5,-1.5) -- (1.5,-1) (2,-0.5) -- (2.5,-0.5) -- (2.5,0) -- (2,0) -- (2,-0.5) (2.5,0.5) -- (4,0.5) -- (4,1) -- (2.5,1) -- (2.5,0.5) (3,0.5) -- (3,1) (3.5,0.5) -- (3.5,1) (0.5,1.5) -- (3,1.5) -- (3,2) -- (0.5,2) -- (0.5,1.5) (1,1.5) -- (1,2) (1.5,1.5) -- (1.5,2) (2,1.5) -- (2,2) (2.5,1.5) -- (2.5,2);
\end{tikzpicture}\hspace{10pt}
\begin{tikzpicture}
\node[shape=circle,draw=black,minimum size=10mm](1) at (0,0) {$4$};
\node at (0,0.75) {$R_1$};
\node[shape=circle,draw=black,minimum size=10mm](2) at (2,0) {$5$};
\node at (2,0.75) {$R_4$};
\node[shape=circle,draw=black,minimum size=10mm](3) at (2,-1.8) {$1$};
\node at (2,-2.55) {$R_2$};
\node[shape=circle,draw=black,minimum size=10mm](4) at (4,0) {$3$};
\node at (4,0.75) {$R_3$};
\path [-](1) edge node [above]{$3$} (2);
\path [-](2) edge node [right]{$1$} (3);
\path [-](2) edge node [above]{$2$} (4);
\end{tikzpicture}
\begin{tikzpicture}
\draw (-0.25,0.25) node (0) {$\bm \mu=$};
\draw (2,-1.5) -- (2.5,-1.5) -- (2.5,-1) -- (2,-1) -- (2,-1.5) (2.5,-0.5) -- (4.5,-0.5) -- (4.5,0) -- (2.5,0) -- (2.5,-0.5) (3,-0.5) -- (3,0) (3.5,-0.5) -- (3.5,0) (4,-0.5) -- (4,0) (1,0.5) -- (3.5,0.5) -- (3.5,1) -- (1,1) -- (1,0.5) (1.5,0.5) -- (1.5,1) (2,0.5) -- (2,1) (2.5,0.5) -- (2.5,1) (3,0.5) -- (3,1) (0,1.5) -- (1.5,1.5) -- (1.5,2) -- (0,2) -- (0,1.5) (0.5,1.5) -- (0.5,2) (1,1.5) -- (1,2);
\end{tikzpicture}$$

\end{example}

We can now state our main Theorem. We will state the key Lemma \ref{lem:key} in Section \ref{section:maintheorem} and we will show how it implies Theorem \ref{thm:main}. We will then prove Lemma \ref{lem:key} in Section \ref{section:keylemma}. 

\begin{theorem} \label{thm:main} Let $\bm\lambda$ and $\bm\mu$ be horizontal-strips. If the weighted graphs $\Pi(\bm\lambda)$ and $\Pi(\bm\mu)$ are isomorphic, then the LLT polynomials $G_{\bm\lambda}(\bm x;q)$ and $G_{\bm\mu}(\bm x;q)$ are equal. In other words,
\begin{equation}
\text{ if }\Pi(\bm\lambda)\cong\Pi(\bm\mu),\text{ then }G_{\bm\lambda}(\bm x;q)=G_{\bm\mu}(\bm x;q).
\end{equation}\end{theorem}

\begin{example} The horizontal-strips $\bm\lambda=(4/0,5/4,8/5,6/1)$ and $\bm\mu=(5/4,9/5,7/2,3/0)$ in Example \ref{ex:wgraph} have isomorphic weighted graphs, so it follows from Theorem \ref{thm:main} that $G_{\bm\lambda}(\bm x;q)=G_{\bm\mu}(\bm x;q)$. \end{example}

We define a weighted graph $\Pi$ to be \emph{admissible} if it arises from our construction, in other words if $\Pi\cong\Pi(\bm\lambda)$ for some horizontal-strip $\bm\lambda$. Then an alternative restatement of Theorem \ref{thm:main} is that if $\Pi$ is admissible, then there is a well-defined LLT polynomial $G_\Pi(\bm x;q)$ given by $G_\Pi(\bm x;q)=G_{\bm\lambda}(\bm x;q)$, where $\bm\lambda$ is any horizontal-strip with $\Pi\cong\Pi(\bm\lambda)$. \\

Theorem \ref{thm:main} justifies our choice to associate to $\bm\lambda$ this mysterious weighted graph $\Pi(\bm\lambda)$. We now mention some advantages of this definition.
\begin{enumerate}
\item \cite[Theorem 4.6]{caterpillarllt} If $\Pi(\bm\lambda)$ is triangle-free, then there is a combinatorial Schur expansion of $G_{\bm\lambda}(\bm x;q)$ defined purely in terms of $\Pi(\bm\lambda)$, but not $\bm\lambda$ itself.\\

\item \cite[Example 4.10]{caterpillarllt} As an example of (1), if we let $P_{ab}(M)$ denote a graph consisting of two vertices of weights $a$ and $b$, joined by an edge of weight $M$, where we may assume that $a\geq b\geq M$ without loss of generality, then the Schur function expansion of the LLT polynomial $G_{P_{ab}(M)}(\bm x;q)$ can be expressed in terms of $a$, $b$, and $M$ as\begin{equation} \label{eq:gpi2} G_{P_{ab}(M)}=\sum_{k=0}^bq^{\min\{M,k\}}s_{(a+b-k)k}=s_{(a+b)}+qs_{(a+b-1)1}+\cdots+q^Ms_{(a+b-M)M}+\cdots+q^Ms_{ab}.\end{equation}
\item \cite[Lemma 3.15]{caterpillarllt} If $\bm\lambda=(R_1,\ldots R_i,R_{i+1},\ldots,R_n)$ and $\bm\mu=(R_1,\ldots,R_{i+1},R_i,\ldots,R_n)$ are horizontal-strips that differ by switching a pair of adjacent rows, then we have \begin{equation}G_{\bm\lambda}(\bm x;q)=G_{\bm\mu}(\bm x;q)\text{ if and only if }\Pi(\bm\lambda)\cong\Pi(\bm\mu).\end{equation} 
In other words, equalities of LLT polynomials in this case are characterized by the associated weighted graphs. \\

\item \cite[Theorem 3.5]{caterpillarllt} The total edge weight of $\Pi(\bm\lambda)$, which we denote $M(\bm\lambda)$, is equal to the largest power of $q$ in $G_{\bm\lambda}(\bm x;q)$. By contrast, there are horizontal-strips $\bm\lambda$ and $\bm\mu$ for which $\tilde\Gamma(\bm\lambda)$ and $\tilde\Gamma(\bm\mu)$ have different numbers of edges, and yet $G_{\bm\lambda}(\bm x;q)=G_{\bm\mu}(\bm x;q)$. To be specific, let $\bm\lambda=(2/1,2/0)$ and $\bm\mu=(2/0,2/1)$ as below.

$$
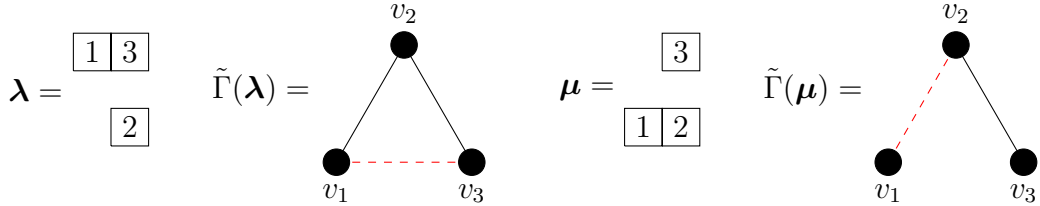
\begin{tikzpicture}
\draw (-0.5,0.75) node {$\bm\lambda=$};
\draw (0.25,1.25) node (1) {1} (0.75,0.25) node (2) {2} (0.75,1.25) node (3) {3} (0.25,-0.75) node {};
\draw (0,1) -- (0.5,1) -- (0.5,1.5) -- (0,1.5) -- (0,1) (0.5,0) -- (0.5,0.5) -- (1,0.5) -- (1,0) -- (0.5,0) (0.5,1) -- (0.5,1.5) -- (1,1.5) -- (1,1) -- (0.5,1);
\end{tikzpicture} \hspace{20pt}
\begin{tikzpicture}
\draw (-1,0.5) node {$\tilde\Gamma(\bm\lambda)=$};
\draw (0,-0.5) -- (0.9,1.06) -- (1.8,-0.5);
\draw [color=red, dashed] (0,-0.5) -- (1.8,-0.5);
\filldraw (0,-0.5) circle (5pt) node[align=center, below] (1){};
\filldraw (0.9,1.06) circle (5pt) node[align=center, above] (2){};
\filldraw (1.8,-0.5) circle (5pt) node[align=center, below] (3){};
\node [below] at (1) {$v_1$};
\node [above] at (2) {$v_2$};
\node [below] at (3) {$v_3$};
\end{tikzpicture}\hspace{20pt}
\begin{tikzpicture}
\draw (-0.5,0.75) node {$\bm\mu=$};
\draw (0.25,0.25) node (1) {1} (0.75,0.25) node (2) {2} (0.75,1.25) node (3) {3} (0.25,-0.75) node {};
\draw (0,0) -- (0.5,0) -- (0.5,0.5) -- (0,0.5) -- (0,0) (0.5,0) -- (0.5,0.5) -- (1,0.5) -- (1,0) -- (0.5,0) (0.5,1) -- (0.5,1.5) -- (1,1.5) -- (1,1) -- (0.5,1);
\end{tikzpicture} \hspace{20pt}
\begin{tikzpicture}
\draw (-1,0.5) node {$\tilde\Gamma(\bm\mu)=$};
\draw (0.9,1.06) -- (1.8,-0.5);
\draw [color=red, dashed] (0,-0.5) -- (0.9,1.06);
\filldraw (0,-0.5) circle (5pt) node[align=center, below] (1){};
\filldraw (0.9,1.06) circle (5pt) node[align=center, above] (2){};
\filldraw (1.8,-0.5) circle (5pt) node[align=center, below] (3){};
\node [below] at (1) {$v_1$};
\node [above] at (2) {$v_2$};
\node [below] at (3) {$v_3$};
\end{tikzpicture}
$$

Note that $\tilde\Gamma(\bm\lambda)$ has two normal edges and one special edge, while $\tilde\Gamma(\bm\mu)$ has one normal edge and one special edge. Meanwhile, we have $\Pi(\bm\lambda)\cong\Pi(\bm\mu)\cong P_{21}(1)$, so by \eqref{eq:gpi2} we have $G_{\bm\lambda}(\bm x;q)=G_{\bm\mu}(\bm x;q)=s_3+qs_{21}$.\\
\end{enumerate}

We saw in Example \ref{ex:gamma} that even in the unicellular case, the converse to Theorem \ref{thm:main} does not hold because there are horizontal-strips $\bm\lambda$ and $\bm\mu$ with $G_{\bm\lambda}(\bm x;q)=G_{\bm\mu}(\bm x;q)$ but $\Pi(\bm\lambda)\ncong\Pi(\bm\mu)$. However, we hope that such cases underlie a deep connection to equalities of chromatic symmetric functions, which are an area of active research \cite{extendedchromsymeq,propcatchrom,tuttechromsymeq,chromtreesnote}. We explore this connection in Theorem \ref{thm:mpath} and Theorem \ref{thm:compgraph} in Section \ref{section:chromatic}. 

\section{Proof of Theorem \ref{thm:main}} \label{section:maintheorem}

In this section we prove Theorem \ref{thm:main}, modulo a technical result, Lemma \ref{lem:key}, whose proof we postpone to Section \ref{section:keylemma}. The general idea is to rewrite $\bm\lambda$ and $\bm\mu$ in a more convenient form while preserving their weighted graphs and LLT polynomials. The following definition will be convenient to describe the relationship between rows within a fixed horizontal-strip.

\begin{definition} Let $\bm\lambda=(R_1,\ldots,R_n)$ be a horizontal-strip and $1\leq i,j\leq n$ with $i\neq j$. We define $M_{i,j}(\bm\lambda)$ to be the number of edges in $\Pi(\bm\lambda)$ joining $R_i$ and $R_j$, that is \begin{equation}M_{i,j}(\bm\lambda)=\begin{cases}M(R_i,R_j)&\text{ if }i<j,\\M(R_j,R_i)&\text{ if }i>j.\end{cases}\end{equation}
We abbreviate $M_{i,j}(\bm\lambda)$ as $M_{i,j}$ if the context is clear. Note that \begin{equation}\label{eq:mijbound}0\leq M_{i,j}\leq\min\{|R_i|,|R_j|\}.\end{equation} \end{definition}

We also take this opportunity to formally define an isomorphism of weighted graphs.

\begin{definition}
Let $\bm\lambda=(R_1,\ldots,R_n)$ and $\bm\mu=(S_1,\ldots,S_n)$ be horizontal-strips. An \emph{isomorphism of weighted graphs}, denoted $\varphi:\Pi(\bm\lambda)\xrightarrow\sim\Pi(\bm\mu)$, is a permutation \begin{equation}\varphi:\{1,\ldots,n\}\to\{1,\ldots,n\}\end{equation}
that satisfies $|R_i|=|S_{\varphi_i}|$ and $M_{i,j}(\bm\lambda)=M_{\varphi_i,\varphi_j}(\bm\mu)$ for all $i$ and $j$. The weighted graphs $\Pi(\bm\lambda)$ and $\Pi(\bm\mu)$ are \emph{isomorphic}, denoted $\Pi(\bm\lambda)\cong\Pi(\bm\mu)$, if there exists such an isomorphism. We say that $\bm\lambda$ and $\bm\mu$ are \emph{similar} if $\Pi(\bm\lambda)\cong\Pi(\bm\mu)$ and $G_{\bm\lambda}(\bm x;q)=G_{\bm\mu}(\bm x;q)$ and we denote by $\mathcal S(\bm\lambda)$ the set of horizontal-strips that are similar to $\bm\lambda$.  
\end{definition}

\begin{remark} Theorem \ref{thm:main} states that $\Pi(\bm\lambda)\cong\Pi(\bm\mu)$ implies that $G_{\bm\lambda}(\bm x;q)=G_{\bm\mu}(\bm x;q)$, but until we have proven this it will be convenient to temporarily define this concept of similarity. \end{remark}

\begin{example}
In Example \ref{ex:wgraph}, we have that $\varphi:\Pi(\bm\lambda)\xrightarrow\sim\Pi(\bm\mu)$, where $\varphi$ is the permutation given by $\varphi_1=2$, $\varphi_2=1$, $\varphi_3=4$, and $\varphi_4=3$.
\end{example}

Note that similarity is an equivalence relation. We now describe some operations we can perform on a horizontal-strip while preserving similarity. 

\begin{proposition} \label{prop:similar} Let $\bm\lambda=(R_1,\ldots,R_n)$ be a horizontal-strip. 
\begin{enumerate}
\item Let $\bm\lambda^+=(R_1^+,\ldots,R_n^+)$ and $\bm\lambda^-=(R_1^-,\ldots,R_n^-)$ be the horizontal-strips obtained by translating all rows right by one cell or left by one cell respectively. Then we have $\bm\lambda^+,\bm\lambda^-\in\mathcal S(\bm\lambda)$. \\

\item Define the \emph{cycle} of $\bm\lambda$ to be $\kappa(\bm\lambda)=(R_2,\ldots,R_n,R_1^-)$. Then we have $\kappa(\bm\lambda)\in\mathcal S(\bm\lambda)$. \\

\item For a sufficiently large integer $N$, define the \emph{rotation} of $\bm\lambda$ to be $N-\bm\lambda=(N-R_n,\ldots,N-R_1)$, where $N-R=\{(1,N-j): \ j\in R\}$. Then we have $N-\bm\lambda\in\mathcal S(\bm\lambda)$.
\end{enumerate}
\end{proposition}

\begin{remark}
In (3), we take $N$ to be sufficiently large simply because we assume that our cells have nonnegative content. Because of (1), the precise value of $N$ will not matter to us.
\end{remark}

\begin{proof}[Proof of Proposition \ref{prop:similar}. ]\hspace{2pt}
\begin{enumerate}
\item This follows directly from the definition because $\Pi(\bm\lambda)$ and $G_{\bm\lambda}(\bm x;q)$ only depend on the relative positions of the rows of $\bm\lambda$. \\

\item This follows directly from the definition because $M(R_i,R_j)=M(R_j,R_i^-)$ and because the condition for a cell $u\in R_i$ for $i\geq 2$ to make an inversion with a cell $v=(1,j')\in R_1$ in $\bm\lambda$ is exactly the condition for $u$ to make an inversion with the cell $v^-=(1,j'-1)\in R_1^-$ in $\kappa(\bm\lambda)$.\\

\item It follows directly from the definition that $M(N-R_i,N-R_j)=M(R_j,R_i)$ and therefore $\Pi(N-\bm\lambda)\cong\Pi(\bm\lambda)$. If we restrict to a finite set of variables $(x_1,\ldots,x_k)$, then by associating a tableau $\bm T=(T^{(1)},\ldots,T^{(n)})\in\text{SSYT}_{\bm\lambda}$ to a tableau $-\bm T=(-T^{(1)},\ldots,-T^{(n)})\in\text{SSYT}_{N-\bm\lambda}$ by setting $-T^{(i)}_{1,j}=k+1-T^{(n+1-i)}_{1,N-j}$, we have \begin{equation}G_{N-\bm\lambda}(x_1,\ldots,x_k;q)=G_{\bm\lambda}(x_k,\ldots,x_1;q)\end{equation}
and because LLT polynomials are symmetric, it follows that $G_{N-\bm\lambda}(\bm x;q)=G_{\bm\lambda}(\bm x;q)$.
\end{enumerate}
\end{proof}

\begin{example} Let $\bm\lambda=(4/0,5/4,8/5,6/1)$ as in Example \ref{ex:wgraph}. The cycle $\kappa(\bm\lambda)$ has a negative content, so for convenience we will first translate right. Then the cycle $\kappa(\bm\lambda^+)=(6/5,9/6,7/2,4/0)$ and the rotation $7-\bm\lambda=(7/2,3/0,4/3,7/4)$ are drawn below.
$$
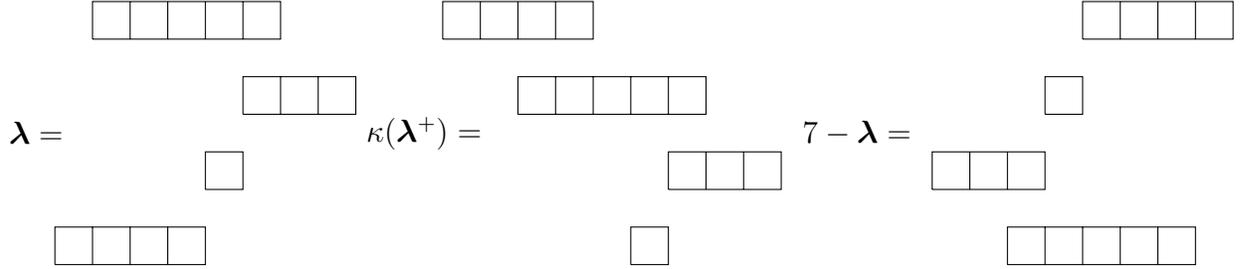
\begin{tikzpicture}
\draw (-0.25,0.25) node (0) {$\bm \lambda=$};
\draw (0,-1.5) -- (2,-1.5) -- (2,-1) -- (0,-1) -- (0,-1.5) (0.5,-1.5) -- (0.5,-1) (1,-1.5) -- (1,-1) (1.5,-1.5) -- (1.5,-1) (2,-0.5) -- (2.5,-0.5) -- (2.5,0) -- (2,0) -- (2,-0.5) (2.5,0.5) -- (4,0.5) -- (4,1) -- (2.5,1) -- (2.5,0.5) (3,0.5) -- (3,1) (3.5,0.5) -- (3.5,1) (0.5,1.5) -- (3,1.5) -- (3,2) -- (0.5,2) -- (0.5,1.5) (1,1.5) -- (1,2) (1.5,1.5) -- (1.5,2) (2,1.5) -- (2,2) (2.5,1.5) -- (2.5,2);
\end{tikzpicture}
\begin{tikzpicture}
\draw (-0.75,0.25) node (0) {$\kappa(\bm \lambda^+)=$};
\draw (-0.5,1.5) -- (1.5,1.5) -- (1.5,2) -- (-0.5,2) -- (-0.5,1.5) (0,1.5) -- (0,2) (0.5,1.5) -- (0.5,2) (1,1.5) -- (1,2) (2,-1.5) -- (2.5,-1.5) -- (2.5,-1) -- (2,-1) -- (2,-1.5) (2.5,-0.5) -- (4,-0.5) -- (4,0) -- (2.5,0) -- (2.5,-0.5) (3,-0.5) -- (3,0) (3.5,-0.5) -- (3.5,0) (0.5,0.5) -- (3,0.5) -- (3,1) -- (0.5,1) -- (0.5,0.5) (1,0.5) -- (1,1) (1.5,0.5) -- (1.5,1) (2,0.5) -- (2,1) (2.5,0.5) -- (2.5,1);
\end{tikzpicture} \ 
\begin{tikzpicture}
\draw (-1,0.25) node (0) {$7-\bm\lambda=$};
\draw (1,-1.5) -- (3.5,-1.5) -- (3.5,-1) -- (1,-1) -- (1,-1.5) (1.5,-1.5) -- (1.5,-1) (2,-1.5) -- (2,-1) (2.5,-1.5) -- (2.5,-1) (3,-1.5) -- (3,-1) (0,-0.5) -- (1.5,-0.5) -- (1.5,0) -- (0,0) -- (0,-0.5) (0.5,-0.5) -- (0.5,0) (1,-0.5) -- (1,0) (1.5,0.5) -- (2,0.5) -- (2,1) -- (1.5,1) -- (1.5,0.5) (2,1.5) -- (4,1.5) -- (4,2) -- (2,2) -- (2,1.5) (2.5,1.5) -- (2.5,2) (3,1.5) -- (3,2) (3.5,1.5) -- (3.5,2);
\end{tikzpicture}$$
\end{example}

\begin{definition}
We say that rows $R$ and $R'$ \emph{commute}, denoted $R\leftrightarrow R'$, if we have $M(R,R')=M(R',R)$, and otherwise we write $R\nleftrightarrow R'$.
\end{definition}

\begin{lemma} \cite[Lemma 3.15]{caterpillarllt} Let $\bm\lambda=(R_1,\ldots,R_n)$ be a horizontal-strip. If $R_i\leftrightarrow R_{i+1}$, then $(R_1,\ldots,R_{i+1},R_i,\ldots,R_n)\in\mathcal S(\bm\lambda)$. \end{lemma}

We now examine the integers $M_{i,j}(\bm\lambda)$ in more detail, particularly their relationship to commutation. The following Proposition is a straightforward calculation.

\begin{proposition} \label{prop:mrirj} \cite[Proposition 3.8]{caterpillarllt}
Let $\bm\lambda=(R_1,\ldots,R_n)$ be a horizontal-strip and let $1\leq i,j\leq n$. Without loss of generality, assume that $l(R_i)\leq l(R_j)$.
\begin{enumerate}
\item If $r(R_i)<l(R_j)-1$, then $M(R_i,R_j)=M(R_j,R_i)=0$, so $R_i\leftrightarrow R_j$. \\

\item If $l(R_i)=l(R_j)$ or $r(R_j)\leq r(R_i)$, then $M(R_i,R_j)=M(R_j,R_i)=\min\{|R_i|,|R_j|\}$, so $R_i\leftrightarrow R_j$.\\

\item Otherwise, we have $l(R_i)<l(R_j)\leq r(R_i)+1\leq r(R_j)$, and \begin{equation}\label{eq:mrirj}M(R_i,R_j)=r(R_i)-l(R_j)+1\text{ and }M(R_j,R_i)=r(R_i)-l(R_j)+2,\text{ so }R_i\nleftrightarrow R_j.\end{equation}
In particular, we have 
\begin{equation}\label{eq:mij} M_{i,j}=r(R_i)-l(R_j)+1+\chi(i>j),\end{equation} where $\chi(i>j)=1$ if $i>j$ and $0$ otherwise.
\end{enumerate}
\end{proposition}

\begin{example} The three cases of Proposition \ref{prop:mrirj} are illustrated below. The pairs on the left and the middle commute and the pair on the right does not. As a visual description, we have that two rows commute if and only if they are either disjoint and separated by at least one cell, or if one is contained in the other. 
$$\tableau{&&&&&& \ & \ & \ \\ \\ \ & \ & \ & \ & \ }\hspace{60pt}\tableau{& \ & \ & \ \\ \\ \ & \ & \ & \ & \ }\hspace{60pt}\tableau{&&& \ & \ & \ \\ \\ \ & \ & \ & \ & \ }$$
\end{example}

\begin{remark}
Informally, once we fix the edge weight $M_{i,j}$ between two rows $R_i$ and $R_j$, if $R_i\leftrightarrow R_j$, then there is some flexibility in their positions because they need only be disjoint (and separated by at least one cell) or one is contained in the other, while if $R_i\nleftrightarrow R_j$, then by \eqref{eq:mij} their relative positions are specifically constrained.
\end{remark}

\begin{corollary} \label{cor:mrirj}
Let $\bm\lambda=(R_1,\ldots,R_n)$ be a horizontal-strip. \begin{enumerate}
\item If $R_i\leftrightarrow R_j$, then $M_{i,j}$ is either $0$ or $\min\{|R_i|,|R_j|\}$. Equivalently, if $0<M_{i,j}<\min\{|R_i|,|R_j|\}$, then $R_i\nleftrightarrow R_j$. \\

\item If $R_i\nleftrightarrow R_j$, then we either have $l(R_i)<l(R_j)$ and $r(R_i)<r(R_j)$, or we have $l(R_i)>l(R_j)$ and $r(R_i)>r(R_j)$. \\


\item Suppose that $i<j$ and $R_i\nleftrightarrow R_j$. If $M_{i,j}=0$, then $l(R_j)=r(R_i)+1$, so in particular $l(R_i)<l(R_j)$. If $M_{i,j}=|R_i|$, then $r(R_j)=r(R_i)-1$, while if $M_{i,j}=|R_j|$, then $l(R_j)=l(R_i)-1$, so in particular we have $l(R_j)<l(R_i)$ in both cases.
\end{enumerate}
\end{corollary}

\begin{proof} \hspace{2pt}
\begin{enumerate}
\item Assuming without loss of generality that $l(R_i)\leq l(R_j)$, then all of the possibilities are enumerated in Proposition \ref{prop:mrirj} and we have $R_i\leftrightarrow R_j$ only when $M_{i,j}=0$ or $\min\{|R_i|,|R_j|\}$.\\

\item Assuming without loss of generality that $l(R_i)\leq l(R_j)$, then all of the possibilities are enumerated in Proposition \ref{prop:mrirj} and we have $R_i\nleftrightarrow R_j$ only when $l(R_i)<l(R_j)$ and $l(R_j)\leq r(R_i)+1\leq r(R_j)$, so in particular $r(R_i)<r(R_j)$.\\


\item By \eqref{eq:mrirj}, if $M_{i,j}=M(R_i,R_j)=0$, then if $l(R_j)<l(R_i)$ we would have $M(R_j,R_i)=M(R_i,R_j)-1=-1$, contradicting \eqref{eq:mijbound}, so we must have $l(R_i)<l(R_j)$ and $M_{i,j}=0=r(R_i)-l(R_j)+1$. Similarly, by \eqref{eq:mrirj}, if $M_{i,j}=M(R_i,R_j)=\min\{|R_i|,|R_j|\}$, then if $l(R_i)<l(R_j)$ we would have $M(R_j,R_i)=\min\{|R_i|,|R_j|\}+1$, contradicting \eqref{eq:mijbound}, so we must have $l(R_j)<l(R_i)$ and $M_{i,j}=r(R_j)-l(R_i)+2$. The result now follows by noting that $|R_i|=r(R_i)-l(R_i)+1$ and $|R_j|=r(R_j)-l(R_j)+1$. 
\end{enumerate}
\end{proof}

\begin{remark}
Note that the converse to (1) does not hold. It is possible to have $M(R_i,R_j)=0$ and $R_i\nleftrightarrow R_j$ as below left. However, (3) tells us that if this occurs, the higher row must be to the right. Similarly, it is possible to have $M(R_i,R_j)=\min\{|R_i|,|R_j|\}$ and $R_i\nleftrightarrow R_j$ as below middle or below right. However, (3) tells us that if this occurs, the higher row must be to the left.
$$\tableau{&&&&&& \ & \ & \ & \ \\ \\ \ & \ & \ & \ & \ & \  } \hspace{50pt} \tableau{\ & \ & \ & \ \\ \\ & \ & \ & \ & \ & \ & \ } \hspace{50pt} \tableau{\ & \ & \ & \ & \ & \ \\ \\ &&& \ & \ & \ & \ }$$
\end{remark}

We now show how we can use cycling and commuting to prove Theorem \ref{thm:main} in a special case. Recall that if $\bm\lambda=(R_1,\ldots,R_n)$, then we have \begin{equation} \label{eq:Mn} M(\bm\lambda)=\sum_{1\leq i<j\leq n}M(R_i,R_j)\leq\sum_{1\leq i<j\leq n}\min\{|R_i|,|R_j|\}=n(\bm\lambda).\end{equation}

\begin{lemma} \label{lem:hlcase} Theorem \ref{thm:main} holds if $M(\bm\lambda)=n(\bm\lambda)$. 
\end{lemma}

\begin{proof} Let $\bm\lambda=(R_1,\ldots,R_n)$. By \eqref{eq:Mn}, we have that if $M(\bm\lambda)=n(\bm\lambda)$, then $M(R_i,R_j)=\min\{|R_i|,|R_j|\}$ for all $1\leq i<j\leq n$. Recall that we denote by $\lambda$ the partition determined by the row lengths of $\bm\lambda$. We now show that the horizontal-strip $\bm H(\lambda)=(\lambda_1/0,\ldots,\lambda_n/0)$ is similar to $\bm\lambda$, meaning that the LLT polynomial $G_{\bm\lambda}(\bm x;q)$ only depends on $\lambda$ and therefore only on the weighted graph $\Pi(\bm\lambda)$. By translating, we may assume without loss of generality that $\min\{l(R_i): \ 1\leq i\leq n\}=0$, and suppose that $l(R_a)=0$. Because $M_{i,j}=\min\{|R_i|,|R_j|\}$ for every $i,j$, we have by Proposition \ref{prop:mrirj}, Part 2 an upper bound $l(R_j)\leq r(R_a)+1=|R_a|$, so let us further assume that $\bm\lambda$ has $\sum_{i=1}^n l(R_i)$ minimal among all horizontal-strips similar to $\bm\lambda$. \\

We now claim that $l(R_i)=0$ for every $1\leq i\leq n$. If not, let $j$ be such that $l(R_j)\geq 1$ is maximal. By Corollary \ref{cor:mrirj}, Part 3, if $i<j$ and $R_i\nleftrightarrow R_j$, then because $M_{i,j}=\min\{|R_i|,|R_j|\}$, we must have $l(R_j)<l(R_i)$, contradicting maximality of $l(R_j)$, so we must have $R_i\leftrightarrow R_j$ for every $i<j$. By Proposition \ref{prop:similar}, we can now commute and cycle to find that $(R_1,\ldots,R_n,R_j^-)\in\mathcal S(\bm\lambda)$, contradicting minimality of $\sum_{i=1}^n l(R_i)$. Therefore, we indeed have $l(R_i)=0$ for every $1\leq i\leq n$, so by Proposition \ref{prop:mrirj}, Part 2, we have $R_i\leftrightarrow R_j$ for every $1\leq i,j\leq n$ and by commuting once again we have $\bm H(\lambda)\in\mathcal S(\bm\lambda)$. This completes the proof. 
\end{proof}

\begin{example}
Figure \ref{fig:hlcase} illustrates the idea of the proof of Lemma \ref{lem:hlcase}. The row $R_3$, which has $l(R_3)$ maximal, commutes with all rows below, so by commuting and cycling, we can move it to the left. Continuing in this way, the horizontal-strip $\bm\lambda$ is shown to be similar to $\bm H(4432)$ on the right. 

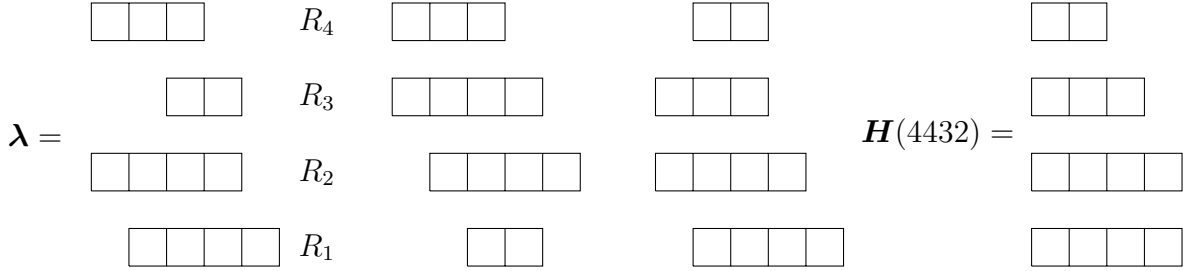
\begin{figure}\caption{An example of using commuting and cycling to show that $\bm H(\lambda)\in\mathcal S(\bm\lambda)$}\label{fig:hlcase}
$$
\begin{tikzpicture}
\draw (-0.75,0.25) node (0) {$\bm \lambda=$};
\draw (3,-1.25) node (1) {$R_1$} (3,-0.25) node (2) {$R_2$} (3,0.75) node (3) {$R_3$} (3,1.75) node (4) {$R_4$};
\draw (0.5,-1.5) -- (2.5,-1.5) -- (2.5,-1) -- (0.5,-1) -- (0.5,-1.5) (1,-1.5) -- (1,-1) (1.5,-1.5) -- (1.5,-1) (2,-1.5) -- (2,-1) (0,-0.5) -- (2,-0.5) -- (2,0) -- (0,0) -- (0,-0.5) (0.5,-0.5) -- (0.5,0) (1,-0.5) -- (1,0) (1.5,-0.5) -- (1.5,0) (1,0.5) -- (2,0.5) -- (2,1) -- (1,1) -- (1,0.5) (1.5,0.5) -- (1.5,1) (0,1.5) -- (1.5,1.5) -- (1.5,2) -- (0,2) -- (0,1.5) (0.5,1.5) -- (0.5,2) (1,1.5) -- (1,2);
\draw (5,-1.5) -- (6,-1.5) -- (6,-1) -- (5,-1) -- (5,-1.5) (5.5,-1.5) -- (5.5,-1) (4.5,-0.5) -- (6.5,-0.5) -- (6.5,0) -- (4.5,0) -- (4.5,-0.5) (5,-0.5) -- (5,0) (5.5,-0.5) -- (5.5,0) (6,-0.5) -- (6,0) (4,0.5) -- (6,0.5) -- (6,1) -- (4,1) -- (4,0.5) (4.5,0.5) -- (4.5,1) (5,0.5) -- (5,1) (5.5,0.5) -- (5.5,1) (4,1.5) -- (5.5,1.5) -- (5.5,2) -- (4,2) -- (4,1.5) (4.5,1.5) -- (4.5,2) (5,1.5) -- (5,2);
\draw (8,-1.5) -- (10,-1.5) -- (10,-1) -- (8,-1) -- (8,-1.5) (8.5,-1.5) -- (8.5,-1) (9,-1.5) -- (9,-1) (9.5,-1.5) -- (9.5,-1) (7.5,-0.5) -- (9.5,-0.5) -- (9.5,0) -- (7.5,0) -- (7.5,-0.5) (8,-0.5) -- (8,0) (8.5,-0.5) -- (8.5,0) (9,-0.5) -- (9,0) (7.5,0.5) -- (9,0.5) -- (9,1) -- (7.5,1) -- (7.5,0.5) (8,0.5) -- (8,1) (8.5,0.5) -- (8.5,1) (8,1.5) -- (9,1.5) -- (9,2) -- (8,2) -- (8,1.5) (8.5,1.5) -- (8.5,2);
\draw (11.25,0.25) node (0) {$\bm H(4432)=$};
\draw (12.5,-1.5) -- (14.5,-1.5) -- (14.5,-1) -- (12.5,-1) -- (12.5,-1.5) (13,-1.5) -- (13,-1) (13.5,-1.5) -- (13.5,-1) (14,-1.5) -- (14,-1) (12.5,-0.5) -- (14.5,-0.5) -- (14.5,0) -- (12.5,0) -- (12.5,-0.5) (13,-0.5) -- (13,0) (13.5,-0.5) -- (13.5,0) (14,-0.5) -- (14,0) (12.5,0.5) -- (14,0.5) -- (14,1) -- (12.5,1) -- (12.5,0.5) (13,0.5) -- (13,1) (13.5,0.5) -- (13.5,1) (12.5,1.5) -- (13.5,1.5) -- (13.5,2) -- (12.5,2) -- (12.5,1.5) (13,1.5) -- (13,2);
\end{tikzpicture}$$
\end{figure}
\end{example}

\begin{remark} In this case where $M(\bm\lambda)=n(\bm\lambda)$, the LLT polynomial $G_{\bm\lambda}(\bm x;q)$ is the transformed modified Hall--Littlewood polynomial  $\tilde H_\lambda(\bm x;q)$, which has many connections in algebraic combinatorics \cite{qtcat,cdm}. \end{remark}

There is another useful linear relationship between LLT polynomials, which we can think of as a deletion-contraction relation. It will be convenient for us to rearrange this relation as follows, so that we add an edge and contract, rather than delete an edge and contract.

\begin{lemma} \label{lem:inductiverelation} \cite[Lemma 3.17]{caterpillarllt} Let $\bm\lambda=(R_1,\ldots,R_n)$ be a horizontal-strip with $l(R_{i+1})>l(R_i)$ and $R_i\nleftrightarrow R_{i+1}$. Define the horizontal-strips  \begin{align}\label{eq:inductivestrips}\bm\lambda'&=(R_1,\ldots,R_{i+1},R_i,\ldots,R_n)\text{ and }\\\nonumber\bm\lambda''&=(R_1,\ldots,R_i\cup R_{i+1},R_i\cap R_{i+1},\ldots,R_n).\end{align}
Then we have
\begin{equation}\label{eq:inductiverelation} G_{\bm\lambda}(\bm x;q)=\frac 1qG_{\bm\lambda'}(\bm x;q)+\frac{q-1}qG_{\bm\lambda''}(\bm x;q).\end{equation}
\end{lemma}

Note that by Proposition \ref{prop:mrirj}, Part 1, the condition $R_i\nleftrightarrow R_{i+1}$ implies that $r(R_i)\geq\ell(R_{i+1})-1$, so the row $R_i\cup R_{i+1}$ makes sense. \\

We now describe the weighted graphs associated to these horizontal-strips $\bm\lambda'$ and $\bm\lambda''$.

\begin{lemma} \label{lem:newgraphs} \cite[Proposition 3.18]{caterpillarllt} Let $R_1$, $R_2$, and $R$ be rows with $R_1\nleftrightarrow R_2$ and $l(R_2)>l(R_1)$ and let $M=M(R_1,R_2)$, $M_1=M(R_1,R)$, and $M_2=M(R_2,R)$. Then \begin{align} M(R_2,R_1)&=M+1, \\\label{eq:capedge}M(R_1\cap R_2,R)&=\min\{M,M_1,M_2\},\text{ and }\\\label{eq:cupedge}M(R_1\cup R_2,R)&=\min\{|R|,\max\{M_1,M_2,M_1+M_2-M\}\}.\end{align}
\end{lemma}

\begin{remark}
Lemma \ref{lem:newgraphs} describes exactly how the weighted graphs of the horizontal-strips in \eqref{eq:inductivestrips} can be obtained from $\Pi=\Pi(\bm\lambda)$. The weighted graph $\Pi'$ is obtained from $\Pi$ by increasing the weight of the edge $(R_i,R_{i+1})$ by one. The weighted graph $\Pi''$ is obtained from $\Pi$ by replacing $R_i$ and $R_{i+1}$ by new vertices $R_i\cap R_{i+1}$ and $R_i\cup R_{i+1}$ of weights $M$ and $|R_i|+|R_{i+1}|-M$ respectively, joined by an edge of weight $M$, and joining them to each other vertex $R$ by edges of weights given in \eqref{eq:capedge} and \eqref{eq:cupedge}. In particular, we have 
\begin{equation}\label{eq:nMinduction}
n(\bm\lambda')=n(\bm\lambda), \ M(\bm\lambda')=M(\bm\lambda)+1,\text{ and }n(\bm\lambda'')<n(\bm\lambda),
\end{equation}
so the relation \eqref{eq:inductiverelation} will allow us to perform induction on $n(\bm\lambda)$ and on $n(\bm\lambda)-M(\bm\lambda)$. 
\end{remark}

\begin{example} Let $\bm\lambda=(4/0,5/4,6/1,8/5)$ and note that $l(R_4)>l(R_3)$ and $R_3\nleftrightarrow R_4$. Therefore, letting $\bm\lambda'=(4/0,5/4,8/5,6/1)$ and $\bm\lambda''=(4/0,5/4,8/1,6/5)$, we have that \begin{equation}G_{\bm\lambda}(\bm x;q)= \frac 1qG_{\bm\lambda'}(\bm x;q)+\frac{q-1}qG_{\bm\lambda''}(\bm x;q).\end{equation}
The horizontal-strips $\bm\lambda$, $\bm\lambda'$, and $\bm\lambda''$, and their weighted graphs $\Pi(\bm\lambda)$, $\Pi(\bm\lambda')$, and $\Pi(\bm\lambda'')$ are given below. We can think of $\Pi(\bm\lambda)$ and $\Pi(\bm\lambda'')$ as a deletion and contraction of $\Pi(\bm\lambda')$. 

$$
\begin{tikzpicture}
\draw (-0.75,0.25) node (0) {$\bm \lambda=$};
\draw (0,-1.5) -- (2,-1.5) -- (2,-1) -- (0,-1) -- (0,-1.5) (0.5,-1.5) -- (0.5,-1) (1,-1.5) -- (1,-1) (1.5,-1.5) -- (1.5,-1) (2,-0.5) -- (2.5,-0.5) -- (2.5,0) -- (2,0) -- (2,-0.5) (2.5,1.5) -- (4,1.5) -- (4,2) -- (2.5,2) -- (2.5,1.5) (3,1.5) -- (3,2) (3.5,1.5) -- (3.5,2) (0.5,0.5) -- (3,0.5) -- (3,1) -- (0.5,1) -- (0.5,0.5) (1,0.5) -- (1,1) (1.5,0.5) -- (1.5,1) (2,0.5) -- (2,1) (2.5,0.5) -- (2.5,1);
\end{tikzpicture} \hspace{10pt}
\begin{tikzpicture}
\draw (-0.25,0.25) node (0) {$\bm \lambda'=$};
\draw (0,-1.5) -- (2,-1.5) -- (2,-1) -- (0,-1) -- (0,-1.5) (0.5,-1.5) -- (0.5,-1) (1,-1.5) -- (1,-1) (1.5,-1.5) -- (1.5,-1) (2,-0.5) -- (2.5,-0.5) -- (2.5,0) -- (2,0) -- (2,-0.5) (2.5,0.5) -- (4,0.5) -- (4,1) -- (2.5,1) -- (2.5,0.5) (3,0.5) -- (3,1) (3.5,0.5) -- (3.5,1) (0.5,1.5) -- (3,1.5) -- (3,2) -- (0.5,2) -- (0.5,1.5) (1,1.5) -- (1,2) (1.5,1.5) -- (1.5,2) (2,1.5) -- (2,2) (2.5,1.5) -- (2.5,2);
\end{tikzpicture} \hspace{10pt}
\begin{tikzpicture}
\draw (-0.75,0.25) node (0) {$\bm\lambda''=$};
\draw (0,-1.5) -- (2,-1.5) -- (2,-1) -- (0,-1) -- (0,-1.5) (0.5,-1.5) -- (0.5,-1) (1,-1.5) -- (1,-1) (1.5,-1.5) -- (1.5,-1) (2,-0.5) -- (2.5,-0.5) -- (2.5,0) -- (2,0) -- (2,-0.5) (0.5,0.5) -- (4,0.5) -- (4,1) -- (0.5,1) -- (0.5,0.5) (1,0.5) -- (1,1) (1.5,0.5) -- (1.5,1) (2,0.5) -- (2,1) (2.5,0.5) -- (2.5,1) (3,0.5) -- (3,1) (3.5,0.5) -- (3.5,1) (2.5,1.5) -- (3,1.5) -- (3,2) -- (2.5,2) -- (2.5,1.5);
\end{tikzpicture}$$

$$
\begin{tikzpicture}
\node[shape=circle,draw=black,minimum size=10mm](1) at (0,0) {$4$};
\node[shape=circle,draw=black,minimum size=10mm](2) at (2,0) {$5$};
\node[shape=circle,draw=black,minimum size=10mm](3) at (2,-1.8) {$1$};
\node[shape=circle,draw=black,minimum size=10mm](4) at (4,0) {$3$};
\path [-](1) edge node [above]{$3$} (2);
\path [-](2) edge node [right]{$1$} (3);
\path [-](2) edge node [above]{$1$} (4);
\end{tikzpicture}\hspace{15pt}
\begin{tikzpicture}
\node[shape=circle,draw=black,minimum size=10mm](1) at (0,0) {$4$};
\node[shape=circle,draw=black,minimum size=10mm](2) at (2,0) {$5$};
\node[shape=circle,draw=black,minimum size=10mm](3) at (2,-1.8) {$1$};
\node[shape=circle,draw=black,minimum size=10mm](4) at (4,0) {$3$};
\path [-](1) edge node [above]{$3$} (2);
\path [-](2) edge node [right]{$1$} (3);
\path [-](2) edge node [above]{$\textcolor{red}2$} (4);
\end{tikzpicture}\hspace{15pt}
\begin{tikzpicture}
\node[shape=circle,draw=black,minimum size=10mm](1) at (0,0) {$4$};
\node[shape=circle,draw=black,minimum size=10mm](2) at (2,0) {$\textcolor{red}7$};
\node[shape=circle,draw=black,minimum size=10mm](3) at (2,-1.8) {$1$};
\node[shape=circle,draw=black,minimum size=10mm](4) at (4,0) {$\textcolor{red}1$};
\path [-](1) edge node [above]{$3$} (2);
\path [-](2) edge node [right]{$1$} (3);
\path [-](2) edge node [above]{$\textcolor{red}1$} (4);
\end{tikzpicture}
$$
\end{example}

Our strategy is to replace $\bm\lambda$ and $\bm\mu$ by similar horizontal-strips to which we can apply \eqref{eq:inductiverelation}. We make the following definition. 

\begin{definition} Let $\bm\lambda$ and $\bm\mu$ be horizontal-strips with $\Pi(\bm\lambda)\cong\Pi(\bm\mu)$. A \emph{good substitute} for $(\bm\lambda,\bm\mu)$ is a pair of horizontal-strips $(\bm\lambda',\bm\mu')$, where $\bm\lambda'=(R_1,\ldots,R_n)\in\mathcal S(\bm\lambda)$ and $\bm\mu'=(S_1,\ldots,S_n)\in\mathcal S(\bm\mu)$ satisfy \begin{equation} l(R_1)<l(R_2), \ R_1\nleftrightarrow R_2, \ l(S_1)<l(S_2), \ S_1\nleftrightarrow S_2,\end{equation}
and $\varphi_1=1$ and $\varphi_2=2$, where $\varphi:\Pi(\bm\lambda')\xrightarrow\sim\Pi(\bm\mu')$. A single horizontal-strip $\bm\lambda$ is \emph{good} if for any horizontal-strip $\bm\mu$ such that $\Pi(\bm\lambda)\cong\Pi(\bm\mu)$, there is a good substitute for $(\bm\lambda,\bm\mu)$.  
\end{definition}

We now state a key Lemma. The proof is quite technical so we postpone it to Section \ref{section:keylemma}. 

\begin{lemma} \label{lem:key} Let $\bm\lambda=(R_1,\ldots,R_n)$ be a horizontal-strip with $n(\bm\lambda)-M(\bm\lambda)\geq 1$. Suppose that $\bm\lambda$ satisfies the condition that \begin{align} \label{eq:IH} \text{Theorem \ref{thm:main} holds for horizontal-strips }\bm\lambda'\text{ and }\bm\mu'\text{ with either }\\\nonumber n(\bm\lambda')<n(\bm\lambda),\text{ or with }n(\bm\lambda')=n(\bm\lambda)\text{ and }M(\bm\lambda')>M(\bm\lambda).\end{align}
Then $\bm\lambda$ is good.
\end{lemma}

We are now ready to prove that our key Lemma implies our main Theorem. 

\begin{proof}[Proof of Theorem \ref{thm:main} assuming Lemma \ref{lem:key}. ] We use induction on $n(\bm\lambda)$. If $n(\bm\lambda)=0$, then $\bm\lambda$ has only one row and the result follows from Proposition \ref{prop:similar}, Part 1, so assume that $n(\bm\lambda)\geq 1$ and that Theorem \ref{thm:main} holds for horizontal-strips $\bm\lambda'$ and $\bm\mu'$ with $n(\bm\lambda')<n(\bm\lambda)$. We also use induction on $n(\bm\lambda)-M(\bm\lambda)$. If $n(\bm\lambda)-M(\bm\lambda)=0$, then the result follows from Lemma \ref{lem:hlcase}, so assume that $n(\bm\lambda)-M(\bm\lambda)\geq 1$ and that Theorem \ref{thm:main} holds for horizontal-strips $\bm\lambda'$ and $\bm\mu'$ with $n(\bm\lambda')=n(\bm\lambda)$ and $n(\bm\lambda')-M(\bm\lambda')<n(\bm\lambda)-M(\bm\lambda)$. This is exactly the condition \eqref{eq:IH}, so assuming Lemma \ref{lem:key}, we have that $\bm\lambda$ is good.\\

Now by replacing $\bm\lambda=(R_1,\ldots,R_n)$ and $\bm\mu=(S_1,\ldots,S_n)$ by a good substitute as necessary, we may assume that $l(R_1)<l(R_2)$, $R_1\nleftrightarrow R_2$, $l(S_1)<l(S_2)$, $S_1\nleftrightarrow S_2$, and $\varphi_1=1$ and $\varphi_2=2$, where $\varphi:\Pi(\bm\lambda)\xrightarrow\sim\Pi(\bm\mu)$. Consider the horizontal-strips \begin{equation} \bm\lambda'=(R_2,R_1,R_3,\ldots,R_n)\text{ and }\bm\lambda''=(R_1\cup R_2,R_1\cap R_2,R_3,\ldots,R_n)\end{equation}
and similarly define $\bm\mu'$ and $\bm\mu''$. Lemma \ref{lem:newgraphs} describes exactly how to construct $\Pi(\bm\lambda')$ and $\Pi(\bm\lambda'')$ from $\Pi(\bm\lambda)$, and therefore we have $\Pi(\bm\lambda')\cong\Pi(\bm\mu')$ and $\Pi(\bm\lambda'')\cong\Pi(\bm\mu'')$. By \eqref{eq:nMinduction}, our induction hypothesis implies that $G_{\bm\lambda'}(\bm x;q)=G_{\bm\mu'}(\bm x;q)$ and $G_{\bm\lambda''}(\bm x;q)=G_{\bm\mu''}(\bm x;q)$. Therefore, by \eqref{eq:inductiverelation}, we have \begin{equation} G_{\bm\lambda}(\bm x;q)=\frac 1qG_{\bm\lambda'}(\bm x;q)+\frac{q-1}qG_{\bm\lambda''}(\bm x;q)=\frac 1qG_{\bm\mu'}(\bm x;q)+\frac{q-1}qG_{\bm\mu''}(\bm x;q)=G_{\bm\mu}(\bm x;q).\end{equation}
This completes the proof. 

\end{proof}

\section{A connection to extended chromatic symmetric functions}\label{section:chromatic}

Theorem \ref{thm:main} tells us that if $\Pi$ is an admissible weighted graph, then there is a well-defined horizontal-strip LLT polynomial $G_\Pi(\bm x;q)$ given by setting $G_\Pi(\bm x;q)=G_{\Pi(\bm\lambda)}(\bm x;q)$, where $\bm\lambda$ is any horizontal-strip with $\Pi\cong\Pi(\bm\lambda)$. We now apply Theorem \ref{thm:main} by exploring a connection between $G_\Pi(\bm x;q)$ and the extended chromatic symmetric function defined by Crew and Spirkl \cite{extendedchromsym}.

\begin{definition} \cite[Equation 1]{extendedchromsym} Let $G$ be a vertex-weighted graph. The \emph{extended chromatic symmetric function} of $G$ is \begin{equation}X_G(\bm x)=\sum_{\substack{\kappa:G\to\mathbb N\\\kappa\text{ proper}}}\prod_{v\in G}x_{\kappa(v)}^{w(v)},\end{equation}
where $w(v)$ denotes the weight of the vertex $v$.
\end{definition}

Their motivation was to establish a deletion-contraction relation \cite[Lemma 2]{extendedchromsym}, which exists for the chromatic polynomial of a graph but not for the chromatic symmetric function. Aliniaeifard, Wang, and van Willigenburg used this deletion-contraction relation to prove some equalities of extended chromatic symmetric functions \cite[Theorem 4.12, Theorem 7.3]{extendedchromsymeq}. In this section, we will use our deletion-contraction relation in Lemma \ref{lem:inductiverelation} to extend these results to equalities of horizontal-strip LLT polynomials.\\

It will be convenient to first define the following operations and relations on compositions.

\begin{definition}
Let $\alpha=\alpha_1\cdots\alpha_n$ and $\beta=\beta_1\cdots\beta_m$ be compositions. The \emph{reverse} of $\alpha$ is $\alpha^{\text{rev}}=\alpha_n\cdots\alpha_1$. The \emph{concatenation} and \emph{near-concatenation} of $\alpha$ and $\beta$ are \begin{equation}
\alpha\cdot\beta=\alpha_1\cdots\alpha_n\beta_1\cdots\beta_m\text{ and }\alpha\odot\beta=\alpha_1\cdots\alpha_{n-1}(\alpha_n+\beta_1)\beta_2\cdots\beta_m.\end{equation}
The \emph{composition} of $\alpha$ and $\beta$, defined by Billera, Thomas, and van Willigenburg \cite[Section 3]{ribeq} is \begin{equation}\alpha\circ\beta=\beta^{\odot\alpha_1}\cdots\beta^{\odot\alpha_n},\end{equation}
where $\beta^{\odot k}$ denotes the $k$-fold near-concatenation of $\beta$. We say that $\beta$ is a \emph{coarsening} of $\alpha$ or alternatively that $\alpha$ is a \emph{refinement} of $\beta$, denoted $\alpha\prec\beta$, if $\beta$ can be obtained from $\alpha$ by summing adjacent parts. We also define the multiset $\mathcal M(\alpha)=\{\lambda(\beta): \ \alpha\prec\beta\}$, where $\lambda(\beta)$ denotes the partition determined by rearranging the parts of $\beta$ in weakly decreasing order.
\end{definition}

There is a bijection between the set $C^N$ of compositions $\alpha$ with sum $N$ and subsets of $\{1,\ldots,N-1\}$ given by taking the partial sums of $\alpha$, other than $N$ itself. Under this bijection, a coarsening corresponds to a subset so the partially ordered set $(C^N,\prec)$ is (anti-)isomorphic to the boolean lattice $B_{N-1}$. In particular, by M\"obius inversion, if $A$ is an abelian group and $f,g:C^N\to A$, then  \begin{equation}\label{eq:mobius} f(\alpha)=\sum_{\alpha\prec\beta}g(\beta)\text{ if and only if  }g(\alpha)=\sum_{\alpha\prec\beta}(-1)^{\ell(\alpha)-\ell(\beta)}f(\beta).\end{equation}

\begin{example}\label{ex:comps}
Consider the compositions $\alpha=21231$, $\beta=23121$, $\delta=12$, and $\gamma=21$. We have that $\delta^{\text{rev}}=\gamma$,
\begin{align}\delta\circ\gamma&=\gamma^{\odot 1}\cdot\gamma^{\odot 2}=\gamma\cdot(\gamma\odot\gamma)=21\cdot(21\odot 21)=21\cdot 231=21231=\alpha,\text{ and }\\\gamma\circ\gamma&=\gamma^{\odot 2}\cdot\gamma^{\odot 1}=(\gamma\odot\gamma)\cdot\gamma=(21\odot 21)\cdot 21=231\cdot 21=23121=\beta.\end{align}
Some coarsenings of $\alpha$ include $3231$, obtained by summing the first two parts, and $54$, obtained by summing the first three parts and the last two parts. We have that 
\begin{equation}\mathcal M(\alpha)=\{32211,5211,4221,3321,3321,621,621,531,531,432,432,81,72,63,54,9\}=\mathcal M(\beta).\end{equation}
\end{example}

Billera, Thomas, and van Willigenburg found the following characterization of when $\mathcal M(\alpha)=\mathcal M(\beta)$ in terms of their operation $\circ$. 

\begin{theorem}\label{thm:coareq}\cite[Theorem 4.1]{ribeq} Let $\alpha$ and $\beta$ be compositions. We have $\mathcal M(\alpha)=\mathcal M(\beta)$ if and only if there are factorizations
\begin{equation}\alpha=\delta^{(1)}\circ\cdots\circ\delta^{(k)}\text{ and }\beta=\gamma^{(1)}\circ\cdots\circ\gamma^{(k)}\end{equation}
so that every $\gamma^{(i)}$ is either $\delta^{(i)}$ or $\delta^{(i)\text{rev}}$. \end{theorem}

\begin{example} We saw that the compositions $\alpha$ and $\beta$ from Example \ref{ex:comps} can be factorized as $\alpha=\delta\circ\gamma$ and $\beta=\delta^{\text{rev}}\circ\gamma$, so it follows from Theorem \ref{thm:coareq} that $\mathcal M(\alpha)=\mathcal M(\beta)$. 
\end{example}

Their motivation was to classify equalities of ribbon Schur functions.

\begin{definition}
Let $\alpha$ be a composition. The \emph{ribbon Schur function} $r_\alpha(\bm x)$ is the skew Schur function indexed by the skew diagram whose $i$-th row has $\alpha_i$ cells and where adjacent rows overlap in exactly one column. When $\alpha$ has a single part $k$, then $r_\alpha(\bm x)$ is the \emph{complete homogeneous symmetric function} $h_k(\bm x)$. If $\lambda=\lambda_1\cdots\lambda_\ell$ is a partition, we also set $h_\lambda(\bm x)=h_{\lambda_1}(\bm x)\cdots h_{\lambda_\ell}(\bm x)$.
\end{definition}

\begin{theorem} \label{thm:ribeq}\cite[Equation 2.2, Proposition 2.1, Theorem 2.6]{ribeq}
The ribbon Schur functions satisfy the relation \begin{equation} \label{eq:ribrelation} r_\alpha(\bm x) r_\beta(\bm x)=r_{\alpha\cdot\beta}(\bm x)+r_{\alpha\odot\beta}(\bm x).\end{equation}
By iterating \eqref{eq:ribrelation}, we have
\begin{equation}
h_{\lambda(\alpha)}(\bm x)=r_{\alpha_1}(\bm x)\cdots r_{\alpha_\ell}(\bm x)=\sum_{\alpha\prec\beta}r_\beta(\bm x)
\end{equation}
and therefore, by \eqref{eq:mobius}, we have \begin{equation}r_\alpha(\bm x)=\sum_{\alpha\prec\beta}(-1)^{\ell(\alpha)-\ell(\beta)}h_{\lambda(\beta)}(\bm x)=\sum_{\lambda\in\mathcal M(\alpha)}(-1)^{\ell(\alpha)-\ell(\lambda)}h_\lambda(\bm x).\end{equation}
In particular, because the $h_\lambda(\bm x)$ are linearly independent, we have $r_\alpha(\bm x)=r_\beta(\bm x)$ if and only if $\mathcal M(\alpha)=\mathcal M(\beta)$.
\end{theorem}

Aliniaeifard, Wang, and van Willigenburg found a similar result for the extended chromatic symmetric functions of vertex-weighted paths. These graphs can be indexed by compositions. 

\begin{definition}
Let $\alpha=\alpha_1\cdots\alpha_n$ be a composition. Define $P_\alpha$ to be the weighted graph with vertices $\{v_1,\ldots,v_n\}$ where $v_i$ has weight $\alpha_i$, and with edges $(v_i,v_{i+1})$ of weight one for $1\leq i\leq n-1$. When $\alpha$ has a single part $k$, then the extended chromatic symmetric function $X_{P_k}(\bm x)$ is the \emph{power sum symmetric function} $p_k(\bm x)$. If $\lambda=\lambda_1\cdots\lambda_\ell$ is a partition, we also set $p_\lambda(\bm x)=p_{\lambda_1}(\bm x)\cdots p_{\lambda_\ell}(\bm x)$. 
\end{definition}

\begin{proposition}\label{prop:horizontalstrippath} Let $\alpha=\alpha_1\cdots\alpha_n$ be a composition of $N$ and consider the horizontal-strip $\bm\lambda_\alpha=(R_1,\ldots,R_n)$, where $R_i=(\sum_{j=1}^{n-i+1}\alpha_j)/(\sum_{j=1}^{n-i}\alpha_j).$
Then $\Pi(\bm\lambda_\alpha)\cong P_\alpha$.
\end{proposition}

\begin{proof}
We have that $|R_i|=\alpha_{n-i+1}$ is the weight of $v_{n-i+1}$, $M(R_i,R_{i+1})=r(R_{i+1})-l(R_i)+2=1$, and $M(R_i,R_j)=0$ for $j\geq i+2$. 
\end{proof}

\begin{example} Let $\alpha=21231$ and $\beta=23121$ as in Example \ref{ex:comps}. The horizontal-strips $\bm\lambda_\alpha=(9/8,8/5,5/3,3/2,2/0)$ and $\bm\lambda_\beta=(9/8,8/6,6/5,5/2,2/0)$ and the paths $P_\alpha$ and $P_\beta$ are given below. Because the (nonzero) edge weights are all one, they are not written.
$$
\begin{tikzpicture}
\draw (-0.75,-0.25) node (0) {$\bm \lambda_\alpha=$};
\draw (0,1.5) -- (1,1.5) -- (1,2) -- (0,2) -- (0,1.5) (0.5,1.5) -- (0.5,2) (1,0.5) -- (1.5,0.5) -- (1.5,1) -- (1,1) -- (1,0.5) (1.5,-0.5) -- (2.5,-0.5) -- (2.5,0) -- (1.5,0) -- (1.5,-0.5) (2,-0.5) -- (2,0) (2.5,-1.5) -- (4,-1.5) -- (4,-1) -- (2.5,-1) -- (2.5,-1.5) (3,-1.5) -- (3,-1) (3.5,-1.5) -- (3.5,-1) (4,-2.5) -- (4.5,-2.5) -- (4.5,-2) -- (4,-2) -- (4,-2.5);
\end{tikzpicture} \hspace{50pt}
\begin{tikzpicture}
\draw (-0.75,-0.25) node (0) {$\bm \lambda_\beta=$};
\draw (0,1.5) -- (1,1.5) -- (1,2) -- (0,2) -- (0,1.5) (0.5,1.5) -- (0.5,2) (1,0.5) -- (2.5,0.5) -- (2.5,1) -- (1,1) -- (1,0.5) (1.5,0.5) -- (1.5,1) (2,0.5) -- (2,1) (2.5,-0.5) -- (3,-0.5) -- (3,0) -- (2.5,0) -- (2.5,-0.5) (3,-1.5) -- (4,-1.5) -- (4,-1) -- (3,-1) -- (3,-1.5) (3.5,-1.5) -- (3.5,-1) (4,-2.5) -- (4.5,-2.5) -- (4.5,-2) -- (4,-2) -- (4,-2.5);
\end{tikzpicture}
$$

$$
\begin{tikzpicture}
\node[shape=circle,draw=black,minimum size=10mm](1) at (0,0) {$2$};
\node[shape=circle,draw=black,minimum size=10mm](2) at (1.5,0) {$1$};
\node[shape=circle,draw=black,minimum size=10mm](3) at (3,0) {$2$};
\node[shape=circle,draw=black,minimum size=10mm](4) at (4.5,0) {$3$};
\node[shape=circle,draw=black,minimum size=10mm](5) at (6,0) {$1$};
\path [-](1) edge node [above]{} (2);
\path [-](2) edge node [above]{} (3);
\path [-](3) edge node [above]{} (4);
\path [-](4) edge node [above]{} (5);
\end{tikzpicture}\hspace{50pt}
\begin{tikzpicture}
\node[shape=circle,draw=black,minimum size=10mm](1) at (0,0) {$2$};
\node[shape=circle,draw=black,minimum size=10mm](2) at (1.5,0) {$3$};
\node[shape=circle,draw=black,minimum size=10mm](3) at (3,0) {$1$};
\node[shape=circle,draw=black,minimum size=10mm](4) at (4.5,0) {$2$};
\node[shape=circle,draw=black,minimum size=10mm](5) at (6,0) {$1$};
\path [-](1) edge node [above]{} (2);
\path [-](2) edge node [above]{} (3);
\path [-](3) edge node [above]{} (4);
\path [-](4) edge node [above]{} (5);
\end{tikzpicture}
$$
\end{example}

\begin{theorem} \label{thm:chrompath}\cite[Equation 2, Equation 3, Theorem 4.12]{extendedchromsymeq} The extended chromatic symmetric functions satisfy the relation \begin{equation}\label{eq:mpathrelation}X_{P_\alpha}(\bm x)X_{P_\beta}(\bm x)=X_{P_{\alpha\cdot\beta}}(\bm x)+X_{P_{\alpha\odot\beta}}(\bm x).\end{equation}
By iterating \eqref{eq:mpathrelation}, we have
\begin{equation}
p_{\lambda(\alpha)}(\bm x)=X_{P_{\alpha_1}}(\bm x)\cdots X_{P_{\alpha_\ell}}(\bm x)=\sum_{\alpha\prec\beta}X_{P_\beta}(\bm x)
\end{equation}
and therefore, by \eqref{eq:mobius}, we have
\begin{equation}
X_{P_\alpha}(\bm x)=\sum_{\alpha\prec\beta}(-1)^{\ell(\alpha)-\ell(\beta)}p_{\lambda(\beta)}(\bm x)=\sum_{\lambda\in\mathcal M(\alpha)}(-1)^{\ell(\alpha)-\ell(\lambda)}p_\lambda(\bm x).
\end{equation}
In particular, because the $p_\lambda(\bm x)$ are linearly independent, we have $X_{P_\alpha}(\bm x)=X_{P_\beta}(\bm x)$ if and only if $\mathcal M(\alpha)=\mathcal M(\beta)$. 
\end{theorem}

We now show the analogous result for horizontal-strip LLT polynomials of vertex-weighted paths.

\begin{theorem} \label{thm:mpath}
The horizontal-strip LLT polynomials satisfy the relation
\begin{equation}\label{eq:lltmpathrelation} G_{P_\alpha}(\bm x;q)G_{P_\beta}(\bm x;q)=\frac 1qG_{P_{\alpha\cdot\beta}}(\bm x;q)+\frac{q-1}qG_{P_{\alpha\odot\beta}}(\bm x;q).\end{equation}
By iterating \eqref{eq:lltmpathrelation}, we have
\begin{equation}
h_{\lambda(\alpha)}(\bm x)=G_{P_{\alpha_1}}(\bm x;q)\cdots G_{P_{\alpha_\ell}}(\bm x;q)=\sum_{\alpha\prec\beta}\frac{(q-1)^{\ell(\alpha)-\ell(\beta)}}{q^{\ell(\alpha)-1}}G_{P_\beta}(\bm x;q)
\end{equation}
and rearranging, we have 
\begin{equation}
q^{\ell(\alpha)-1}(q-1)^{-\ell(\alpha)}h_{\lambda(\alpha)}(\bm x)=\sum_{\alpha\prec\beta}(q-1)^{-\ell(\beta)}G_{P_\beta}(\bm x;q).
\end{equation}
Therefore, by \eqref{eq:mobius}, we have 
\begin{equation}
(q-1)^{-\ell(\alpha)}G_{P_\alpha}(\bm x;q)=\sum_{\alpha\prec\beta}q^{\ell(\beta)-1}(q-1)^{-\ell(\beta)}h_{\lambda(\beta)}(\bm x)
\end{equation}
and rearranging again, we have 
\begin{equation}
G_{P_\alpha}(\bm x;q)=\sum_{\alpha\prec\beta}q^{\ell(\beta)-1}(q-1)^{\ell(\alpha)-\ell(\beta)}h_{\lambda(\beta)}(\bm x)=\sum_{\lambda\in\mathcal M(\alpha)}q^{\ell(\lambda)-1}(q-1)^{\ell(\alpha)-\ell(\lambda)}h_\lambda(\bm x).
\end{equation}
In particular, because the $h_\lambda(\bm x)$ are linearly independent, we have $G_{P_\alpha}(\bm x;q)=G_{P_\beta}(\bm x;q)$ if and only if $\mathcal M(\alpha)=\mathcal M(\beta)$. 
\end{theorem}

\begin{proof} We need only prove the first statement \eqref{eq:lltmpathrelation}. Let $\bm\lambda_\alpha=(R_1,\ldots,R_n)$ and $\bm\lambda_\beta=(S_1,\ldots,S_m)$ be the horizontal-strips defined in Proposition \ref{prop:horizontalstrippath}. Suppose that $\alpha$ has sum $N$ and consider the horizontal-strips
\begin{align}\bm\lambda&=(N+S_1,\ldots,N+S_{m-1},R_1,N+S_m,R_2,\ldots,R_n),\\\nonumber \bm\lambda'&=(N+S_1,\ldots,N+S_{m-1},N+S_m,R_1,R_2\ldots,R_n),\text{ and }\\\nonumber\bm\lambda''&=(N+S_1,\ldots,N+S_{m-1},(N+S_{m-1})\cup R_1,(N+S_{m-1})\cap R_1,R_2,\ldots,R_n).
\end{align} Note that $\bm\lambda$ was constructed by translating $\bm\lambda_\beta$, concatenating, and switching rows $(R_t,N+S_{t'})$ with $(t,t')\neq (1,m)$, which commute by Proposition \ref{prop:mrirj}, Part 1 because $\ell(N+S_{t'})>r(R_t)+1$. Therefore, we have that $\Pi(\bm\lambda)$ is the disjoint union of $P_\alpha$ and $P_\beta$ and so $G_{\bm\lambda}(\bm x;q)=G_{P_\alpha}(\bm x;q)G_{P_\beta}(\bm x;q)$. Lemma \ref{lem:newgraphs} describes exactly how the weighted graphs $\Pi(\bm\lambda')$ and $\Pi(\bm\lambda'')$ are constructed and we see that $\Pi(\bm\lambda')\cong P_{\alpha\cdot\beta}$ and $\Pi(\bm\lambda'')\cong P_{\alpha\odot\beta}$. Because $\ell(N+S_m)=r(R_1)+1>l(R_1)$, we have $R_1\nleftrightarrow (N+S_m)$ by Proposition \ref{prop:mrirj}, Part 3 and now \eqref{eq:lltmpathrelation} follows from Lemma \ref{lem:inductiverelation}. 
\end{proof}

We take a moment to summarize these results.

\begin{corollary}
Let $\alpha$ and $\beta$ be compositions. The following are equivalent.
\begin{enumerate}
\item $\mathcal M(\alpha)=\mathcal M(\beta)$
\item $\alpha=\delta^{(1)}\circ\cdots\circ\delta^{(k)}$ and $\beta=\gamma^{(1)}\circ\cdots\circ\gamma^{(k)}$, with every $\gamma^{(i)}=\delta^{(i)}$ or $\delta^{(i)\text{rev}}$
\item $r_\alpha(\bm x)=r_\beta(\bm x)$
\item $X_{P_\alpha}(\bm x)=X_{P_\beta}(\bm x)$
\item $G_{P_\alpha}(\bm x;q)=G_{P_\beta}(\bm x;q)$
\end{enumerate}
\end{corollary}

\begin{proof} We have 
$(1)\iff(2)$ by Theorem \ref{thm:coareq}, $(1)\iff(3)$ by Theorem \ref{thm:ribeq}, $(1)\iff(4)$ by Theorem \ref{thm:chrompath}, and $(1)\iff(5)$ by Theorem \ref{thm:mpath}. 
\end{proof}

Aliniaeifard, Wang, and van Willigenburg generalize Theorem \ref{thm:chrompath} to the following construction. 

\begin{definition}\label{def:compgraph}\cite[Definition 7.1]{extendedchromsymeq}
Let $G$ and $H$ be vertex-weighted graphs with distinguished (not necessarily distinct) vertices $a_G,z_G\in G$ and $a_H,z_H\in H$. The \emph{concatenation} of $G$ and $H$, denoted $G\cdot H$, is the disjoint union of $G$ and $H$ with an edge joining $z_G$ and $a_H$. The \emph{near-concatenation} of $G$ and $H$, denoted $G\odot H$, is the graph $G\cdot H$ with the edge $(z_G,a_H)$ contracted. If $\alpha=\alpha_1\cdots\alpha_n$ is a composition, the \emph{composition} of $\alpha$ and $G$ is the graph
\begin{equation}\alpha\circ G=G^{\odot\alpha_1}\cdots G^{\odot\alpha_n},\end{equation}
where $G^{\odot k}$ denotes the $k$-fold near-concatenation of $G$. 
\end{definition}

\begin{example}
For any composition $\alpha$, we have $\alpha\circ P_1\cong P_\alpha$. For $G=P_{121}$ with $a=v_1$ and $z=v_2$, the graphs $12\circ G$ and $21\circ G$ are given below.
$$
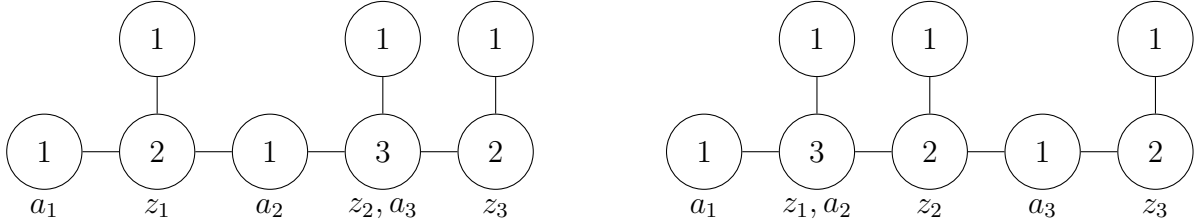
\begin{tikzpicture}
\node[shape=circle,draw=black,minimum size=10mm](1) at (0,0) {$1$};
\node[shape=circle,draw=black,minimum size=10mm](2) at (1.5,0) {$2$};
\node[shape=circle,draw=black,minimum size=10mm](3) at (1.5,1.5) {$1$};
\node[shape=circle,draw=black,minimum size=10mm](4) at (3,0) {$1$};
\node[shape=circle,draw=black,minimum size=10mm](5) at (4.5,0) {$3$};
\node[shape=circle,draw=black,minimum size=10mm](6) at (4.5,1.5) {$1$};
\node[shape=circle,draw=black,minimum size=10mm](7) at (6,0) {$2$};
\node[shape=circle,draw=black,minimum size=10mm](8) at (6,1.5) {$1$};
\draw (0,-0.75) node {$a_1$};
\draw (1.5,-0.75) node {$z_1$};
\draw (3,-0.75) node {$a_2$};
\draw (4.5,-0.75) node {$z_2,a_3$};
\draw (6,-0.75) node {$z_3$};
\path [-](1) edge node [above]{} (2);
\path [-](2) edge node [left]{} (3);
\path [-](2) edge node [above]{} (4);
\path [-](4) edge node [above]{} (5);
\path [-](5) edge node [left]{} (6);
\path [-](5) edge node [above]{} (7);
\path [-](7) edge node [left]{} (8);
\end{tikzpicture}\hspace{50pt}
\begin{tikzpicture}
\node[shape=circle,draw=black,minimum size=10mm](1) at (0,0) {$1$};
\node[shape=circle,draw=black,minimum size=10mm](2) at (1.5,0) {$3$};
\node[shape=circle,draw=black,minimum size=10mm](3) at (1.5,1.5) {$1$};
\node[shape=circle,draw=black,minimum size=10mm](4) at (3,0) {$2$};
\node[shape=circle,draw=black,minimum size=10mm](5) at (3,1.5) {$1$};
\node[shape=circle,draw=black,minimum size=10mm](6) at (4.5,0) {$1$};
\node[shape=circle,draw=black,minimum size=10mm](7) at (6,0) {$2$};
\node[shape=circle,draw=black,minimum size=10mm](8) at (6,1.5) {$1$};
\draw (0,-0.75) node {$a_1$};
\draw (1.5,-0.75) node {$z_1,a_2$};
\draw (3,-0.75) node {$z_2$};
\draw (4.5,-0.75) node {$a_3$};
\draw (6,-0.75) node {$z_3$};
\path [-](1) edge node [above]{} (2);
\path [-](2) edge node [left]{} (3);
\path [-](2) edge node [above]{} (4);
\path [-](4) edge node [left]{} (5);
\path [-](4) edge node [above]{} (6);
\path [-](6) edge node [above]{} (7);
\path [-](7) edge node [left]{} (8);
\end{tikzpicture}
$$
\end{example}

Note that it follows from the definition that \begin{equation}\label{eq:compgraph} (\alpha\cdot\beta)\circ G=(\alpha\circ G)\cdot(\beta\circ G)\text{ and }(\alpha\odot\beta)\circ G=(\alpha\circ G)\odot(\beta\circ G).\end{equation}

\begin{theorem}\cite[Theorem 7.3]{extendedchromsymeq} Let $G$ be a vertex-weighted graph with distinguished vertices $a$ and $z$. If $\mathcal M(\alpha)=\mathcal M(\beta)$, then $X_{P_{\alpha\circ G}}(\bm x)=X_{P_{\beta\circ G}}(\bm x)$. Moreover, if $G$ is simple and connected, then the converse holds. \end{theorem}

We prove a similar result for horizontal-strip LLT polynomials, but we require a condition on the distinguished vertices to ensure that the resulting weighted graphs be admissible. Recall that \emph{content reading order} is the total ordering on cells by increasing content and from bottom to top along constant content lines. 

\begin{definition}
Let $\Pi=\Pi(\bm\lambda)$ be an admissible weighted graph. Let $a$ and $z$ be the vertices corresponding to the rows containing the first and last cells of $\bm\lambda$ in content reading order. Then we define concatenation, near-concatenation, and composition as in Definition \ref{def:compgraph} with respect to these distinguished vertices $a$ and $z$. 
\end{definition}

\begin{theorem}\label{thm:compgraph}
Let $\Pi=\Pi(\bm\lambda)$ be an admissible weighted graph. For any composition $\alpha$, the weighted graph $\alpha\circ\Pi$ is admissible, and if $\mathcal M(\alpha)=\mathcal M(\beta)$, then $G_{\alpha\circ\Pi}(\bm x;q)=G_{\beta\circ\Pi}(\bm x;q)$. 
\end{theorem}

\begin{proof}
Let $\Pi_1=\Pi(\bm\lambda)$ and $\Pi_2=\Pi(\bm\mu)$ be admissible weighted graphs. We first show that $\Pi_1\cdot\Pi_2$ and $\Pi_1\odot\Pi_2$ are admissible, which proves the first statement. By cycling and translating, we may assume without loss of generality that $\bm\lambda=(R_1,\ldots,R_n)$ has a unique cell $u\in R_1$ of maximal content $N-1$ and $\bm\mu=(S_1,\ldots,S_m)$ has a unique cell $v\in S_m$ of minimal content $0$. Then by translating $\bm\mu$ and concatenating, we can consider the horizontal-strip $(R_1,\ldots,R_n,N+S_1,\ldots,N+S_m)$, whose weighted graph is the disjoint union $\Pi_1\sqcup \Pi_2$. Consider the horizontal-strips
\begin{align}\bm\lambda&=(N+S_1,\ldots,N+S_{m-1},R_1,N+S_m,R_2,\ldots,R_n),\\\nonumber \bm\lambda'&=(N+S_1,\ldots,N+S_{m-1},N+S_m,R_1,R_2\ldots,R_n),\text{ and }\\\nonumber\bm\lambda''&=(N+S_1,\ldots,N+S_{m-1},(N+S_{m-1})\cup R_1,(N+S_{m-1})\cap R_1,R_2,\ldots,R_n).
\end{align} Note that $\bm\lambda$ was constructed by switching rows $(R_t,N+S_{t'})$ with $(t,t')\neq (1,m)$, which commute by Proposition \ref{prop:mrirj}, Part 1 because $\ell(N+S_{t'})>r(R_t)+1$. Therefore, we have that $\Pi(\bm\lambda)\cong\Pi_1\sqcup \Pi_2$ and so $G_{\bm\lambda}(\bm x;q)=G_{\Pi_1}(\bm x;q)G_{\Pi_2}(\bm x;q)$. Lemma \ref{lem:newgraphs} describes exactly how the weighted graphs $\Pi(\bm\lambda')$ and $\Pi(\bm\lambda'')$ are constructed and we see that $\Pi(\bm\lambda')\cong\Pi_1\cdot\Pi_2$ and $\Pi(\bm\lambda'')\cong\Pi_1\odot\Pi_2$, so indeed these weighted graphs are admissible. Because $\ell(N+S_m)=r(R_1)+1>l(R_1)$, we have $R_1\nleftrightarrow (N+S_m)$ by Proposition \ref{prop:mrirj}, Part 3 and now by Lemma \ref{lem:inductiverelation} we have 
\begin{equation}
G_{\Pi_1}(\bm x;q)G_{\Pi_2}(\bm x;q)=\frac 1qG_{\Pi_1\cdot\Pi_2}(\bm x;q)+\frac{q-1}qG_{\Pi_1\odot\Pi_2}(\bm x;q).
\end{equation}
Therefore, by \eqref{eq:compgraph}, for compositions $\alpha$ and $\beta$, the horizontal-strip LLT polynomials satisfy the relation 
\begin{equation}\label{eq:compgraphrelation} G_{\alpha\circ\Pi}(\bm x;q)G_{\beta\circ\Pi}(\bm x;q)=\frac 1qG_{(\alpha\cdot\beta)\circ \Pi}(\bm x;q)+\frac{q-1}qG_{(\alpha\odot\beta)\circ\Pi}(\bm x;q).\end{equation}
By iterating \eqref{eq:compgraphrelation}, we have
\begin{equation}G_{\Pi^{\odot\alpha_1}}(\bm x;q)\cdots G_{\Pi^{\odot\alpha_\ell}}(\bm x;q)=\sum_{\alpha\prec\beta}\frac{(q-1)^{\ell(\alpha)-\ell(\beta)}}{q^{\ell(\alpha)-1}}G_{\beta\circ\Pi}(\bm x;q)\end{equation}
and rearranging, we have
\begin{equation}q^{\ell(\alpha)-1}(q-1)^{-\ell(\alpha)}\prod_{i=1}^\ell G_{\Pi^{\odot\alpha_i}}(\bm x;q)=\sum_{\alpha\prec\beta}(q-1)^{-\ell(\beta)}G_{\beta\circ\Pi}(\bm x;q).\end{equation}
Therefore, by \eqref{eq:mobius}, we have
\begin{equation}(q-1)^{-\ell(\alpha)}G_{\alpha\circ\Pi}(\bm x;q)=\sum_{\alpha\prec\beta}q^{\ell(\beta)-1}(q-1)^{-\ell(\beta)}\prod_{i=1}^\ell G_{\Pi^{\odot\beta_i}}(\bm x;q)\end{equation}
and rearranging again, we have
\begin{align}
G_{\alpha\circ \Pi}(\bm x;q)&=\sum_{\alpha\prec\beta}q^{\ell(\beta)-1}(q-1)^{\ell(\alpha)-\ell(\beta)}\prod_{i=1}^\ell G_{\Pi^{\odot\beta_i}}(\bm x;q)\\\nonumber&=\sum_{\lambda\in\mathcal M(\alpha)}q^{\ell(\lambda)-1}(q-1)^{\ell(\alpha)-\ell(\lambda)}\prod_{i=1}^\ell G_{\Pi^{\odot\lambda_i}}(\bm x;q).
\end{align}
In particular, if $\mathcal M(\alpha)=\mathcal M(\beta)$, then $G_{\alpha\circ\Pi}(\bm x;q)=G_{\beta\circ\Pi}(\bm x;q)$. 
\end{proof}

\section{Proof of Lemma \ref{lem:key}} \label{section:keylemma}

In this section we prove Lemma \ref{lem:key}, which completes the proof of Theorem \ref{thm:main}. The general idea will be to use rotating, cycling, and commuting to move rows into the desired position. We will be able to do this unless there is a \emph{noncommuting path}, which is a concept we will introduce in Definition \ref{def:ncp}. In Lemma \ref{lem:mncp1} and Lemma \ref{lem:mncp2}, we will describe the structure of a minimal noncommuting path very precisely, which will allow us to rule out several cases. We begin with the following definition, which expresses a relationship between rows that will be convenient to consider.

\begin{definition} Let $\bm\lambda=(R_1,\ldots,R_n)$ be a horizontal-strip. We write $R_i\prec R_j$ if $M_{i,j}=|R_i|$ and $R_i\nprec R_j$ otherwise. We also write $R_i\precnsim R_j$ to mean that $R_i\prec R_j$ and $R_j\nprec R_i$. \end{definition}

\begin{proposition} \label{prop:prec}
Let $\bm\lambda=(R_1,\ldots,R_n)$ be a horizontal-strip. \begin{enumerate}
\item Suppose that $i<j$. Then we have $R_i\prec R_j$ if and only if either $R_i\subseteq R_j$ or $R_i\subseteq R_j^+$.\\

\item Suppose that $i>j$. Then we have $R_i\prec R_j$ if and only if either $R_i\subseteq R_j$ or $R_i\subseteq R_j^-$.\\

\item We have $R_i\prec R_j$ and $R_i\leftrightarrow R_j$ if and only if $R_i\subseteq R_j$.
\end{enumerate}
\end{proposition}

\begin{proof}\hspace{2pt}
\begin{enumerate}
\item If $R_i\prec R_j$, then by definition we have \begin{equation}M_{i,j}=|R_i|=\begin{cases} |R_i\cap R_j|&\text{ if }l(R_i)\leq l(R_j),\\|R_i\cap R_j^+|&\text{ if }l(R_i)>l(R_j),\end{cases}\end{equation}
and therefore we must have $R_i\subseteq R_j$ or $R_i\subseteq R_j^+$. Conversely, if $R_i\subseteq R_j^+$, then $l(R_i)>l(R_j)$ and $M_{i,j}=|R_i|$ by Proposition \ref{prop:mrirj}, Part 2, while if $R_i\nsubseteq R_j^+$ and $R_i\subseteq R_j$, then $l(R_i)=l(R_j)$ and again $M_{i,j}=|R_i|$ by Proposition \ref{prop:mrirj}, Part 2. \\

\item This follows from the previous part by considering a rotation of $\bm\lambda$. \\

\item If $R_i\subseteq R_j$, then $R_i\prec R_j$ by the previous parts and we have $l(R_j)\leq l(R_i)\leq r(R_i)\leq r(R_j)$, so $R_i\leftrightarrow R_j$ by Proposition \ref{prop:mrirj}, Part 2. Conversely, if $i<j$, $R_i\prec R_j$, and $R_i\nsubseteq R_j$, then by Part 1 we have $R_i\subseteq R_j^+$, so $l(R_j)<l(R_i)\leq r(R_j)+1\leq r(R_i)$ and $R_i\nleftrightarrow R_j$ by Proposition \ref{prop:mrirj}, Part 3. The case where $i>j$ follows by rotating. 
\end{enumerate}
\end{proof}

\begin{example} Let $\bm\lambda=(R_1,R_2,R_3)=(7/4,6/3,5/0)$ as in the example below. We have $M(R_1,R_2)=3=|R_1|=|R_2|$, so $R_1\prec R_2$ and $R_2\prec R_1$. We have $M(R_1,R_3)=2<|R_1|$, so $R_1\nprec R_3$, and we have $M(R_2,R_3)=3=|R_2|<|R_3|$, so $R_2\precnsim R_3$.

$$\begin{tikzpicture}
\draw (-0.75,-0.25) node (0) {$\bm \lambda=$};
\draw (4,-1.25) node (1) {$R_1$} (4,-0.25) node (2) {$R_2$} (4,0.75) node (3) {$R_3$};
\draw (2,-1.5) -- (3.5,-1.5) -- (3.5,-1) -- (2,-1) -- (2,-1.5) (2.5,-1.5) -- (2.5,-1) (3,-1.5) -- (3,-1) (1.5,-0.5) -- (3,-0.5) -- (3,0) -- (1.5,0) -- (1.5,-0.5) (2,-0.5) -- (2,0) (2.5,-0.5) -- (2.5,0) (0,0.5) -- (2.5,0.5) -- (2.5,1) -- (0,1) -- (0,0.5) (0.5,0.5) -- (0.5,1) (1,0.5) -- (1,1) (1.5,0.5) -- (1.5,1) (2,0.5) -- (2,1);
\end{tikzpicture}$$
Informally, we can think of the relation $R_i\prec R_j$ as being very similar to the relation $R_i\subseteq R_j$, except that we may need to shift a row by one cell. Because of this possible shift, the relation $\prec$ is not transitive. In the above example, we have $R_1\prec R_2$ and $R_2\prec R_3$, but $R_1\nprec R_3$.  \end{example}

The following concept will be important to define the potential obstruction to commuting. 

\begin{definition} \label{def:ncp} A sequence of $n\geq 3$ rows $(R_1,\ldots,R_n)$ is a \emph{noncommuting path} from $R_1$ to $R_n$ if $R_i\nleftrightarrow R_{i+1}$ for every $1\leq i\leq n-1$. A noncommuting path is \emph{minimal} if there is no subsequence of rows $(R_1=R_{i_1},R_{i_2},\ldots,R_{i_k}=R_n)$ with $i_1<i_2<\cdots<i_k$ and $3\leq k<n$ that forms a noncommuting path from $R_1$ to $R_n$. \end{definition}

In particular, if $(R_1,\ldots,R_n)$ is a minimal noncommuting path, then we have $R_i\leftrightarrow R_j$ for every $i,j$ with $1<|j-i|<n-1$. Because we require a noncommuting path to have length at least $3$, we cannot conclude that $R_1\leftrightarrow R_n$ in a minimal noncommuting path.\\

Our next Lemma shows that, given two rows $R_i$ and $R_j$ of $\bm\lambda$ with $i<j$, either there is a minimal noncommuting path in $\bm\lambda$ from $R_i$ to $R_j$, or we may assume that $j=i+1$. 

\begin{lemma} \label{lem:aoncp} Let $\bm\lambda=(R_1,\ldots,R_n)$ be a horizontal-strip and let $1\leq i<j\leq n$. Then one of the following holds. \begin{enumerate}
\item There is a minimal noncommuting path $(R_i=R_{i_1},\ldots,R_{i_k}=R_j)$ in $\bm\lambda$.\\

\item There is a horizontal-strip $\bm\mu=(S_1,\ldots,S_n)\in\mathcal S(\bm\lambda)$ and an isomorphism of weighted graphs $\varphi:\Pi(\bm\lambda)\xrightarrow\sim\Pi(\bm\mu)$ such that $l(S_{\varphi_i})=l(R_i)$, $l(S_{\varphi_j})=l(R_j)$, and $\varphi_j=\varphi_i+1$.
\end{enumerate}
\end{lemma}

\begin{proof}
We use induction on $j-i$. If $j-i=1$, then the second possibility holds by simply taking $\bm\lambda=\bm\mu$, so assume that $j-i\geq 2$. If $R_j\leftrightarrow R_{j-1}$, then by commuting we have $(R_1,\ldots,R_j,R_{j-1},\ldots,R_n)\in\mathcal S(\bm\lambda)$ and we are done by our induction hypothesis on $j-i$, so we may assume that $R_j\nleftrightarrow R_{j-1}$. Similarly, if $R_{j-1}\leftrightarrow R_t$ for every $i\leq t\leq j-2$, then by commuting we would have $(R_1,\ldots,R_{j-1},R_i,\ldots,R_j,\ldots,R_n)\in\mathcal S(\bm\lambda)$ and we are again done by induction. So we may assume that $R_{j-1}\nleftrightarrow R_t$ for some $i\leq t\leq j-2$, and continuing in this way there must be a noncommuting path in $\bm\lambda$ from $R_i$ to $R_j$. Finally, if this noncommuting path is not minimal, then it contains a minimal one. 
\end{proof}

Examples of minimal noncommuting paths are given in Figure \ref{fig:mncps}. Informally, our next goal is to show that all minimal noncommuting paths look like these examples. Specifically, our goal is to prove Lemma \ref{lem:mncp1}, which describes what a minimal noncommuting path may look like, and Lemma \ref{lem:mncp2}, which describes the extent to which our weighted graph determines the structure of a minimal noncommuting path. This will allow us to prove Corollary \ref{cor:adjacentnoncommutingdone}, which proves Lemma \ref{lem:key} in many cases. We will first prove some elementary Propositions.  

\begin{figure}
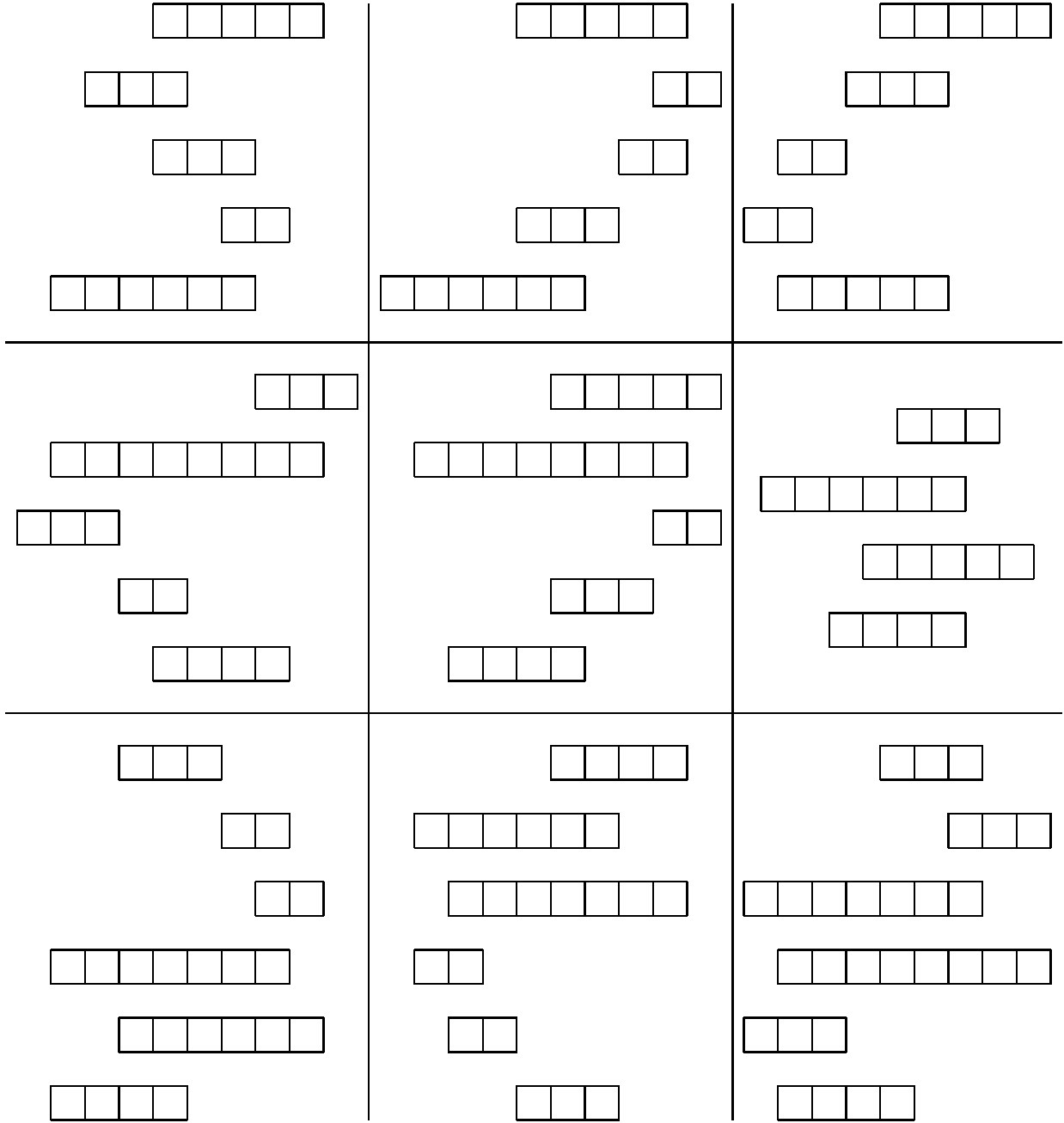

\caption{Some examples of minimal noncommuting paths}\label{fig:mncps}
$$
\begin{tabular}{c|c|c}
\tableau{&&& \ & \ & \ & \ & \ \\ \\ & \ & \ & \ \\ \\ &&& \ & \ & \ \\ \\ &&&&& \ & \ \\ \\ \ & \ & \ & \ & \ & \ } &
\tableau{&&&& \ & \ & \ & \ & \ \\ \\ &&&&&&&& \ & \ \\ \\ &&&&&&& \ & \ \\ \\ &&&& \ & \ & \ \\ \\ \ & \ & \ & \ & \ & \ } &
\tableau{&&&& \ & \ & \ & \ & \ \\ \\ &&& \ & \ & \ \\ \\ & \ & \ \\ \\ \ & \ \\ \\ & \ & \ & \ & \ & \ }
\\ && \\ \hline && \\
\tableau{&&&&&&& \ & \ & \ \\ \\ & \ & \ & \ & \ & \ & \ & \ & \ \\ \\ \ & \ & \ \\ \\ &&& \ & \ \\ \\ &&&& \ & \ & \ & \ }&
\tableau{&&&&& \ & \ & \ & \ & \ \\ \\ & \ & \ & \ & \ & \ & \ & \ & \  \\ \\ &&&&&&&& \ & \ \\ \\ &&&&& \ & \ & \ \\ \\ && \ & \ & \ & \ }&
\tableau{ \\ &&&& \ & \ & \ \\ \\ \ & \ & \ & \ & \ & \ \\ \\ &&& \ & \ & \ & \ & \ \\ \\ && \ & \ & \ & \ }
\\ && \\ \hline && \\ 
\tableau{&& \ & \ & \ \\ \\ &&&&& \ & \ \\ \\ &&&&&& \ & \ \\ \\ \ & \ & \ & \ & \ & \ & \ \\ \\ && \ & \ & \ & \ & \ & \ \\ \\  \ & \ & \ & \ }&
\tableau{&&&& \ & \ & \ & \ \\ \\ \ & \ & \ & \ & \ & \ \\ \\ & \ & \ & \ & \ & \ & \ & \ \\ \\ \ & \ \\ \\ & \ & \ \\ \\ &&& \ & \ & \ }&
\tableau{&&&& \ & \ & \ \\ \\ &&&&&& \ & \ & \ \\ \\ \ & \ & \ & \ & \ & \ & \ \\ \\ & \ & \ & \ & \ & \ & \ & \ & \ \\ \\ \ & \ & \ \\ \\ & \ & \ & \ & \  }
\end{tabular}$$
\end{figure}

\begin{proposition} \label{prop:prec2} Let $\bm\lambda=(R_1,\ldots,R_n)$ be a horizontal-strip. \begin{enumerate}
\item Suppose that $R_i\nleftrightarrow R_j$, $R_j\nleftrightarrow R_k$, $R_i\leftrightarrow R_k$, and that the integers $l(R_j)-l(R_i)$ and $l(R_k)-l(R_j)$ have the same sign. Then $M_{i,k}=0$. \\

\item Suppose that $R_i\nleftrightarrow R_j$, $R_j\nleftrightarrow R_k$, $R_i\leftrightarrow R_k$, and that the integers $l(R_j)-l(R_i)$ and $l(R_k)-l(R_j)$ have opposite signs. Then $R_i\prec R_k$ or $R_k\prec R_i$. \\

\item Suppose that $R_i\leftrightarrow R_j$, $R_i\leftrightarrow R_k$, and $R_j\nleftrightarrow R_k$. Then $R_j\prec R_i$ if and only if $R_k\prec R_i$. 
\end{enumerate}
\end{proposition}

Note that in (1) and (2), because $R_i\nleftrightarrow R_j$ and $R_j\nleftrightarrow R_k$, the integers $l(R_j)-l(R_i)$ and $l(R_k)-l(R_j)$ are nonzero by Proposition \ref{prop:mrirj}, Part 2.

\begin{proof} [Proof of Proposition \ref{prop:prec2}. ]\hspace{2pt}
\begin{enumerate}
\item Without loss of generality, we may assume that $l(R_i)<l(R_j)<l(R_k)$, and then because $R_i\nleftrightarrow R_j$ and $R_j\nleftrightarrow R_k$, by Corollary \ref{cor:mrirj}, Part 2, we must have $r(R_i)<r(R_j)<r(R_k)$ as well. But now we cannot have $R_i\subseteq R_k$ or $R_k\subseteq R_i$, so by Proposition \ref{prop:prec}, Part 3, we cannot have $R_i\prec R_k$ or $R_k\prec R_i$. Therefore, because $R_i\leftrightarrow R_k$, by Corollary \ref{cor:mrirj}, Part 1, we must have $M_{i,k}=0$.\\

\item By rotating, we may assume without loss of generality that $l(R_i)\leq l(R_j)-1$ and $l(R_k)\leq l(R_j)-1$. Because $R_i\nleftrightarrow R_j$ and $R_k\nleftrightarrow R_j$, by Proposition \ref{prop:mrirj}, Part 1, we have $r(R_i)\geq l(R_j)-1$ and $r(R_k)\geq l(R_j)-1$. But now $M_{i,k}>0$ and because $R_i\leftrightarrow R_k$, we have by Corollary \ref{cor:mrirj}, Part 1, that $R_i\prec R_k$ or $R_k\prec R_i$. \\

\item By symmetry, it suffices to prove that $R_j\prec R_i$ implies $R_k\prec R_i$. Suppose that $R_j\prec R_i$, and then because $R_i\leftrightarrow R_j$, by Proposition \ref{prop:prec}, Part 1, we have $R_j\subseteq R_i$, that is $l(R_i)\leq l(R_j)\leq r(R_j)\leq r(R_i)$. Suppose that $l(R_k)<l(R_i)$. By Proposition \ref{prop:mrirj} Parts 1 and 2, if $r(R_k)<l(R_i)-1\leq l(R_j)-1$, then $R_j\leftrightarrow R_k$, if $l(R_i)-1\leq r(R_k)\leq r(R_i)-1$, then $R_i\nleftrightarrow R_k$, and if $r(R_k)\geq r(R_i)\geq r(R_j)$, then again $R_j\leftrightarrow R_k$, a contradiction in all cases, so we must have $l(R_k)\geq l(R_i)$. Similarly, by rotating, we must have $r(R_k)\leq r(R_i)$, so we have $R_k\subseteq R_i$ and $R_k\prec R_i$. 
\end{enumerate}
\end{proof}

\begin{proposition} \label{prop:mncp} Let $\bm\lambda=(R_1,\ldots,R_n)$ be a minimal noncommuting path. 
\begin{enumerate}
\item If $R_i\prec R_j$ and $R_j\prec R_i$ for some $j\geq i+2$, then we must have $n=3$ or $n=4$ and $j=i+2$. In particular, if $n\geq 5$, then $R_i\prec R_j$ implies that in fact $R_i\precnsim R_j$. \\

\item If $R_j\prec R_i$ for some $j\geq i+2$, then $R_k\prec R_i$ for every $k\geq i+2$, with the possible exception of $k=n$ if $i=1$. Similarly, if $R_j\prec R_i$ for some $j\leq i-2$, then $R_k\prec R_i$ for every $k\leq i-2$, with the possible exception of $k=1$ if $i=n$. \\

\item Suppose that $R_j\precnsim R_i$ for some $i\geq j+2$, that $i$ is minimal with these two properties, and that $i\neq n$ if $j=1$. Then either $l(R_{i-1})>\cdots>l(R_j)$ and $l(R_i)<l(R_{i-1})$, or $l(R_{i-1})<\cdots<l(R_j)$ and $l(R_i)>l(R_{i-1})$.
\end{enumerate}
\end{proposition}

\begin{remark}
Figure \ref{fig:mncpequalrows} shows that if $n\leq 4$, then it is possible to have a minimal noncommuting path with $R_1\prec R_3$ and $R_3\prec R_1$. If $n\geq 5$, then this will not happen because there will be some $R_t\nleftrightarrow R_{t'}$ with $t'\geq t+2$ and $(t,t')\neq(1,n)$, contradicting minimality.
\begin{figure}
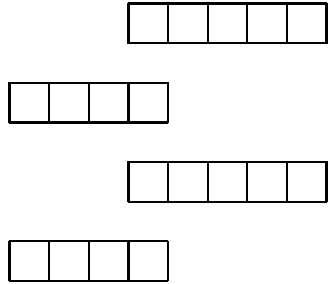
\caption{A minimal noncommuting path with $R_1\prec R_3$ and $R_3\prec R_1$}\label{fig:mncpequalrows}
$$\tableau{&&& \ & \ & \ & \ & \ \\ \\ \ & \ & \ & \ \\ \\ &&& \ & \ & \ & \ & \ \\ \\ \ & \ & \ & \ }$$
\end{figure}
Informally, (2) states that if we are contained in $R_i$, then we must remain stuck in $R_i$ and (3) states that if $R_j$ is contained in some minimal $R_i$, then we must move in the same direction until $R_i$. 
\end{remark}

\begin{proof} [Proof of Proposition \ref{prop:mncp}. ]\hspace{2pt}
\begin{enumerate}
\item Because $j\geq i+2$, we have $R_i\leftrightarrow R_j$ by minimality and therefore $R_i=R_j$ by Proposition \ref{prop:prec}. Now $R_{i+1}\nleftrightarrow R_j$, so by minimality of the noncommuting path we must have $j=i+2$. Similarly, if $i\geq 2$, then $R_{i-1}\nleftrightarrow R_j$, so we must then have $i=2$ and $j=n=4$, if $j\leq n-1$, then $R_i\nleftrightarrow R_{j+1}$, so we must have $i=1$ and $j=3=n-1$ so again $n=4$, and otherwise we have $i=1$ and $j=3=n$. \\

\item If $i+2\leq k\leq n-1$, $R_k\prec R_i$, and $k+1\neq n$ in the case of $i=1$, then because $R_i\leftrightarrow R_k$, $R_i\leftrightarrow R_{k+1}$, and $R_k\nleftrightarrow R_{k+1}$, by Proposition \ref{prop:prec2}, Part 3, we have $R_k\prec R_i$ if and only if $R_{k+1}\prec R_i$, so the first statement follows by induction on $k$ and the second statement follows by rotating.\\

\item Recall that because $R_j\prec R_i$ and $R_j\leftrightarrow R_i$, by Proposition \ref{prop:prec} we must have $l(R_i)\leq l(R_j)$ and $r(R_i)\geq r(R_j)$. Suppose that either $l(R_{j+1})>l(R_j)$ and $l(R_{t+1})<l(R_t)$ for some minimal $j+1\leq t\leq i-2$, or $l(R_{j+1})<l(R_j)$ and $l(R_{t+1})>l(R_t)$ for some minimal $j+1\leq t\leq i-2$, so in particular, $n\geq 5$. Then by Proposition \ref{prop:prec2}, Part 2, we have either $R_{t-1}\prec R_{t+1}$ or $R_{t+1}\prec R_{t-1}$. If $R_{t-1}\prec R_{t+1}$, then by the previous two parts we have $R_j\precnsim R_{t+1}$, contradicting minimality of $i$. If $R_{t+1}\prec R_{t-1}$, then by the previous two parts we have $R_i\precnsim R_{t-1}$, but this means that $t\geq j+2$, and now by Proposition \ref{prop:prec}, Part 3, if $l(R_{j+1})>l(R_j)$ this is impossible because $l(R_i)\leq l(R_j)<l(R_t)$, and if $l(R_{j+1})<l(R_j)$ this is impossible because $r(R_i)\geq r(R_j)>r(R_t)$. Therefore, we must have either $l(R_{i-1})>\cdots>l(R_j)$ and $l(R_i)\leq l(R_j)<l(R_{i-1})$, or $l(R_{i-1})<\cdots<l(R_j)$ and $r(R_i)\geq r(R_j)>r(R_{i-1})$, so $l(R_i)>l(R_{i-1})$ by Corollary \ref{cor:mrirj}, Part 2.
\end{enumerate}
\end{proof}

\begin{proposition} \label{prop:ellgreat}
Let $\bm\lambda=(R_1,\ldots,R_n)$ be a horizontal-strip with $R_1\nleftrightarrow R_2$. Let $\bm\mu=(S_1,\ldots,S_n)$ be a horizontal-strip with $\varphi:\Pi(\bm\lambda)\xrightarrow\sim\Pi(\bm\mu)$ and let $i=\varphi_1$ and $j=\varphi_2$. If $l(R_2)>l(R_1)$, also assume that $l(S_j)>l(S_i)$. Then \begin{equation}
l(S_j)-l(S_i)\geq l(R_2)-l(R_1)\text{ with equality only if }i<j.\end{equation}
\end{proposition}

\begin{remark}
Informally, Proposition \ref{prop:ellgreat} states that the leftmost possible position of $S_j$, given the weighted graph, occurs when $S_j$ is above $S_i$ and $S_i\nleftrightarrow S_j$. The following example demonstrates the necessity of the hypothesis that if $l(R_2)>l(R_1)$, then $l(S_j)>l(S_i)$. 
$$\begin{tikzpicture}
\draw (-0.25,0.25) node (0) {$\bm \lambda=$};
\draw (4.5,-0.25) node (1) {$R_1$} (4.5,0.75) node (2) {$R_2$};
\draw (0.5,-0.5) -- (2.5,-0.5) -- (2.5,0) -- (0.5,0) -- (0.5,-0.5) (1,-0.5) -- (1,0) (1.5,-0.5) -- (1.5,0) (2,-0.5) -- (2,0) (2.5,0.5) -- (4,0.5) -- (4,1) -- (2.5,1) -- (2.5,0.5) (3,0.5) -- (3,1) (3.5,0.5) -- (3.5,1);
\end{tikzpicture}\hspace{50pt}
\begin{tikzpicture}
\draw (-0.25,0.25) node (0) {$\bm \mu=$};
\draw (5,-0.25) node (1) {$S_i$} (5,0.75) node (2) {$S_j$};
\draw (2.5,-0.5) -- (4.5,-0.5) -- (4.5,0) -- (2.5,0) -- (2.5,-0.5) (3,-0.5) -- (3,0) (3.5,-0.5) -- (3.5,0) (4,-0.5) -- (4,0) (0.5,0.5) -- (2,0.5) -- (2,1) -- (0.5,1) -- (0.5,0.5) (1,0.5) -- (1,1) (1.5,0.5) -- (1.5,1);
\end{tikzpicture}$$
\end{remark}

\begin{proof}[Proof of Proposition \ref{prop:ellgreat}. ]
We calculate directly, using \eqref{eq:mij} and noting that $|R|=r(R)-l(R)+1$. If $l(R_2)<l(R_1)$, then the statement holds unless $l(S_j)<l(S_i)$, and because $R_1\nleftrightarrow R_2$, we have by Corollary \ref{cor:mrirj}, Part 3, that $M_{1,2}(\bm\lambda)=M_{i,j}(\bm\mu)>0$. By Proposition \ref{prop:mrirj}, Parts 2 and 3, we have $M_{1,2}(\bm\lambda)=r(R_2)-l(R_1)+2$ and we either have $M_{i,j}(\bm\mu)=r(S_j)-l(S_i)+1+\chi(i<j)$, or $M_{i,j}(\bm\mu)=\min\{|S_i|,|S_j|\}\leq r(S_j)-l(S_i)+1+\chi(i<j)$. Therefore we have \begin{align} l(S_j)-l(S_i)&=r(S_j)-|S_j|+1-l(S_i)\geq M_{i,j}(\bm\mu)-|S_j|-\chi(i<j)\\\nonumber&\geq M_{1,2}(\bm\lambda)-|R_2|-1=r(R_2)-|R_2|+1-l(R_1)=l(R_2)-l(R_1),\end{align} with equality only if $i<j$.\\

Similarly, if $l(R_2)>l(R_1)$, then by hypothesis we have $l(S_j)>l(S_i)$, and because $R_1\nleftrightarrow R_2$, we have by Corollary \ref{cor:mrirj}, Part 3, that $M_{i,j}(\bm\mu)<\min\{|S_i|,|S_j|\}$. By Proposition \ref{prop:mrirj}, Parts 1 and 3, we have $M_{1,2}(\bm\lambda)=r(R_1)-l(R_2)+1$ and we either have $M_{i,j}(\bm\mu)=r(S_i)-l(S_j)+1+\chi(i<j)$ or $M_{i,j}(\bm\mu)=0\geq r(S_i)-l(S_j)+1+\chi(i<j)$. Therefore we have \begin{align}
l(S_j)-l(S_i)&=l(S_j)-r(S_i)-1+|S_i|\geq |S_i|-M_{i,j}(\bm\mu)+\chi(i>j)\\\nonumber&\geq |R_1|-M_{1,2}(\bm\lambda)=l(R_2)-r(R_1)-1+|R_1|=l(R_2)-l(R_1),
\end{align} with equality only if $i<j$. 
\end{proof}

\begin{proposition} \label{prop:mncp3support} Let $\bm\lambda=(R_1,\ldots,R_n)$ be a horizontal-strip with $l(R_1)<l(R_3)$ and $R_1\nleftrightarrow R_3$. Let $\bm\mu=(S_1,\ldots,S_n)$ be a horizontal-strip and $\varphi:\Pi(\bm\lambda)\xrightarrow\sim\Pi(\bm\mu)$ such that $\varphi_1<\varphi_3$, $l(S_{\varphi_1})<l(S_{\varphi_3})$, and $S_{\varphi_1}\nleftrightarrow S_{\varphi_3}$. Then if $l(R_2)>l(R_1)$ and $R_1\nleftrightarrow R_2$, then $l(S_{\varphi_2})>l(S_{\varphi_1})$. Similarly, if $l(R_3)>l(R_2)$ and $R_2\nleftrightarrow R_3$, then $l(S_{\varphi_3})>l(S_{\varphi_2})$. \end{proposition}

\begin{remark} Informally, Proposition \ref{prop:mncp3support} describes the extent to which rows $R_1$ and $R_3$ with $l(R_1)<l(R_3)$ and $R_1\nleftrightarrow R_3$, given the weighted graph data, determine the relative horizontal position of another row $R_2$.\end{remark}

\begin{proof}[Proof of Proposition \ref{prop:mncp3support}. ]
We first suppose for a contradiction that $l(R_2)>l(R_1)$ and $R_1\nleftrightarrow R_2$, but $l(S_{\varphi_2})\leq l(S_{\varphi_1})$. Because $l(R_2)>l(R_1)$ and $R_1\nleftrightarrow R_2$, we have by Corollary \ref{cor:mrirj}, Part 2, that $r(R_2)>r(R_1)$ and by Corollary \ref{cor:mrirj}, Part 3, that $M_{1,2}(\bm\lambda)<\min\{|R_1|,|R_2|\}$. Therefore, we must have $M_{\varphi_1,\varphi_2}(\bm\mu)<\min\{|S_{\varphi_1}|,|S_{\varphi_2}|\}$ and by Proposition \ref{prop:mrirj}, Part 2, we have that $l(S_{\varphi_2})<l(S_{\varphi_1})<l(S_{\varphi_3})$ and $r(S_{\varphi_2})<r(S_{\varphi_1})<r(S_{\varphi_3})$. In particular, we have $M_{\varphi_2,\varphi_3}(\bm\mu)<\min\{|S_{\varphi_2}|,|S_{\varphi_3}|\}$, so $M_{2,3}(\bm\lambda)<\min\{|R_2|,|R_3|\}$ and $l(R_2)\neq l(R_3)$. \\

Now if $l(R_2)<l(R_3)$, then by \eqref{eq:mrirj} we have \begin{equation} M_{2,3}(\bm\lambda)=r(R_2)-l(R_3)+1>r(R_1)-l(R_3)+1=M_{1,3}(\bm\lambda),\end{equation}
but if $r(S_{\varphi_2})<l(S_{\varphi_3})-1$, then $M_{\varphi_2,\varphi_3}(\bm\mu)=0\leq M_{\varphi_1,\varphi_3}(\bm\mu)$, and if $r(S_{\varphi_2})\geq l(S_{\varphi_3})-1$, then \begin{equation} M_{\varphi_2,\varphi_3}(\bm\mu)=r(S_{\varphi_2})-l(S_{\varphi_3})+1+\chi(\varphi_2>\varphi_3)\leq r(S_{\varphi_1})-l(S_{\varphi_3})+1=M_{\varphi_1,\varphi_3}(\bm\mu),\end{equation}
a contradiction in either case. Similarly, if $l(R_2)>l(R_3)$, then by \eqref{eq:mrirj} we have 
\begin{equation}
M_{2,3}(\bm\lambda)=r(R_3)-l(R_2)+2>r(R_1)-l(R_2)+1=M_{1,2}(\bm\lambda),\end{equation}
but if $r(S_{\varphi_2})<l(S_{\varphi_3})-1$, then $M_{\varphi_2,\varphi_3}(\bm\mu)=0\leq M_{\varphi_1,\varphi_2}(\bm\mu)$, and if $r(S_{\varphi_2})\geq l(S_{\varphi_3})-1$, then 
\begin{equation}
M_{\varphi_2,\varphi_3}(\bm\mu)= r(S_{\varphi_2})-l(S_{\varphi_3})+1+\chi(\varphi_2>\varphi_3)\leq r(S_{\varphi_2})-l(S_{\varphi_3})+1=M_{\varphi_1,\varphi_2}(\bm\mu),
\end{equation}
a contradiction in either case. This proves the first claim and the second follows by rotating.
\end{proof}

\begin{proposition} \label{prop:ellgreat2} Let $\bm\lambda=(R_1,\ldots,R_n)$ be a minimal noncommuting path with $l(R_n)>\cdots>l(R_1)$. Let $\bm\mu=(S_1,\ldots,S_n)$ be a horizontal-strip, $\varphi:\Pi(\bm\lambda)\xrightarrow\sim\Pi(\bm\mu)$, and let $i$ be such that $l(S_i)$ is minimal and $j$ be such that $l(S_j)$ is maximal. Then \begin{equation}
r(S_j)-l(S_i)\geq r(R_n)-l(R_1)\text{ with equality only if }i=1\text{ and }j=n.
\end{equation}
\end{proposition}

\begin{example}
Proposition \ref{prop:ellgreat2} describes a situation like the one in Figure \ref{fig:ellgreat2}. Informally, when the rows of $\bm\lambda$ move up and to the right, it is possible to have $M_{t,t+1}(\bm\lambda)=0$, and therefore the weighted graph $\Pi(\bm\lambda)$ does not know that these rows do not commute. As a result, the corresponding rows may not be adjacent in $\bm\mu$ but could be permuted in some way. This is a crucial example to keep in mind because this is a recurring challenge we must address. Proposition \ref{prop:ellgreat2} states that we can still conclude that the leftmost possible position of the rightmost row occurs when it is at the top. 

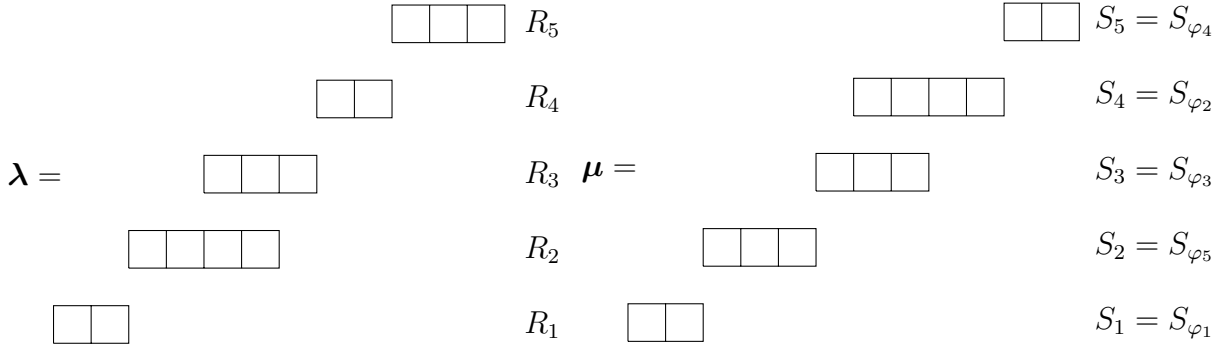
\begin{figure}\caption{A minimal noncommuting path $\bm\lambda=(R_1,\ldots,R_5)$ with $l(R_5)>\cdots>l(R_1)$ and a horizontal-strip $\bm\mu=(S_1,\ldots,S_5)$ with $\Pi(\bm\lambda)\cong\Pi(\bm\mu)$}\label{fig:ellgreat2}
$$\begin{tikzpicture}
\draw (1.25,-0.25) node (0) {$\bm \lambda=$};
\draw (8,-2.25) node (1) {$R_1$} (8,-1.25) node (2) {$R_2$} (8,-0.25) node (3) {$R_3$} (8,0.75) node (4) {$R_4$} (8,1.75) node (5) {$R_5$};
\draw (1.5,-2.5) -- (2.5,-2.5) -- (2.5,-2) -- (1.5,-2) -- (1.5,-2.5) (2,-2.5) -- (2,-2) (2.5,-1.5) -- (4.5,-1.5) -- (4.5,-1) -- (2.5,-1) -- (2.5,-1.5) (3,-1.5) -- (3,-1) (3.5,-1.5) -- (3.5,-1) (4,-1.5) -- (4,-1) (3.5,-0.5) -- (5,-0.5) -- (5,0) -- (3.5,0) -- (3.5,-0.5) (4,-0.5) -- (4,0) (4.5,-0.5) -- (4.5,0) (5,0.5) -- (6,0.5) -- (6,1) -- (5,1) -- (5,0.5) (5.5,0.5) -- (5.5,1) (6,1.5) -- (7.5,1.5) -- (7.5,2) -- (6,2) -- (6,1.5) (6.5,1.5) -- (6.5,2) (7,1.5) -- (7,2);
\end{tikzpicture}
\begin{tikzpicture}
\draw (1.25,-0.25) node (0) {$\bm \mu=$};
\draw (8.5,-2.25) node (1) {$S_1=S_{\varphi_1}$} (8.5,-1.25) node (2) {$S_2=S_{\varphi_5}$} (8.5,-0.25) node (3) {$S_3=S_{\varphi_3}$} (8.5,0.75) node (4) {$S_4=S_{\varphi_2}$} (8.5,1.75) node (5) {$S_5=S_{\varphi_4}$};
\draw (1.5,-2.5) -- (2.5,-2.5) -- (2.5,-2) -- (1.5,-2) -- (1.5,-2.5) (2,-2.5) -- (2,-2) (2.5,-1.5) -- (4,-1.5) -- (4,-1) -- (2.5,-1) -- (2.5,-1.5) (3,-1.5) -- (3,-1) (3.5,-1.5) -- (3.5,-1) (4,-0.5) -- (5.5,-0.5) -- (5.5,0) -- (4,0) -- (4,-0.5) (4.5,-0.5) -- (4.5,0) (5,-0.5) -- (5,0) (4.5,0.5) -- (6.5,0.5) -- (6.5,1) -- (4.5,1) -- (4.5,0.5) (5,0.5) -- (5,1) (5.5,0.5) -- (5.5,1) (6,0.5) -- (6,1) (6.5,1.5) -- (7.5,1.5) -- (7.5,2) -- (6.5,2) -- (6.5,1.5) (7,1.5) -- (7,2);
\end{tikzpicture}$$
\end{figure}
\end{example}

\begin{proof}[Proof of Proposition \ref{prop:ellgreat2}. ]
Note that we cannot directly apply Proposition \ref{prop:ellgreat} because we do not know that $l(S_{\varphi_{t+1}})>l(S_{\varphi_t})$ for every $t$. Instead, we will reorder the rows of $\bm\mu$ and compute directly. By Proposition \ref{prop:prec2}, Part 1, we have $M_{t,t'}(\bm\lambda)=0$ if $|t'-t|\geq 2$ and by Corollary \ref{cor:mrirj}, Part 3, we have $M_{t,t+1}(\bm\lambda)<\min\{|R_t|,|R_{t+1}|\}$, so $M_{\varphi_t,\varphi_{t+1}}(\bm\mu)<\min\{|S_{\varphi_t}|,|S_{\varphi_{t+1}}|\}$ and in particular the $l(S_{\varphi_t})$ are distinct. Let $\sigma$ be the permutation that sorts the rows of $\bm\mu$ so that \begin{equation}
l(S_i)=l(S_{\sigma_1})<l(S_{\sigma_2})<\cdots<l(S_{\sigma_{n-1}})<l(S_{\sigma_n})=l(S_j).
\end{equation}
Now by \eqref{eq:mij}, we have
\begin{align}
r(S_j)-l(S_i)&=l(S_{\sigma_n})-l(S_{\sigma_1})+|S_{\sigma_n}|-1\\\nonumber&=l(S_{\sigma_{n-1}})-l(S_{\sigma_1})+|S_{\sigma_n}|+|S_{\sigma_{n-1}}|-M_{\sigma_{n-1},\sigma_n}(\bm\mu)-1+\chi(\sigma_{n-1}>\sigma_n)\\\nonumber\cdots&=\sum_{t=1}^n|S_{\sigma_t}|-M(\bm\mu)-1+\sum_{t=1}^{n-1}\chi(\sigma_t>\sigma_{t+1})\\\nonumber&\geq\sum_{t=1}^n|R_t|-M(\bm\lambda)-1\\\nonumber\cdots&=l(R_{n-1})-l(R_1)+|R_n|+|R_{n-1}|-M_{n-1,n}(\bm\lambda)-1\\\nonumber&=r(S_n)-l(S_1)+|R_n|-1=r(R_n)-l(R_1),
\end{align}
with equality only if $i=\sigma_1<\cdots<\sigma_n=j$, so $i=1$ and $j=n$. 
\end{proof}

We now describe the structure of a minimal noncommuting path $(R_1,\ldots,R_n)$, where $l(R_1)<l(R_n)$ and $R_1\nleftrightarrow R_n$. Informally, such a minimal noncommuting path must look loosely like one of the examples in Figure \ref{fig:mncps}. 

\begin{lemma}\label{lem:mncp1}
Let $\bm\lambda=(R_1,\ldots,R_n)$ be a minimal noncommuting path with $n\geq 4$, $l(R_1)<l(R_n)$, and $R_1\nleftrightarrow R_n$. 
\begin{enumerate}
\item Suppose that there is no $i\geq 3$ for which $R_1\precnsim R_i$ and that there is no $j\leq n-2$ for which $R_n\precnsim R_j$. Then one of the following holds.
\begin{itemize}
\item We have $l(R_2)>l(R_1)$, $l(R_{n-1})<\cdots<l(R_2)$, and $l(R_n)>l(R_{n-1})$. 
\item We have $l(R_{n-1})>\cdots>l(R_1)$, $R_t\prec R_n$ for $2\leq t\leq n-2$, and $l(R_n)<l(R_{n-1})$. 
\item We have $l(R_2)<l(R_1)$, $l(R_n)>\cdots>l(R_2)$, and $R_t\prec R_1$ for $3\leq t\leq n-1$.\\
\end{itemize}

\item Now suppose that $R_1\precnsim R_i$ for some minimal $i\geq 3$, and that in fact $i=n-1$. Then $l(R_n)>l(R_{n-1})$ and one of the following holds. 
\begin{itemize}
\item We have $l(R_{i-1})<\cdots<l(R_1)$ and $l(R_i)>l(R_{i-1})$. 
\item We have $l(R_{i-1})>\cdots>l(R_1)$, $R_t\prec R_n$ for $2\leq t\leq i-1$, and $l(R_i)<l(R_{i-1})$. 
\item We have $n=4$, $l(R_2)>l(R_1)$, $l(R_3)<l(R_2)$, and $R_4\precnsim R_2$.\\
\end{itemize}

\item Now suppose that $R_1\precnsim R_i$ for some minimal $i\geq 3$, and that $i\leq n-2$. Then $R_n\precnsim R_i$, so $R_n\precnsim R_j$ for some maximal $i\leq j\leq n-2$. Additionally, one of the following holds. 
\begin{itemize}
\item We have $l(R_{i-1})<\cdots<l(R_1)$ and $l(R_i)>l(R_{i-1})$. 
\item We have $i=j=3$, $l(R_2)>l(R_1)$, $R_n\precnsim R_2$, and $l(R_3)<l(R_2)$. 
\end{itemize}
Similarly, one of the following holds.
\begin{itemize}
\item We have $l(R_n)<\cdots<l(R_{j+1})$ and $l(R_{j+1})>l(R_j)$.
\item We have $i=j=n-2$, $l(R_n)>l(R_{n-1})$, $R_1\precnsim R_{n-1}$, and $l(R_{n-1})<l(R_{n-2})$.
\end{itemize}
Finally, we must have either $j=i$, or $j=i+1$ and $l(R_j)<l(R_i)$. 
\end{enumerate}
\end{lemma}

\begin{remark}
If the hypothesis of (1) does not hold, then $R_1\precnsim R_i$ for some $i\geq 3$ or $R_n\precnsim R_j$ for some $j\leq n-2$. By rotating, we may assume that $R_1\precnsim R_i$ for some $i\geq 3$, and by Corollary \ref{cor:mrirj}, Part 3, we cannot have $R_1\prec R_n$, so $i\leq n-1$. Therefore (2) and (3) cover all of the cases we will need.
\end{remark}

\begin{proof} [Proof of Lemma \ref{lem:mncp1}. ] Note that because $n\geq 4$, we must have $R_2\leftrightarrow R_n$ by minimality. \\

Suppose that there is no $i\geq 3$ for which $R_1\precnsim R_i$ and that there is no $j\leq n-2$ for which $R_n\precnsim R_j$. Additionally, suppose that $l(R_2)>l(R_1)$. Because $l(R_1)<l(R_n)$, by Proposition \ref{prop:prec2}, Part 2, we must have $R_2\prec R_n$ or $R_n\prec R_2$, so by our hypothesis we must have $R_2\prec R_n$ and by Proposition \ref{prop:mncp}, Part 2, that $R_t\prec R_n$ for $2\leq t\leq n-2$. Note that if $R_2\prec R_k$ for any $4\leq k\leq n-1$, then we would have $R_1\prec R_k$ by Proposition \ref{prop:mncp}, Part 2, and because $n\geq 5$, $R_1\precnsim R_k$ by Proposition \ref{prop:mncp} Part 1, contradicting our hypothesis. Therefore, by Proposition \ref{prop:mncp}, Part 3, we have either $l(R_{n-1})<\cdots<l(R_2)$ and $l(R_n)>l(R_{n-1})$ and the first possibility holds, or we have $l(R_{n-1})>\cdots>l(R_2)>l(R_1)$ and $l(R_n)<l(R_{n-1})$ and the second possibility holds. Now suppose that $l(R_2)<l(R_1)$. If $l(R_n)>l(R_{n-1})$, then by rotating we can reduce to the case where $l(R_2)>l(R_1)$ and it follows that the third possibility holds, so it remains to consider the case where $l(R_n)<l(R_{n-1})$. Because $l(R_{n-1})>l(R_n)>r(R_2)\geq l(R_2)$, we have $R_2\leftrightarrow R_{n-1}$ and $n\geq 5$, so Proposition \ref{prop:mncp}, Part 1 applies. Also, we must have $l(R_{t+1})>l(R_t)$ for some minimal $2\leq t\leq n-2$. By Proposition \ref{prop:prec2}, Part 2, we must have $R_{t-1}\prec R_{t+1}$ or $R_{t+1}\prec R_{t-1}$. However, by Proposition \ref{prop:mncp}, Part 2, if $R_{t-1}\prec R_{t+1}$, then $R_1\precnsim R_{t+1}$, and if $R_{t+1}\prec R_{t-1}$, then $R_n\precnsim R_{t-1}$, contradicting our hypothesis in both cases.\\

Now suppose that $R_1\precnsim R_i$ for some minimal $i\geq 3$, and that in fact $i=n-1$. Then $l(R_{n-1})\leq l(R_1)<l(R_n)$ by Proposition \ref{prop:prec}. By Proposition \ref{prop:mncp}, Part 2, we must have either $l(R_{i-1})<\cdots<l(R_1)$ and $l(R_i)>l(R_{i-1})$ and the first possibility holds, or we have $l(R_{i-1})>\cdots>l(R_1)\geq l(R_i)$ and $l(R_i)<l(R_{i-1})$. Then because $l(R_n)>l(R_1)$, by Proposition \ref{prop:prec2}, Part 2, we must have either $R_2\prec R_n$, in which case we have $R_t\prec R_n$ for $2\leq t\leq n-2=i-1$ by Proposition \ref{prop:mncp}, Part 2, and the second possibility holds, or we have $R_n\precnsim R_2$. In this case, if $n\geq 5$ then $R_i\prec R_2$ and $R_2\prec R_i$ by Proposition \ref{prop:mncp}, Part 2, but this is impossible by Proposition \ref{prop:mncp}, Part 1, so we must have $n=4$ and the third possibility holds.\\

Now suppose that $R_1\precnsim R_i$ for some minimal $i\geq 3$, and that $i\leq n-2$, so in particular, $n\geq 5$ and Proposition \ref{prop:mncp}, Part 1 applies. Then because $R_1\nleftrightarrow R_n$, by Proposition \ref{prop:prec2}, Part 3, we have $R_n\precnsim R_i$, so $R_n\precnsim R_j$ for some maximal $i\leq j\leq n-2$. By Proposition \ref{prop:mncp}, Part 3 we must have either $l(R_{i-1})<\cdots<l(R_1)$ and $l(R_i)>l(R_{i-1})$, in which case the first possibility holds, or we have $l(R_{i-1})>\cdots>l(R_1)$ and $l(R_i)<l(R_{i-1})$. But in this case, because $l(R_n)>l(R_1)$, by Proposition \ref{prop:prec2}, Part 2, we have $R_2\prec R_n$ or $R_n\prec R_2$. If $R_2\prec R_n$, then by Proposition \ref{prop:mncp}, Part 2, we would have $R_j\prec R_n$, which is impossible by definition of $j$, so we must have $R_n\precnsim R_2$. If $j\geq 4$, then we would have $R_j\prec R_2$ by Proposition \ref{prop:mncp}, Part 2, but because $R_n\prec R_j$ and $R_1\nleftrightarrow R_n$, by Proposition \ref{prop:prec2}, Part 3, we have $R_1\prec R_j$ and $R_2\prec R_j$ by Proposition \ref{prop:mncp}, Part 2 again, which is impossible by Proposition \ref{prop:mncp}, Part 1. Therefore, we must have $i=j=3$ and the second possibility holds. This proves the first claim and the second follows by rotating. Finally, if $j\geq i+2$, then by Proposition \ref{prop:mncp}, Part 2, we would have $R_i\prec R_j$ and $R_j\prec R_i$, which is impossible by Proposition \ref{prop:mncp}, Part 1, so we must have either $j=i$ or $j=i+1$. If $j=i+1$, then by Proposition \ref{prop:mncp}, Part 2, we have $R_{i-1}\prec R_j$ and now $l(R_j)\leq l(R_{i-1})<l(R_i)$ by Proposition \ref{prop:prec}. This completes the proof. 
\end{proof}

We now describe the image of a minimal noncommuting path under an isomorphism of weighted graphs.

\begin{lemma} \label{lem:mncp2}
Let $\bm\lambda=(R_1,\ldots,R_n)$ be a minimal noncommuting path with $l(R_1)<l(R_n)$ and $R_1\nleftrightarrow R_n$. Let $\bm\mu=(S_1,\ldots,S_n)$ and $\varphi:\Pi(\bm\lambda)\xrightarrow\sim\Pi(\bm\mu)$ satisfy $\varphi_1=1$, $l(S_1)<l(S_{\varphi_n})$, and $S_1\nleftrightarrow S_{\varphi_n}$. Then $\varphi_n=n$.
\end{lemma}

\begin{example}
Informally, Lemma \ref{lem:mncp2} describes a situation like the one below. The conditions $l(R_1)<l(R_6)$, $R_1\nleftrightarrow R_6$, $l(S_1)<l(S_{\varphi_6})$, and $S_1\nleftrightarrow S_{\varphi_6}$ fix the horizontal positions of $R_1$, $R_6$, $S_1$, and $S_{\varphi_6}$. Then if we have a minimal noncommuting path in $\bm\lambda$ from $R_1$ to $R_6$, the row $S_{\varphi_6}$ must be above the other rows in $\bm\mu$. Note that the intermediate rows can be permuted as in this example.

$$
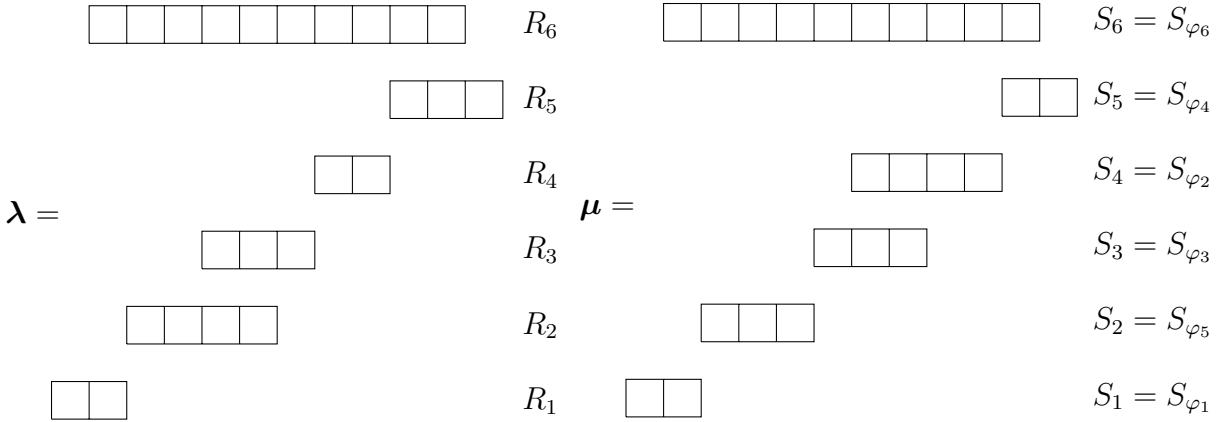
\begin{tikzpicture}
\draw (1.25,0.25) node (0) {$\bm \lambda=$};
\draw (8,-2.25) node (1) {$R_1$} (8,-1.25) node (2) {$R_2$} (8,-0.25) node (3) {$R_3$} (8,0.75) node (4) {$R_4$} (8,1.75) node (5) {$R_5$} (8,2.75) node (6) {$R_6$};
\draw (1.5,-2.5) -- (2.5,-2.5) -- (2.5,-2) -- (1.5,-2) -- (1.5,-2.5) (2,-2.5) -- (2,-2) (2.5,-1.5) -- (4.5,-1.5) -- (4.5,-1) -- (2.5,-1) -- (2.5,-1.5) (3,-1.5) -- (3,-1) (3.5,-1.5) -- (3.5,-1) (4,-1.5) -- (4,-1) (3.5,-0.5) -- (5,-0.5) -- (5,0) -- (3.5,0) -- (3.5,-0.5) (4,-0.5) -- (4,0) (4.5,-0.5) -- (4.5,0) (5,0.5) -- (6,0.5) -- (6,1) -- (5,1) -- (5,0.5) (5.5,0.5) -- (5.5,1) (6,1.5) -- (7.5,1.5) -- (7.5,2) -- (6,2) -- (6,1.5) (6.5,1.5) -- (6.5,2) (7,1.5) -- (7,2) (2,2.5) -- (7,2.5) -- (7,3) -- (2,3) -- (2,2.5) (2.5,2.5) -- (2.5,3) (3,2.5) -- (3,3) (3.5,2.5) -- (3.5,3) (4,2.5) -- (4,3) (4.5,2.5) -- (4.5,3) (5,2.5) -- (5,3) (5.5,2.5) -- (5.5,3) (6,2.5) -- (6,3) (6.5,2.5) -- (6.5,3);
\end{tikzpicture}
\begin{tikzpicture}
\draw (1.25,0.25) node (0) {$\bm \mu=$};
\draw (8.5,-2.25) node (1) {$S_1=S_{\varphi_1}$} (8.5,-1.25) node (2) {$S_2=S_{\varphi_5}$} (8.5,-0.25) node (3) {$S_3=S_{\varphi_3}$} (8.5,0.75) node (4) {$S_4=S_{\varphi_2}$} (8.5,1.75) node (5) {$S_5=S_{\varphi_4}$} (8.5,2.75) node (6) {$S_6=S_{\varphi_6}$};
\draw (1.5,-2.5) -- (2.5,-2.5) -- (2.5,-2) -- (1.5,-2) -- (1.5,-2.5) (2,-2.5) -- (2,-2) (2.5,-1.5) -- (4,-1.5) -- (4,-1) -- (2.5,-1) -- (2.5,-1.5) (3,-1.5) -- (3,-1) (3.5,-1.5) -- (3.5,-1) (4,-0.5) -- (5.5,-0.5) -- (5.5,0) -- (4,0) -- (4,-0.5) (4.5,-0.5) -- (4.5,0) (5,-0.5) -- (5,0) (4.5,0.5) -- (6.5,0.5) -- (6.5,1) -- (4.5,1) -- (4.5,0.5) (5,0.5) -- (5,1) (5.5,0.5) -- (5.5,1) (6,0.5) -- (6,1) (6.5,1.5) -- (7.5,1.5) -- (7.5,2) -- (6.5,2) -- (6.5,1.5) (7,1.5) -- (7,2) (2,2.5) -- (7,2.5) -- (7,3) -- (2,3) -- (2,2.5) (2.5,2.5) -- (2.5,3) (3,2.5) -- (3,3) (3.5,2.5) -- (3.5,3) (4,2.5) -- (4,3) (4.5,2.5) -- (4.5,3) (5,2.5) -- (5,3) (5.5,2.5) -- (5.5,3) (6,2.5) -- (6,3) (6.5,2.5) -- (6.5,3);
\end{tikzpicture}$$
\end{example}

\begin{proof}[Proof of Lemma \ref{lem:mncp2}. ]
By translating all rows, we may assume without loss of generality that $l(R_1)=l(S_1)$, and then by \eqref{eq:mij} we have $l(R_n)=l(S_{\varphi_n})$ as well. The idea is to repeatedly apply Proposition \ref{prop:ellgreat} and Proposition \ref{prop:ellgreat2} to write the inequality $l(S_{\varphi_n})\geq l(R_n)$, for which equality holds only if $\varphi_n=n$. Because equality does indeed hold, we will conclude that $\varphi_n=n$. \\

\noindent\textbf{Case 0: We have $n=3$. }

By Proposition \ref{prop:mncp3support}, if $l(R_2)>l(R_1)$, then $l(S_{\varphi_2})>l(S_1)$, so by Proposition \ref{prop:ellgreat} we have $l(S_{\varphi_2})\geq l(R_2)$. By Proposition \ref{prop:mncp3support} again, if $l(R_3)>l(R_2)$, then $l(S_{\varphi_3})>l(S_{\varphi_2})$, so by Proposition \ref{prop:ellgreat} we have $l(S_{\varphi_3})\geq l(R_3)$ with equality only if $\varphi_3>\varphi_2$. Because equality does indeed hold, we must have $\varphi_3>\varphi_2$ and therefore $\varphi_3=3$. \\

We may now assume that $n\geq 4$, so it remains to consider the several possibilities outlined in Lemma \ref{lem:mncp1}. Cases 1a and 1b illustrate the main ideas.\\

\noindent\textbf{Case 1: There is no $i\geq 3$ for which $R_1\precnsim R_i$ and there is no $j\leq n-2$ for which $R_n\precnsim R_j$. }\\

\noindent\textbf{Case 1a: We have $l(R_2)>l(R_1)$, $l(R_{n-1})<\cdots<l(R_2)$, and $l(R_n)>l(R_{n-1})$. }

By Proposition \ref{prop:mncp3support}, we have $l(S_{\varphi_2})>l(S_1)$, so by Proposition \ref{prop:ellgreat} we have $l(S_{\varphi_2})\geq l(R_2)$. By Proposition \ref{prop:ellgreat} again, we have $l(S_{\varphi_{n-1}})\geq l(R_{n-1})$ with equality only if $\varphi_{n-1}>\cdots>\varphi_2$. By Proposition \ref{prop:mncp3support}, we have $l(S_{\varphi_n})>l(S_{\varphi_{n-1}})$, so by Proposition \ref{prop:ellgreat} we have $l(S_{\varphi_n})\geq l(R_n)$ with equality only if we also have $\varphi_n>\varphi_{n-1}$. Because equality does indeed hold, we must have $\varphi_n>\varphi_{n-1}>\cdots>\varphi_2$, so $\varphi_n=n$.\\

\noindent\textbf{Case 1b: We have $l(R_{n-1})>\cdots>l(R_1)$, $R_t\prec R_n$ for $2\leq t\leq n-2$, and $l(R_n)<l(R_{n-1})$. }

As in Case 1a, by Proposition \ref{prop:mncp3support}, we have $l(S_{\varphi_2})>l(S_1)$, so by Proposition \ref{prop:ellgreat} we have $l(S_{\varphi_2})\geq l(R_2)$. However, we must now be careful because we need not have $l(S_{\varphi_{n-1}})>\cdots>l(S_{\varphi_2})$, so we take a slightly different approach. Let $k\in\{\varphi_t: \ 2\leq t\leq n-1\}$ be such that $l(S_k)$ is maximal. By Proposition \ref{prop:ellgreat2}, we have $r(S_k)\geq r(R_{n-1})$ with equality only if $k=n-1$. If $R_{n-1}\nprec R_n$, then we must have $k=\varphi_{n-1}$ because $S_{\varphi_t}\prec S_{\varphi_n}$ for every $2\leq t\leq n-1$ except for $S_k$. Now by Proposition \ref{prop:ellgreat}, we have $l(S_{\varphi_n})\geq l(R_n)$ with equality only if $\varphi_n>\varphi_{n-1}=k=n-1$. Because equality does indeed hold, we must have $\varphi_n=n$. On the other hand, if $R_{n-1}\prec R_n$, then $S_k\prec S_{\varphi_n}$ and by Corollary \ref{cor:mrirj}, Part 3, we have $r(S_{\varphi_n})\geq r(S_k)-1\geq r(R_{n-1})-1=r(R_n)$ with equality only if $\varphi_n>k=n-1$. Because equality does indeed hold, we must have $\varphi_n=n$. \\

\noindent\textbf{Case 1c: We have $l(R_2)<l(R_1)$, $l(R_n)>\cdots>l(R_2)$, and $R_t\prec R_1$ for $3\leq t\leq n-1$. }

By rotating, the conclusion follows from Case 1b.\\

We now assume that $R_1\precnsim R_i$ for some minimal $i\geq 3$ or $R_n\precnsim R_j$ for some maximal $j\leq n-2$. By rotating, we may assume that $R_1\precnsim R_i$ for some minimal $i\geq 3$. Note that by Corollary \ref{cor:mrirj}, Part 3, we cannot have $R_1\prec R_n$, so we have $i\leq n-1$. It remains to consider the cases where $i=n-1$ and where $i\leq n-2$. \\

\noindent\textbf{Case 2: We have $R_1\precnsim R_i$ for some minimal $i\geq 3$, and in fact $i=n-1$. }\\

\noindent\textbf{Case 2a: We have $l(R_{i-1})<\cdots<l(R_1)$ and $l(R_i)>l(R_{i-1})$. }

By Proposition \ref{prop:ellgreat}, we have $l(S_{\varphi_{i-1}})\geq l(R_{i-1})$ with equality only if $\varphi_{i-1}>\cdots>\varphi_2$. If $l(S_{\varphi_i})>l(S_{\varphi_{i-1}})$, then by Proposition \ref{prop:ellgreat} again, we have $l(S_{\varphi_n})\geq l(R_n)$ with equality only if $\varphi_n>\varphi_i>\varphi_{i-1}>\cdots>\varphi_2$. Because equality does indeed hold, we must have $\varphi_n=n$.\\

Now suppose that $l(S_{\varphi_i})\leq l(S_{\varphi_{i-1}})$. We will show that this is narrowly possible, but $\bm\mu$ will be so specifically determined that we will be able to reduce to a previous case. By Corollary \ref{cor:mrirj}, Part 3, we have $M_{i-1,i}(\bm\lambda)<\min\{|R_{i-1}|,|R_i|\}$, so we must have $r(S_1)\leq r(S_{\varphi_i})+1\leq r(S_{\varphi_{i-1}})$. If $i\geq 4$, then because $M_{t,n}(\bm\lambda)=0$ for $2\leq t\leq i-1$, $M_{1,t}(\bm\lambda)=0$ for $3\leq t\leq i-1$ by Proposition \ref{prop:prec2}, Part 1, and $M_{t,t+1}(\bm\lambda)>0$ for $1\leq t\leq i-2$ by Corollary \ref{cor:mrirj}, Part 3, we must have $r(S_{\varphi_{i-1}})\leq l(S_n)-1\leq r(S_1)$, so $M_{1,\varphi_{i-1}}(\bm\mu)>0$ and $i=3$. Now $r(S_{\varphi_2})\leq l(S_{\varphi_4})-1\leq r(S_1)\leq r(S_{\varphi_3})+1\leq r(S_{\varphi_2})$, so we must have equality everywhere, meaning that $M_{1,\varphi_4}(\bm\mu)=M_{\varphi_3,\varphi_4}(\bm\mu)=0$ and  $M_{1,\varphi_2}(\bm\mu)=\min\{|S_1|,|S_{\varphi_2}|\}$, so $R_1\prec R_2$ or $R_2\prec R_1$. If $R_2\prec R_1$, then we have $l(R_3)>l(R_2)=l(R_1)-1$ and in fact $R_1=R_3$, contradicting $R_1\precnsim R_3$. If $R_1\prec R_2$, then we have $|R_2|-1=M_{2,3}(\bm\lambda)=M_{\varphi_2,\varphi_3}(\bm\mu)=|S_{\varphi_3}|-1=|R_3|-1$, so $|R_2|=|R_3|$. Moreover, we have $R_1\prec R_2$, $R_1\prec R_3$, and $M_{2,4}(\bm\lambda)=M_{3,4}(\bm\lambda)=0$ so in fact these two vertices are equivalent in $\Pi(\bm\lambda)$, and we can swap the roles of these two vertices to reduce to the case where $l(S_{\varphi_i})>l(S_{\varphi_{i-1}})$. \\


\noindent\textbf{Case 2b: We have $l(R_{i-1})>\cdots>l(R_1)$, $R_t\prec R_n$ for $2\leq t\leq i-1$, and $l(R_i)<l(R_{i-1})$. }

Let $k\in\{\varphi_t: \ 2\leq t\leq i-1\}$ be such that $l(S_k)$ is maximal. By Proposition \ref{prop:ellgreat2}, we have $r(S_k)\geq r(R_{i-1})$ with equality only if $k=i-1$. If $R_{i-1}\nprec R_i$, then we must have $k=\varphi_{i-1}$ because $S_{\varphi_t}\prec S_{\varphi_i}$ for every $2\leq t\leq i-1$ except for $S_k$. Now by Proposition \ref{prop:ellgreat}, we have $l(S_{\varphi_n})\geq l(R_n)$ with equality only if $\varphi_n>\varphi_i>\varphi_{i-1}=k=i-1$. Because equality does indeed hold, we must have $\varphi_n=n$. On the other hand, if $R_{i-1}\prec R_i$, then $S_k\prec S_{\varphi_i}$, and by Corollary \ref{cor:mrirj}, Part 3, we have $r(S_{\varphi_i})\geq r(S_k)-1\geq r(R_{i-1})-1=r(R_i)$ with equality only if $\varphi_i>k=i-1$. Then by Proposition \ref{prop:ellgreat} again, we have $l(S_{\varphi_n})\geq l(R_n)$ with equality only if $\varphi_n>\varphi_i$. Because equality does indeed hold, we must have $\varphi_n=n$. \\

\noindent\textbf{Case 2c: We have $n=4$, $l(R_2)>l(R_1)$, $l(R_3)<l(R_2)$, and $R_4\precnsim R_2$. }

By Proposition \ref{prop:mncp3support}, we have $l(S_{\varphi_2})>l(S_1)$ and $l(S_{\varphi_4})>l(S_{\varphi_3})$, so by Proposition \ref{prop:ellgreat}, we have $l(S_{\varphi_4})\geq l(R_4)$ with equality only if $\varphi_4>\varphi_3>\varphi_2$. Because equality does indeed hold, we must have $\varphi_4=4$.\\

\noindent\textbf{Case 3: We have $R_1\precnsim R_i$ for some minimal $3\leq i\leq n-2$. }

We first show that $l(S_{\varphi_i})\geq l(R_i)$ with equality only if $\varphi_i>\varphi_t$ for all $t<i$. If $i=j=3$, $l(R_2)>l(R_1)$, $R_n\precnsim R_2$, and $l(R_3)<l(R_2)$, then we have $l(S_{\varphi_2})>l(S_1)$ by Proposition \ref{prop:mncp3support} and the result follows as before from Proposition \ref{prop:ellgreat}. Now suppose that $l(R_{i-1})<\cdots<l(R_1)$ and $l(R_i)>l(R_{i-1})$. We show that $l(S_{\varphi_i})>l(S_{\varphi_{i-1}})$. Suppose that $l(S_{\varphi_i})\leq l(S_{\varphi_{i-1}})$. By Corollary \ref{cor:mrirj}, Part 3, we have $M_{i-1,i}(\bm\lambda)<\min\{|R_{i-1}|,|R_i|\}$, so we must have $r(S_{\varphi_n})\leq r(S_{\varphi_i})+1\leq r(S_{\varphi_{i-1}})$. Because $M_{t,n}(\bm\lambda)=0$ for $2\leq t\leq i-1$ and $M_{t,t+1}(\bm\lambda)>0$ for $1\leq t\leq i-2$ by Corollary \ref{cor:mrirj}, Part 3,  we must have $r(S_{\varphi_{i-1}})<l(S_{\varphi_n})\leq r(S_{\varphi_n})$, a contradiction, so indeed $l(S_{\varphi_i})>l(S_{\varphi_{i-1}})$. Now by Proposition \ref{prop:ellgreat}, we have $l(S_{\varphi_i})\geq l(R_i)$ with equality only if $\varphi_i>\varphi_t$ for all $t<i$. Because either $j=i$, or $j=i+1$ and $l(R_j)<l(R_i)$, we in fact have $l(S_{\varphi_j})\geq l(R_j)$ with equality only if $\varphi_j>\varphi_t$ for all $t<j$. Finally, by rotating and repeating the previous argument, we have $l(S_{\varphi_n})\geq l(R_n)$ with equality only if $\varphi_n>\varphi_t$ for all $t<n$. Because equality does indeed hold, we must have $\varphi_n=n$. This completes the proof. 
\end{proof}

The payoff of all our work so far is the following Corollary. 

\begin{corollary} \label{cor:adjacentnoncommutingdone}
Let $\bm\lambda=(R_1,\ldots,R_n)$ be a horizontal-strip with $l(R_i)<l(R_{i+1})$ and $R_i\nleftrightarrow R_{i+1}$. Let $\bm\mu=(S_1,\ldots,S_n)$ and $\varphi:\Pi(\bm\lambda)\xrightarrow\sim\Pi(\bm\mu)$ be such that $S_{\varphi_i}\nleftrightarrow S_{\varphi_{i+1}}$. Then there exists a good substitute for $(\bm\lambda,\bm\mu)$. 
\end{corollary}

\begin{proof}
By cycling and rotating, we may assume without loss of generality that $i=1$, $l(S_{\varphi_1})<l(S_{\varphi_2})$, and $\varphi_1=1$. Let $j=\varphi_2$. By Lemma \ref{lem:aoncp}, either we may replace $\bm\mu$ by a similar horizontal-strip to assume that $j=2$, in which case we have our good substitute and we are done, or there is a minimal noncommuting path $(S_1=S_{j_1},\ldots,S_{j_k}=S_j)$ in $\bm\mu$ from $S_1$ to $S_j$. However, in this case, by considering the corresponding rows in $\bm\lambda$, we would have $2=\varphi_j^{-1}=\varphi_{j_k}^{-1}\geq k\geq 3$ by Lemma \ref{lem:mncp2}, a contradiction. Therefore, there is indeed a good substitute for $(\bm\lambda,\bm\mu)$.
\end{proof}

Our next goal is to extend Corollary \ref{cor:adjacentnoncommutingdone} by showing that certain properties of the weighted graph $\Pi(\bm\lambda)$ will force certain rows not to commute. 

\begin{definition} Let $\bm\lambda=(R_1,\ldots,R_n)$ be a horizontal-strip. A pair of rows $(R_i,R_j)$ of $\bm\lambda$ with $i<j$ and $l(R_i)<l(R_j)$ is \emph{strict} if either 
\begin{enumerate}
\item $0<M_{i,j}<\min\{|R_i|,|R_j|\}$, or 
\item $M_{i,j}=0$ and $M_{i,k}+M_{j,k}\geq |R_k|+1$ for some $k$. 
\end{enumerate}
\end{definition}

\begin{example}
The two possibilities for strictness are given below. Note that on the right, we have $M_{i,k}+M_{j,k}=2+3=|R_k|+1$. Informally, in the second possibility where $M_{i,j}=0$, the weighted graph normally would not know about the relationship between $R_i$ and $R_j$. However, the presence of this row $R_k$ glues the rows $R_i$ and $R_j$ together and means that the weighted graph data forces rows $R_i$ and $R_j$ to not commute.
$$
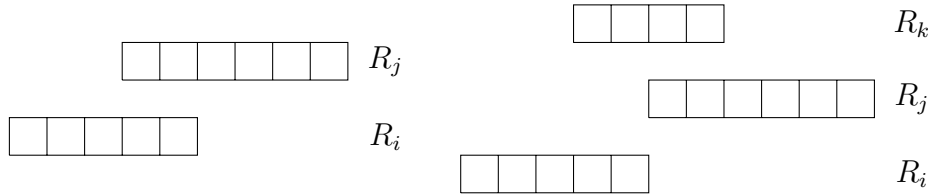
\begin{tikzpicture}
\draw (-0.5,-0.75) node (1) {$R_i$} (-0.5,0.25) node (2) {$R_j$};
\draw (-5.5,-1) -- (-3,-1) -- (-3,-0.5) -- (-5.5,-0.5) -- (-5.5,-1) (-5,-1) -- (-5,-0.5) (-4.5,-1) -- (-4.5,-0.5) (-4,-1) -- (-4,-0.5) (-3.5,-1) -- (-3.5,-0.5) (-4,0) -- (-1,0) -- (-1,0.5) -- (-4,0.5) -- (-4,0) (-3.5,0) -- (-3.5,0.5) (-3,0) -- (-3,0.5) (-2.5,0) -- (-2.5,0.5) (-2,0) -- (-2,0.5) (-1.5,0) -- (-1.5,0.5);
\draw (6.5,-1.25) node (1) {$R_i$} (6.5,-0.25) node (2) {$R_j$} (6.5,0.75) node (3) {$R_k$};
\draw (0.5,-1.5) -- (3,-1.5) -- (3,-1) -- (0.5,-1) -- (0.5,-1.5) (1,-1.5) -- (1,-1) (1.5,-1.5) -- (1.5,-1) (2,-1.5) -- (2,-1) (2.5,-1.5) -- (2.5,-1) (3,-0.5) -- (6,-0.5) -- (6,0) -- (3,0) -- (3,-0.5) (3.5,-0.5) -- (3.5,0) (4,-0.5) -- (4,0) (4.5,-0.5) -- (4.5,0) (5,-0.5) -- (5,0) (5.5,-0.5) -- (5.5,0) (2,0.5) -- (4,0.5) -- (4,1) -- (2,1) -- (2,0.5) (2.5,0.5) -- (2.5,1) (3,0.5) -- (3,1) (3.5,0.5) -- (3.5,1);
\end{tikzpicture}
$$
\end{example}

\begin{remark} Because we define strictness using the weighted graph data, it is preserved under isomorphisms. To be specific, if $\bm\lambda=(R_1,\ldots,R_n)$ and $\bm\mu=(S_1,\ldots,S_n)$ are horizontal-strips with $\varphi:\Pi(\bm\lambda)\xrightarrow\sim\Pi(\bm\mu)$, then if $(R_i,R_j)$ is strict, we can cycle and rotate to assume that $l(S_{\varphi_i})<l(S_{\varphi_j})$ and $\varphi_i<\varphi_j$, and then the pair $(S_{\varphi_i},S_{\varphi_j})$ is strict. 
\end{remark}

\begin{proposition}\label{prop:strict}
Let $\bm\lambda=(R_1,\ldots,R_n)$ be a horizontal-strip and suppose that the pair of rows $(R_i,R_j)$ is strict. Then $R_i\nleftrightarrow R_j$. 
\end{proposition}

\begin{proof}
If $0<M_{i,j}<\min\{|R_i|,|R_j|\}$, then $R_i\nleftrightarrow R_j$ by Corollary \ref{cor:mrirj}, Part 1, so suppose that $M_{i,j}=0$ and $M_{i,k}+M_{j,k}\geq |R_k|+1$ for some $k$. Because $l(R_i)<l(R_j)$ and $M_{i,j}=0$, we have that $l(R_i)\geq r(R_j)+1$ and because $M_{i,k},M_{j,k}\leq |R_k|$, we must have $M_{i,k},M_{j,k}>0$ so $l(R_k)\leq r(R_i)+1\leq l(R_j)\leq r(R_k)+1$ by Proposition \ref{prop:mrirj}, Part 1. Now we either have $M_{i,k}=r(R_i)-l(R_k)+1+\chi(i>k)$ or $M_{i,k}=\min\{|R_i|,|R_k|\}\leq r(R_i)-l(R_k)+1+\chi(i>k)$ and similarly we have $M_{j,k}\leq r(R_k)-l(R_j)+1+\chi(k>j)$. 
Because $i<j$, we have $\chi(i>k)+\chi(k>j)\leq 1$, so \begin{equation}|R_k|+1\leq M_{i,k}+M_{j,k}\leq |R_k|+r(R_i)-l(R_j)+1+\chi(i>k)+\chi(k>j)\leq |R_k|+1,\end{equation}
so we must have equality everywhere and in particular, $r(R_j)-l(R_i)+1=0$, so $R_i\nleftrightarrow R_j$. 
\end{proof}

\begin{remark} This proof shows that if $l(R_i)<l(R_j)$ and $M_{i,j}=0$, then in fact $M_{i,k}+M_{j,k}\leq |R_k|+1$ for all $k$, so we could replace the condition $M_{i,k}+M_{j,k}\geq |R_k|+1$ with the equivalent condition $M_{i,k}+M_{j,k}=|R_k|+1$.  \end{remark}

\begin{corollary} \label{cor:adjacentstrictdone}
Let $\bm\lambda=(R_1,\ldots,R_n)$ be a horizontal-strip with a pair of adjacent strict rows $(R_i,R_{i+1})$. Then $\bm\lambda$ is good.
\end{corollary}

\begin{proof}
Let $\bm\mu=(S_1,\ldots,S_n)$ and $\varphi:\Pi(\bm\lambda)\xrightarrow\sim\Pi(\bm\mu)$. By definition, we have $M_{i,i+1}(\bm\lambda)<\min\{|R_i|,|R_{i+1}|\}$, so $M_{\varphi_i,\varphi_{i+1}}(\bm\mu)<\min\{|S_{\varphi_i}|,|S_{\varphi_{i+1}}|\}$ and $l(S_{\varphi_i})\neq l(S_{\varphi_{i+1}})$. Therefore, by cycling and rotating, we may assume without loss of generality that $l(S_{\varphi_i})<l(S_{\varphi_{i+1}})$ and $\varphi_i<\varphi_{i+1}$, so the pair $(S_{\varphi_i},S_{\varphi_{i+1}})$ is strict. By Proposition \ref{prop:strict}, we have $S_{\varphi_i}\nleftrightarrow S_{\varphi_{i+1}}$, so by Corollary \ref{cor:adjacentnoncommutingdone}, there exists a good substitute for $(\bm\lambda,\bm\mu)$. 
\end{proof}

We now investigate some useful properties of strict pairs.

\begin{proposition} \label{prop:strict2}
Let $\bm\lambda=(R_1,\ldots,R_n)$ be a horizontal-strip with $i<j$, $l(R_i)<l(R_j)$, and $R_i\nleftrightarrow R_j$. Suppose that $R_i\nleftrightarrow R_k$ and $R_j\nleftrightarrow R_k$ for some $k$ with $k<i$ or $k>j$. Then the pair $(R_i,R_j)$ is strict. 
\end{proposition}

\begin{proof}
By Corollary \ref{cor:mrirj}, Part 3, we cannot have $M_{i,j}=\min\{|R_i|,|R_j|\}$, and if $0<M_{i,j}<\min\{|R_i|,|R_j|\}$, then we are done, so suppose that $M_{i,j}=0$ and therefore $l(R_j)=r(R_i)+1$ by Corollary \ref{cor:mrirj}, Part 3. Because $R_i\nleftrightarrow R_k$, by Proposition \ref{prop:mrirj}, Part 1 we must have $l(R_k)\leq r(R_i)+1=l(R_j)$, and because $R_j\nleftrightarrow R_k$, we must have $r(R_k)\geq l(R_j)-1=r(R_i)$, so by Corollary \ref{cor:mrirj}, Part 2, we in fact have $l(R_i)<l(R_k)<l(R_j)$. Now by \eqref{eq:mij}, we have \begin{equation}M_{i,k}+M_{j,k}=r(R_i)-l(R_k)+1+\chi(i>k)+r(R_k)-l(R_j)+1+\chi(k>j)=|R_k|+1,\end{equation}
so the pair $(R_i,R_j)$ is strict.
\end{proof}

\begin{proposition}\label{prop:strict3}
Let $\bm\lambda=(R_1,\ldots,R_n)$ be a minimal noncommuting path and suppose that the pairs $(R_{t'},R_{t'+1})$ are not strict for $1\leq t'<n-1$. Then we have the following.
\begin{enumerate}
\item If $l(R_{t+1})>l(R_t)$, then $M_{t,t+1}=0$ and $l(R_{t+1})=r(R_t)+1$.\\

\item If $l(R_{t+1})>l(R_t)$ and $l(R_t)<l(R_{t-1})$, then $R_{t+1}\precnsim R_{t-1}$. \\

\item If $l(R_{t+1})<l(R_t)$ and $l(R_t)>l(R_{t-1})$, then $R_{t-1}\precnsim R_{t+1}$.
\end{enumerate}
\end{proposition}

\begin{proof}\hspace{2pt}
\begin{enumerate}
\item By Corollary \ref{cor:mrirj}, Part 3 we cannot have $M_{t,t+1}=\min\{|R_t|,|R_{t+1}|\}$, and if $0<M_{t,t+1}<\min\{|R_t|,|R_{t+1}|\}$, then the pair $(R_t,R_{t+1})$ would be strict, so we must have $M_{t,t+1}=0$ and $l(R_{t+1})=r(R_t)+1$.\\

\item By the previous part, we must have $M_{t,t+1}=0$. By Proposition \ref{prop:prec2}, Part 2, we must have $R_{t-1}\prec R_{t+1}$ or $R_{t+1}\prec R_{t-1}$. However, if $R_{t-1}\prec R_{t+1}$, then because $M_{t-1,t}>0$ by Corollary \ref{cor:mrirj}, Part 3, we would have $M_{t-1,t}+M_{t-1,t+1}\geq |R_{t-1}|+1$ and the pair $(R_t,R_{t+1})$ would be strict, so we must have $R_{t+1}\precnsim R_{t-1}$.\\

\item By rotating, this follows from the previous part.
\end{enumerate}
\end{proof}

\begin{example} The diagrams below illustrate the contradictions that we deduce in the proofs of Parts 2 and 3 of Proposition \ref{prop:strict3}. If $R_{t-1}\prec R_{t+1}$ as on the left, then the pair $(R_t,R_{t+1})$ would be strict, so we concluded that $R_{t+1}\precnsim R_{t-1}$. If $R_{t+1}\prec R_{t-1}$ as on the right, then the pair $(R_{t-1},R_t)$ would be strict, so we concluded that $R_{t-1}\precnsim R_{t+1}$. 
$$
\begin{tikzpicture}
\draw (6.5,-1.25) node (1) {$R_{t-1}$} (6.5,-0.25) node (2) {$R_t$} (6.5,0.75) node (3) {$R_{t+1}$};
\draw (3,-1.5) -- (4.5,-1.5) -- (4.5,-1) -- (3,-1) -- (3,-1.5) (3.5,-1.5) -- (3.5,-1) (4,-1.5) -- (4,-1) (0.5,-0.5) -- (3,-0.5) -- (3,0) -- (0.5,0) -- (0.5,-0.5) (1,-0.5) -- (1,0) (1.5,-0.5) -- (1.5,0) (2,-0.5) -- (2,0) (2.5,-0.5) -- (2.5,0) (3,0.5) -- (6,0.5) -- (6,1) -- (3,1) -- (3,0.5) (3.5,0.5) -- (3.5,1) (4,0.5) -- (4,1) (4.5,0.5) -- (4.5,1) (5,0.5) -- (5,1) (5.5,0.5) -- (5.5,1);
\end{tikzpicture}\hspace{50pt}
\begin{tikzpicture}
\draw (6.5,-1.25) node (1) {$R_{t-1}$} (6.5,-0.25) node (2) {$R_t$} (6.5,0.75) node (3) {$R_{t+1}$};
\draw (0.5,-1.5) -- (3,-1.5) -- (3,-1) -- (0.5,-1) -- (0.5,-1.5) (1,-1.5) -- (1,-1) (1.5,-1.5) -- (1.5,-1) (2,-1.5) -- (2,-1) (2.5,-1.5) -- (2.5,-1) (3,-0.5) -- (6,-0.5) -- (6,0) -- (3,0) -- (3,-0.5) (3.5,-0.5) -- (3.5,0) (4,-0.5) -- (4,0) (4.5,-0.5) -- (4.5,0) (5,-0.5) -- (5,0) (5.5,-0.5) -- (5.5,0) (1.5,0.5) -- (3,0.5) -- (3,1) -- (1.5,1) -- (1.5,0.5) (2,0.5) -- (2,1) (2.5,0.5) -- (2.5,1);
\end{tikzpicture}
$$
\end{example}

It will also be convenient to make the following definition.

\begin{definition}
Let $\bm\lambda=(R_1,\ldots,R_n)$ be a horizontal-strip. A \emph{strict sequence} of $\bm\lambda$ is a sequence of rows $(R_{j_1},\ldots,R_{j_k})$ such that $k\geq 2$, $j_1<\cdots<j_k$, $M_{j_t,j_{t'}}=0$ for every $1\leq t<t'\leq k$, and there is some $h$ with $h<j_1$ or $h>j_k$ for which \begin{equation} M_{j_t,h}>0\text{ for all }1\leq t\leq k\text{ and }M_{j_1,h}+\cdots+M_{j_k,h}\geq|R_h|+1.\end{equation}
\end{definition}

Note that if a pair of rows $(R_i,R_j)$ is a strict sequence, then it meets the second condition of being a strict pair. 

\begin{example} \label{ex:strictsequence}
In Figure \ref{fig:strictsequence}, we have $M_{t,t'}=0$ for every $1\leq t<t'\leq 6$, $M_{t,7}>0$ for $1\leq t\leq 6$, and $M_{1,7}+\cdots+M_{6,7}=2+5+2+4+3+3=19=|R_7|+1$, so $(R_1,\ldots,R_6)$ is a strict sequence. Informally, our next Proposition will show that every strict sequence looks very much like this example. Because $M_{t,t'}=0$ for $1\leq t<t'\leq 6$, the weighted graph normally would not know about the relationship between these rows. However, the presence of the row $R_7$ glues these rows together and means that the weighted graph data forces the adjacent rows not to commute. 

\begin{figure}\caption{An example of a strict sequence}\label{fig:strictsequence}
$$\begin{tikzpicture}
\draw (-0.25,-0.25) node (0) {$\bm \lambda=$};
\draw (11,-3.25) node (1) {$R_1$} (11,-2.25) node (2) {$R_2$} (11,-1.25) node (3) {$R_3$} (11,-0.25) node (4) {$R_4$} (11,0.75) node (5) {$R_5$} (11,1.75) node (6) {$R_6$} (11,2.75) node (7) {$R_7$};
\draw (0,-3.5) -- (1.5,-3.5) -- (1.5,-3) -- (0,-3) -- (0,-3.5) (0.5,-3.5) -- (0.5,-3) (1,-3.5) -- (1,-3) (1.5,-2.5) -- (4,-2.5) -- (4,-2) -- (1.5,-2) -- (1.5,-2.5) (2,-2.5) -- (2,-2) (2.5,-2.5) -- (2.5,-2) (3,-2.5) -- (3,-2) (3.5,-2.5) -- (3.5,-2) (4,-1.5) -- (5,-1.5) -- (5,-1) -- (4,-1) -- (4,-1.5) (4.5,-1.5) -- (4.5,-1) (5,-0.5) -- (7,-0.5) -- (7,0) -- (5,0) -- (5,-0.5) (5.5,-0.5) -- (5.5,0) (6,-0.5) -- (6,0) (6.5,-0.5) -- (6.5,0) (7,0.5) -- (8.5,0.5) -- (8.5,1) -- (7,1) -- (7,0.5) (7.5,0.5) -- (7.5,1) (8,0.5) -- (8,1) (8.5,1.5) -- (10.5,1.5) -- (10.5,2) -- (8.5,2) -- (8.5,1.5) (9,1.5) -- (9,2) (9.5,1.5) -- (9.5,2) (10,1.5) -- (10,2) (0.5,2.5) -- (9.5,2.5) -- (9.5,3) -- (0.5,3) -- (0.5,2.5) (1,2.5) -- (1,3) (1.5,2.5) -- (1.5,3) (2,2.5) -- (2,3) (2.5,2.5) -- (2.5,3) (3,2.5) -- (3,3) (3.5,2.5) -- (3.5,3) (4,2.5) -- (4,3) (4.5,2.5) -- (4.5,3) (5,2.5) -- (5,3) (5.5,2.5) -- (5.5,3) (6,2.5) -- (6,3) (6.5,2.5) -- (6.5,3) (7,2.5) -- (7,3) (7.5,2.5) -- (7.5,3) (8,2.5) -- (8,3) (8.5,2.5) -- (8.5,3) (9,2.5) -- (9,3);
\end{tikzpicture}
$$
\end{figure}
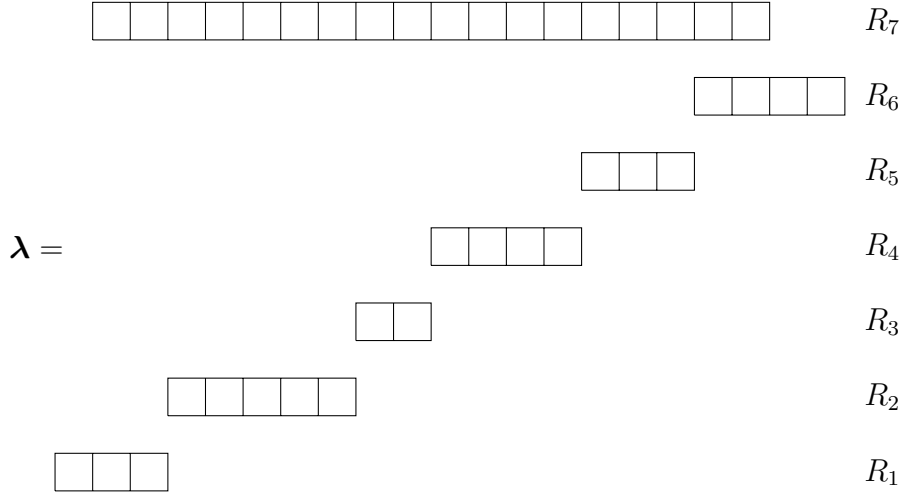
\end{example}

\begin{remark}\label{rem:strictsequencetransfer} Because we define a strict sequence using the weighted graph data, it is preserved under isomorphisms. To be specific, if $\bm\lambda=(R_1,\ldots,R_n)$ and $\bm\mu=(S_1,\ldots,S_n)$ are horizontal-strips with $\varphi:\Pi(\bm\lambda)\xrightarrow\sim\Pi(\bm\mu)$, then if $(R_{j_1},\ldots,R_{j_k})$ is a strict sequence, we can cycle to assume that $\varphi_h>\varphi_{j_t}$ for every $1\leq t\leq k$. Because $M_{j_t,j_{t'}}(\bm\lambda)=0$ for $1\leq t<t'\leq k$, the integers $l(S_{\varphi_{j_t}})$ for $1\leq t\leq k$ must be distinct, so we can let $\sigma:\{1,\ldots,k\}\to\{1,\ldots,k\}$ be the permutation that sorts them in increasing order, in other words \begin{equation}l(S_{\varphi_{j_{\sigma_1}}})<l(S_{\varphi_{j_{\sigma_2}}})<\cdots<l(S_{\varphi_{j_{\sigma_k}}}).\end{equation}
Then the sequence $(S_{\varphi_{j_{\sigma_1}}},\ldots,S_{\varphi_{j_{\sigma_k}}})$ is strict. 
\end{remark}

\begin{proposition} \label{prop:strictsequence}
Let $\bm\lambda=(R_1,\ldots,R_n)$ be a horizontal-strip with a sequence of rows $(R_{j_1},\ldots,R_{j_k})$ with $k\geq 2$, $j_1<\cdots<j_k$, $M_{j_t,j_{t'}}=0$ for every $1\leq t<t'\leq k$, and there is some $h$ with $h<j_1$ or $h>j_k$ for which $M_{j_t,h}>0$ for every $1\leq t\leq k$. Then this sequence is strict if and only if $l(R_{j_{t+1}})=r(R_{j_t})+1$ for every $1\leq t\leq k-1$ and \begin{equation}\label{eq:strictsequencecondition}
l(R_{j_1})+\chi(j_1>h)\leq l(R_h)\leq r(R_h)+\chi(h>j_k)\leq r(R_{j_k}).
\end{equation}
\end{proposition}

\begin{proof}
By Proposition \ref{prop:mrirj} Part 2, because all of the $M_{j_t,j_{t'}}$ are zero, the integers $l(R_{j_t})$ for $1\leq t\leq k$ are distinct, so let $\sigma:\{1,\ldots,k\}\to\{j_1,\ldots,j_k\}$ sort the rows $R_{j_t}$ so that $l(R_{j_t})$ is increasing. Then because the $M_{j_t,j_{t'}}$ are zero we have \begin{equation}
l(R_{\sigma_1})<r(R_{\sigma_1})+1+\chi(\sigma_1>\sigma_2)\leq l(R_{\sigma_2})<r(R_{\sigma_2})+1+\chi(\sigma_2>\sigma_3)\leq\cdots\leq l(R_{\sigma_k}).\end{equation}
Because $M_{\sigma_1,h}>0$, we must have $l(R_h)\leq r(R_{\sigma_1})+1\leq l(R_{\sigma_2})$ and because $M_{\sigma_k,h}>0$, we must have $r(R_h)\geq l(R_{\sigma_k})-1\geq r(R_{\sigma_{k-1}})$. We now have that \begin{align}
M_{\sigma_1,h}&\leq r(R_{\sigma_1})-l(R_h)+1+\chi(\sigma_1>h)\leq l(R_{\sigma_2})-l(R_h)+\chi(\sigma_1>h)-\chi(\sigma_1>\sigma_2)\\\nonumber
M_{\sigma_2,h}&=|R_{\sigma_2}|=r(R_{\sigma_2})-l(R_{\sigma_2})+1\leq l(R_{\sigma_3})-l(R_{\sigma_2})-\chi(\sigma_2>\sigma_3)\\\nonumber\cdots&=\\\nonumber M_{\sigma_{k-1},h}&=|R_{\sigma_{k-1}}|=r(R_{\sigma_{k-1}})-l(R_{\sigma_{k-1}})+1\leq l(R_{\sigma_k})-l(R_{\sigma_{k-1}})-\chi(\sigma_{k-1}>\sigma_k)\\\nonumber 
M_{\sigma_k,h}&\leq r(R_h)-l(R_{\sigma_k})+1+\chi(h>\sigma_k).
\end{align}
Also note that $\chi(\sigma_1>h)-\chi(\sigma_1>\sigma_2)-\cdots-\chi(\sigma_{k-1}>\sigma_k)+\chi(h>\sigma_k)\leq 1$ with equality only if $\sigma_1<\cdots<\sigma_k$. Therefore, by summing the above equations, we have \begin{equation}\sum_{t=1}^kM_{\sigma_t,h}\leq r(R_h)-l(R_h)+1+1=|R_h|+1,
\end{equation} so the sequence is strict if and only if we have equality everywhere, meaning that $\sigma_1<\cdots<\sigma_k$, $l(R_{j_{t+1}})=r(R_{j_t})+1$ for $1\leq t\leq k-1$, and \eqref{eq:strictsequencecondition} holds. \end{proof}

\begin{remark}
This proof shows that if $l(R_{j_1})<\cdots<l(R_{j_k})$ and the $M_{j_t,j_{t'}}=0$, then in fact $M_{j_1,h}+\cdots+M_{j_k,h}\leq |R_h|+1$ for all $h$, so we could replace the condition $M_{j_1,h}+\cdots+M_{j_k,h}\geq |R_h|+1$ with the equivalent condition $M_{j_1,h}+\cdots+M_{j_k,h}=|R_h|+1$.  
\end{remark}

\begin{proposition} \label{prop:strictsequence2} Let $\bm\lambda=(R_1,\ldots,R_n)$ be a horizontal-strip with a strict sequence of rows $(R_{j_1},\ldots,R_{j_k})$. Suppose that there is some $j_t<x<j_{t+1}$ with $l(R_x)=l(R_{j_{t+1}})$ and $R_{j_{t+1}}\precnsim R_x$. Then one of the following holds. \begin{enumerate}
\item The sequence $(R_{j_1},\ldots,R_{j_t},R_x)$ is strict.\\

\item There is a shorter strict sequence of the form $(R_{j_1},\ldots,R_{j_t},R_x,R_{j_{t'}},\ldots,R_{j_k})$ for some $t'\geq t+2$.\\

\item There is a strict pair $(R_x,R_{j_{t'+1}})$ for some $t'\geq t+1$. 
\end{enumerate}
\end{proposition}

\begin{example}
Informally, Proposition \ref{prop:strictsequence2} describes a situation like the one in Figure \ref{fig:strictsequence2}. The sequence $(R_{j_1},\ldots,R_{j_6})$ is the strict sequence from Example \ref{ex:strictsequence}. The rows $R_{x_1}$, $R_{x_2}$, and $R_{x_3}$ illustrate the three possibilities described in Proposition \ref{prop:strictsequence2}. Informally, because $R_{x_1}$ extends past $R_{j_t}$, the sequence $(R_{j_1},R_{j_2},R_{x_1})$ is strict. Because $R_{x_2}=R_{j_3}\cup R_{j_4}\cup R_{j_5}$, it can replace these rows to produce the shorter strict sequence $(R_{j_1},R_{j_2},R_{x_2},R_{j_6})$. Finally, because $R_{x_3}$ ends between the rows $R_{j_4}$ and $R_{j_5}$, it results in the strict pair $(R_{x_3},R_{j_5})$. 

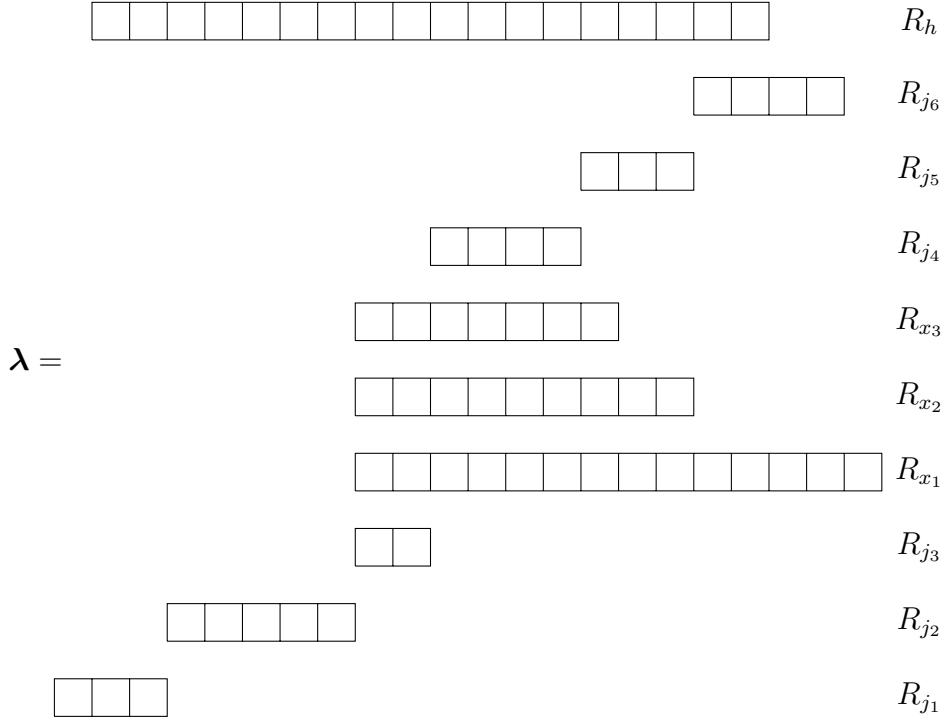
\begin{figure}\caption{A strict sequence with a row $R_x$ as in Proposition \ref{prop:strictsequence2} }\label{fig:strictsequence2}
$$\begin{tikzpicture}
\draw (-0.25,-1.75) node (0) {$\bm \lambda=$};
\draw (11.5,-6.25) node (1) {$R_{j_1}$} (11.5,-5.25) node (2) {$R_{j_2}$} (11.5,-4.25) node (3) {$R_{j_3}$} (11.5,-3.25) node (4) {$R_{x_1}$} (11.5,-2.25) node (5) {$R_{x_2}$} (11.5,-1.25) node (6) {$R_{x_3}$} (11.5,-0.25) node (7) {$R_{j_4}$} (11.5,0.75) node (8) {$R_{j_5}$} (11.5,1.75) node (9) {$R_{j_6}$} (11.5,2.75) node (10) {$R_h$};
\draw (0,-6.5) -- (1.5,-6.5) -- (1.5,-6) -- (0,-6) -- (0,-6.5) (0.5,-6.5) -- (0.5,-6) (1,-6.5) -- (1,-6) (1.5,-5.5) -- (4,-5.5) -- (4,-5) -- (1.5,-5) -- (1.5,-5.5) (2,-5.5) -- (2,-5) (2.5,-5.5) -- (2.5,-5) (3,-5.5) -- (3,-5) (3.5,-5.5) -- (3.5,-5) (4,-4.5) -- (5,-4.5) -- (5,-4) -- (4,-4) -- (4,-4.5) (4.5,-4.5) -- (4.5,-4) (4,-1.5) -- (7.5,-1.5) -- (7.5,-1) -- (4,-1) -- (4,-1.5) (4.5,-1.5) -- (4.5,-1) (5,-1.5) -- (5,-1) (5.5,-1.5) -- (5.5,-1) (6,-1.5) -- (6,-1) (6.5,-1.5) -- (6.5,-1) (7,-1.5) -- (7,-1) (4,-2.5) -- (8.5,-2.5) -- (8.5,-2) -- (4,-2) -- (4,-2.5) (4.5,-2.5) -- (4.5,-2) (5,-2.5) -- (5,-2) (5.5,-2.5) -- (5.5,-2) (6,-2.5) -- (6,-2) (6.5,-2.5) -- (6.5,-2) (7,-2.5) -- (7,-2) (7.5,-2.5) -- (7.5,-2) (8,-2.5) -- (8,-2) (4,-3.5) -- (11,-3.5) -- (11,-3) -- (4,-3) -- (4,-3.5) (4.5,-3.5) -- (4.5,-3) (5,-3.5) -- (5,-3) (5.5,-3.5) -- (5.5,-3) (6,-3.5) -- (6,-3) (6.5,-3.5) -- (6.5,-3) (7,-3.5) -- (7,-3) (7.5,-3.5) -- (7.5,-3) (8,-3.5) -- (8,-3) (8.5,-3.5) -- (8.5,-3) (9,-3.5) -- (9,-3) (9.5,-3.5) -- (9.5,-3) (10,-3.5) -- (10,-3) (10.5,-3.5) -- (10.5,-3) (5,-0.5) -- (7,-0.5) -- (7,0) -- (5,0) -- (5,-0.5) (5.5,-0.5) -- (5.5,0) (6,-0.5) -- (6,0) (6.5,-0.5) -- (6.5,0) (7,0.5) -- (8.5,0.5) -- (8.5,1) -- (7,1) -- (7,0.5) (7.5,0.5) -- (7.5,1) (8,0.5) -- (8,1) (8.5,1.5) -- (10.5,1.5) -- (10.5,2) -- (8.5,2) -- (8.5,1.5) (9,1.5) -- (9,2) (9.5,1.5) -- (9.5,2) (10,1.5) -- (10,2) (0.5,2.5) -- (9.5,2.5) -- (9.5,3) -- (0.5,3) -- (0.5,2.5) (1,2.5) -- (1,3) (1.5,2.5) -- (1.5,3) (2,2.5) -- (2,3) (2.5,2.5) -- (2.5,3) (3,2.5) -- (3,3) (3.5,2.5) -- (3.5,3) (4,2.5) -- (4,3) (4.5,2.5) -- (4.5,3) (5,2.5) -- (5,3) (5.5,2.5) -- (5.5,3) (6,2.5) -- (6,3) (6.5,2.5) -- (6.5,3) (7,2.5) -- (7,3) (7.5,2.5) -- (7.5,3) (8,2.5) -- (8,3) (8.5,2.5) -- (8.5,3) (9,2.5) -- (9,3);
\end{tikzpicture}$$
\end{figure}
\end{example}

\begin{proof}[Proof of Proposition \ref{prop:strictsequence2}. ]
By the definition of a strict sequence, there is some $h$ with $h<j_1$ or $h>j_k$ for which $M_{j_t,h}>0$ for all $1\leq t\leq k$ and $\sum_{t=1}^kM_{j_t,h}\geq |R_h|+1$, and by cycling we may assume without loss of generality that $h>j_k$, so that $l(R_{j_1})\leq l(R_h)<r(R_h)+1\leq r(R_{j_k})$ by \eqref{eq:strictsequencecondition}. Noting that $R_{j_{t+1}}\precnsim R_x$, let $t+1\leq t'\leq k$ be maximal such that $R_{j_{t'}}\prec R_x$. If $t'=k$, then we must have $r(R_x)\geq r(R_{j_k})\geq r(R_h)+1$ by Proposition \ref{prop:prec} and now the sequence $(R_{j_1},\ldots,R_{j_t},R_x)$ is strict by Proposition \ref{prop:strictsequence}, so the first possibility holds and we may now assume that $t'\leq k-1$. By maximality of $t'$, we have that $R_{j_{t'+1}}\nprec R_x$. If $M_{x,j_{t'+1}}=0$, then we must have $r(R_x)=r(R_{j_{t'}})=l(R_{j_{t'+1}})-1$ and the sequence $(R_{j_1},\ldots,R_{j_t},R_x,R_{j_{t'+1}},\ldots,R_{j_k})$ is strict by Proposition \ref{prop:strictsequence}. Also note that because $R_x\nprec R_{j_{t+1}}$, we must have $t'>t+1$ so this strict sequence is indeed shorter and the second possibility holds. Finally, if $M_{x,j_{t'+1}}>0$, then because $l(R_x)=l(R_{j_{t+1}})<l(R_{j_{t'+1}})$, we have $R_x\nprec R_{j_{t'+1}}$, and by maximality of $t$ we have $R_{j_{t'+1}}\nprec R_x$, so $0<M_{x,j_{t'+1}}<\min\{|R_x|,|R_{j_{t'+1}}|\}$, the pair $(R_x,R_{j_{t'+1}})$ is strict and the third possibility holds.
\end{proof}

We now describe the structure of a minimal noncommuting path with no strict pairs.

\begin{proposition}
\label{prop:mncpnostrict}
Let $\bm\lambda=(R_1,\ldots,R_n)$ be a minimal noncommuting path with $l(R_1)<l(R_n)$ and $R_1\nleftrightarrow R_n$ and suppose that the pairs $(R_t,R_{t+1})$ are not strict for $1\leq t\leq n-1$. Then one of the following holds.
\begin{enumerate}
\item One of the sequences $(R_1,\ldots,R_{n-1})$, $(R_2,\ldots,R_{n-1})$, or $(R_2,\ldots,R_n)$ is strict.\\

\item We have $n=4$, $l(R_2)=l(R_4)=r(R_1)-1=r(R_3)-1$, $R_4\precnsim R_2$, and $R_1\precnsim R_3$.
\end{enumerate}
\end{proposition}

\begin{example}
Informally, Proposition \ref{prop:mncpnostrict} tells us that a minimal noncommuting path with no strict pairs (other than possibly $(R_1,R_n)$) must look like one of the examples below. 
$$\begin{tabular}{c|c|c|c}
\tableau{&& \ & \ & \ & \ \\ \\ &&&&& \ & \ \\ \\ &&& \ & \ \\ \\ \ & \ & \ } & \tableau{&& \ & \ & \ & \ \\ \\ &&&&& \ & \ \\ \\ && \ & \ & \ \\ \\ \ & \ } & \tableau{&&&& \ & \ & \ \\ \\ && \ & \ \\ \\ \ & \ \\ \\ & \ & \ & \ & \ } & \tableau{&&& \ & \ & \ \\ \\ \ & \ & \ \\ \\ &&& \ & \ & \ & \ \\ \\ & \ & \ }\end{tabular}$$
\end{example}

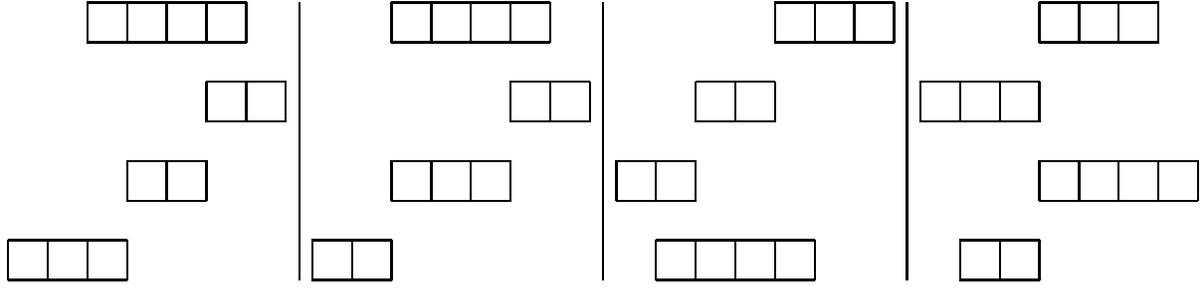
\begin{proof}[Proof of Proposition \ref{prop:mncpnostrict}. ]
We consider several cases.\\

\noindent\textbf{Case 0: We have $n=3$. }

In this case, we have $R_1\nleftrightarrow R_2$ and $R_2\nleftrightarrow R_3$. Because $l(R_1)<l(R_3)$, we must have either $l(R_1)<l(R_2)$, in which case the pair $(R_1,R_2)$ is strict by Proposition \ref{prop:strict2}, or $l(R_2)<l(R_3)$, in which case the pair $(R_2,R_3)$ is strict by Proposition \ref{prop:strict2}, a contradiction. \\

We may now assume that $n\geq 4$, so it remains to consider the several possibilities outlined in Lemma \ref{lem:mncp1}.\\

\noindent\textbf{Case 1: There is no $i\geq 3$ for which $R_1\precnsim R_i$ and there is no $j\leq n-2$ for which $R_n\precnsim R_j$. }\\

\noindent\textbf{Case 1a: We have $l(R_2)>l(R_1)$, $l(R_{n-1})<\cdots<l(R_2)$, and $l(R_n)>l(R_{n-1})$. }

Because $l(R_2)>l(R_1)$ and $l(R_3)<l(R_2)$, we have $R_1\precnsim R_3$ by Proposition \ref{prop:strict3}, contradicting our hypothesis.\\

\noindent\textbf{Case 1b: We have $l(R_{n-1})>\cdots>l(R_1)$, $R_t\prec R_n$ for $2\leq t\leq n-2$, and $l(R_n)<l(R_{n-1})$. }

By Proposition \ref{prop:strict3}, we must have $l(R_{t+1})=r(R_t)+1$ for $1\leq t\leq n-2$. Now by Proposition \ref{prop:strictsequence}, if $M_{1,n}>0$ then the sequence $(R_1,\ldots,R_{n-1})$ is strict and if $M_{1,n}=0$ then the sequence $(R_2,\ldots,R_{n-1})$ is strict, so the first possibility holds.\\

\noindent\textbf{Case 1c: We have $l(R_2)<l(R_1)$, $l(R_n)>\cdots>l(R_2)$, and $R_t\prec R_1$ for $3\leq t\leq n-1$. }

By rotating, the conclusion follows from Case 1b.\\

We now assume that $R_1\precnsim R_i$ for some minimal $i\geq 3$ or $R_n\precnsim R_j$ for some maximal $j\leq n-2$. By rotating, we may assume that $R_1\precnsim R_i$ for some minimal $i\geq 3$. Note that by Corollary \ref{cor:mrirj}, Part 3, we cannot have $R_1\prec R_n$, so we have $i\leq n-1$. It remains to consider the cases where $i=n-1$ and where $i\leq n-2$. \\

\noindent\textbf{Case 2: We have $R_1\precnsim R_i$ for some minimal $i\geq 3$, and in fact $i=n-1$. }\\

\noindent\textbf{Case 2a: We have $l(R_{i-1})<\cdots<l(R_1)$ and $l(R_i)>l(R_{i-1})$. }

By Proposition \ref{prop:mncp}, Part 2, we have $R_{i-2}\prec R_i$, but because $l(R_{i-1})<l(R_{i-2})$ and $l(R_i)>l(R_{i-1})$, we have $R_i\precnsim R_{i-2}$ by Proposition \ref{prop:strict3}, a contradiction.\\

\noindent\textbf{Case 2b: We have $l(R_{i-1})>\cdots>l(R_1)$, $R_t\prec R_n$ for $2\leq t\leq i-1$, and $l(R_i)<l(R_{i-1})$. }

Because $l(R_i)<l(R_{i-1})$ and $l(R_n)>l(R_i)$, by Proposition \ref{prop:strict3}, Part 2 we have $R_n\precnsim R_{i-1}$, contradicting our hypothesis.\\

\noindent\textbf{Case 2c: We have $n=4$, $l(R_2)>l(R_1)$, $l(R_3)<l(R_2)$, and $R_4\precnsim R_2$. }

By Proposition \ref{prop:prec}, Proposition \ref{prop:mrirj}, and Proposition \ref{prop:strict3}, Part 1, we must have $r(R_1)+1=l(R_2)\leq l(R_4)\leq r(R_1)+1$, so $l(R_2)=l(R_4)$, and we must have $l(R_4)-1=r(R_3)\geq r(R_1)=l(R_2)-1=l(R_4)-1$, so $r(R_3)=r(R_1)$. In particular, we have $M_{1,n}=0$. By Proposition \ref{prop:strict3}, Parts 2 and 3, we must have $R_4\precnsim R_2$ and $R_1\precnsim R_3$ and the second possibility holds. \\

\noindent\textbf{Case 3: We have $R_i\precnsim R_i$ for some minimal $3\leq i\leq n-2$. }

If $l(R_{i-1})<l(R_{i-2})$ and $l(R_i)>l(R_{i-1})$, then by Proposition \ref{prop:mncp}, Part 2, we have $R_{i-2}\prec R_i$, but by Proposition \ref{prop:strict3}, Part 2 we have $R_i\precnsim R_{i-2}$, a contradiction. Therefore, by Lemma \ref{lem:mncp1}, we must have $i=j=3$, $l(R_2)>l(R_1)$, $R_n\precnsim R_2$, and $l(R_3)<l(R_2)$. Similarly, by rotating, we must have $i=j=n-2$, so $n=5$, $l(R_5)>l(R_4)$, $R_1\precnsim R_4$, and $l(R_4)<l(R_3)$. However, by Proposition \ref{prop:mncp}, Parts 1 and 2, we have $R_2\prec R_4$ and $R_4\prec R_2$, which is impossible because $n\geq 5$. 
\end{proof}

We now describe another operation that we can perform on a horizontal-strip while preserving similarity. We can think of it as a \emph{local rotation}.

\begin{lemma} \label{lem:localrotation} Let $\bm\lambda=(R_1,\ldots,R_n)$ and suppose that $l(R_i)=r(R_{i-1})+1$ for some $2\leq i\leq n$. Assume that $\bm\lambda$ satisfies the inductive hypothesis \eqref{eq:IH} in Lemma \ref{lem:key}. Let $I=\{1,\ldots,i-2,i+1,\ldots,n\}$ and define the four disjoint subsets of $I$
\begin{align*}
A&=\{t\in I: \ M_{i-1,t}=M_{i,t}=0\},\\
B&=\{t\in I: \ R_{i-1}\prec R_t, \ R_i\prec R_t\},\\
C_{i-1}&=\{t\in I: \ R_t\prec R_{i-1}, \ M_{i,t}=0, \ M_{a,t}=0\text{ for all }a\in A, \ R_t\prec R_b\text{ for all }b\in B\},\\
C_i&=\{t\in I: \ R_t\prec R_i, \ M_{i-1,t}=0, \ M_{a,t}=0\text{ for all }a\in A, \ R_t\prec R_b\text{ for all }b\in B\}.
\end{align*}
Let $C=C_{i-1}\cup C_i\cup\{i-1,i\}$ and suppose that $A\cup B\cup C=\{1,\ldots,n\}$, in other words every row of $\bm\lambda$ falls into one of these categories. Then there is a horizontal-strip $\bm\mu=(S_1,\ldots,S_n)\in\mathcal S(\bm\lambda)$ and $\varphi:\Pi(\bm\lambda)\xrightarrow\sim\Pi(\bm\mu)$ with $l(S_{\varphi_t})=l(R_t)$ for all $t\in A\cup B$ and $l(S_{\varphi_i})=l(R_{i-1})$. 
\end{lemma}

\begin{example}
Figure \ref{fig:localrotation} illustrates a situation where we can apply Lemma \ref{lem:localrotation}. The rows below $\bm\lambda$ with crosses signify rows that cannot be present because the condition $A\cup B\cup C=\{1,\ldots,n\}$ requires that every row of $\bm\lambda$ be either disjoint from $R_{i-1}$ and $R_i$, contained in $R_{i-1}$, contained in $R_i$, or containing both. Lemma \ref{lem:localrotation} allows us to locally rotate the six rows of $C$ to produce the similar horizontal-strip $\bm\mu$. Although our proof constructs this specific $\bm\mu$, we will only need that $l(S_{\varphi_t})=l(R_t)$ for all $t\in A\cup B$ and $l(S_{\varphi_i})=l(R_{i-1})$. 
\end{example}

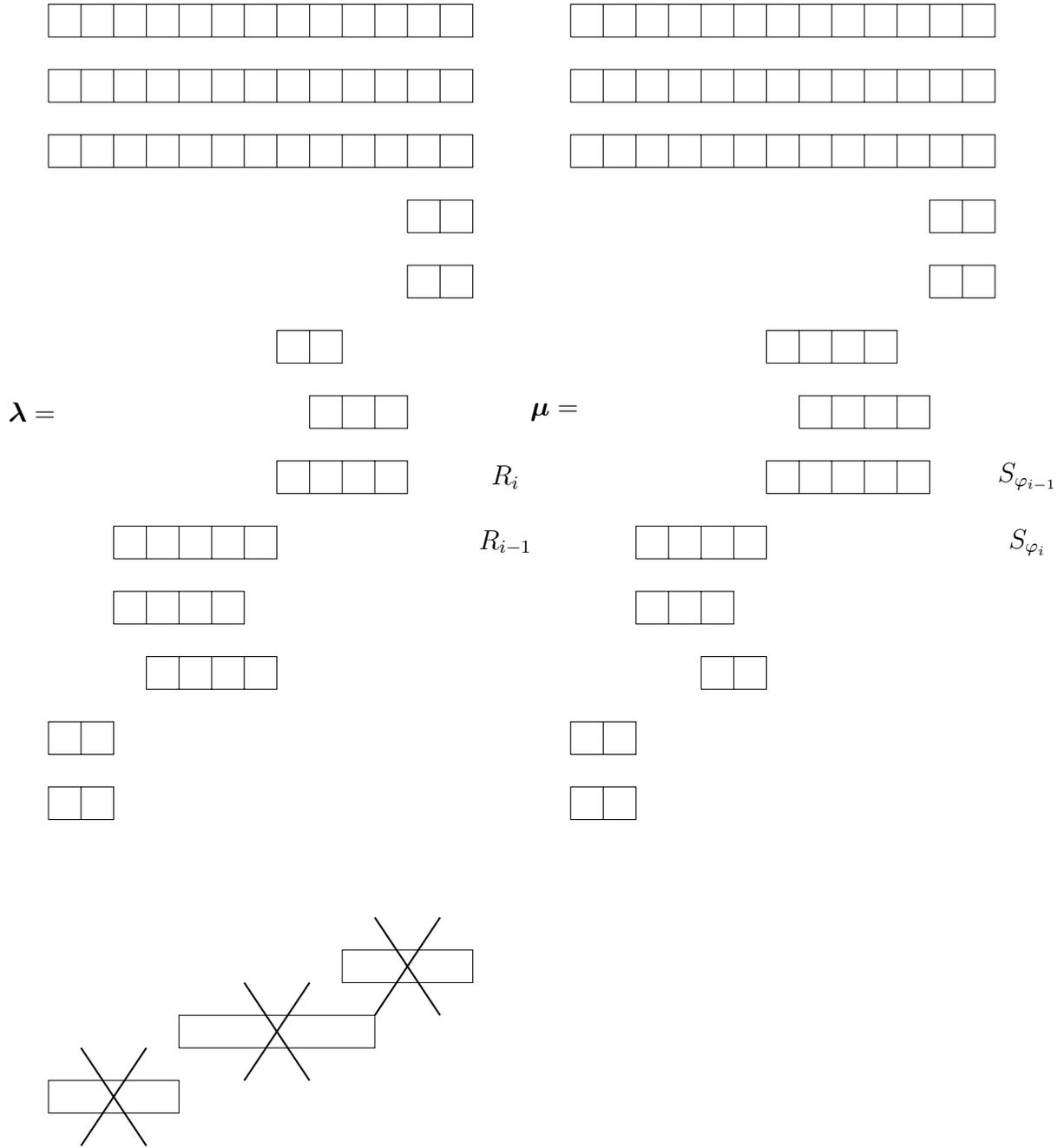
\begin{figure}\caption{A horizontal-strip $\bm\lambda$ and a local rotation $\bm\mu$}\label{fig:localrotation}
$$
\begin{tikzpicture}
\draw (-0.25,-1.25) node (0) {$\bm \lambda=$};
\draw (7,-3.25) node (1) {$R_{i-1}$} (7,-2.25) node (2) {$R_i$};
\draw (0,2.5) -- (6.5,2.5) -- (6.5,3) -- (0,3) -- (0,2.5) (0.5,2.5) -- (0.5,3) (1,2.5) -- (1,3) (1.5,2.5) -- (1.5,3) (2,2.5) -- (2,3) (2.5,2.5) -- (2.5,3) (3,2.5) -- (3,3) (3.5,2.5) -- (3.5,3) (4,2.5) -- (4,3) (4.5,2.5) -- (4.5,3) (5,2.5) -- (5,3) (5.5,2.5) -- (5.5,3) (6,2.5) -- (6,3) (0,3.5) -- (6.5,3.5) -- (6.5,4) -- (0,4) -- (0,3.5) (0.5,3.5) -- (0.5,4) (1,3.5) -- (1,4) (1.5,3.5) -- (1.5,4) (2,3.5) -- (2,4) (2.5,3.5) -- (2.5,4) (3,3.5) -- (3,4) (3.5,3.5) -- (3.5,4) (4,3.5) -- (4,4) (4.5,3.5) -- (4.5,4) (5,3.5) -- (5,4) (5.5,3.5) -- (5.5,4) (6,3.5) -- (6,4) (0,4.5) -- (6.5,4.5) -- (6.5,5) -- (0,5) -- (0,4.5) (0.5,4.5) -- (0.5,5) (1,4.5) -- (1,5) (1.5,4.5) -- (1.5,5) (2,4.5) -- (2,5) (2.5,4.5) -- (2.5,5) (3,4.5) -- (3,5) (3.5,4.5) -- (3.5,5) (4,4.5) -- (4,5) (4.5,4.5) -- (4.5,5) (5,4.5) -- (5,5) (5.5,4.5) -- (5.5,5) (6,4.5) -- (6,5) (5.5,1.5) -- (6.5,1.5) -- (6.5,2) -- (5.5,2) -- (5.5,1.5) (6,1.5) -- (6,2) (5.5,0.5) -- (6.5,0.5) -- (6.5,1) -- (5.5,1) -- (5.5,0.5) (6,0.5) -- (6,1) (0,-6.5) -- (1,-6.5) -- (1,-6) -- (0,-6) -- (0,-6.5) (0.5,-6.5) -- (0.5,-6) (0,-7.5) -- (1,-7.5) -- (1,-7) -- (0,-7) -- (0,-7.5) (0.5,-7.5) -- (0.5,-7);
\draw (3.5,-0.5) -- (4.5,-0.5) -- (4.5,0) -- (3.5,0) -- (3.5,-0.5) (4,-0.5) -- (4,0) (4,-1.5) -- (5.5,-1.5) -- (5.5,-1) -- (4,-1) -- (4,-1.5) (4.5,-1.5) -- (4.5,-1) (5,-1.5) -- (5,-1) (3.5,-2.5) -- (5.5,-2.5) -- (5.5,-2) -- (3.5,-2) -- (3.5,-2.5) (4,-2.5) -- (4,-2) (4.5,-2.5) -- (4.5,-2) (5,-2.5) -- (5,-2) (1,-3.5) -- (3.5,-3.5) -- (3.5,-3) -- (1,-3) -- (1,-3.5) (1.5,-3.5) -- (1.5,-3) (2,-3.5) -- (2,-3) (2.5,-3.5) -- (2.5,-3) (3,-3.5) -- (3,-3) (1,-4.5) -- (3,-4.5) -- (3,-4) -- (1,-4) -- (1,-4.5) (1.5,-4.5) -- (1.5,-4) (2,-4.5) -- (2,-4) (2.5,-4.5) -- (2.5,-4) (1.5,-5.5) -- (3.5,-5.5) -- (3.5,-5) -- (1.5,-5) -- (1.5,-5.5) (2,-5.5) -- (2,-5) (2.5,-5.5) -- (2.5,-5) (3,-5.5) -- (3,-5);
\draw (0,-12) -- (2,-12) -- (2,-11.5) -- (0,-11.5) -- (0,-12) (4.5,-10) -- (6.5,-10) -- (6.5,-9.5) -- (4.5,-9.5) -- (4.5,-10) (2,-11) -- (5,-11) -- (5,-10.5) -- (2,-10.5) -- (2,-11);
\draw [thick] (0.5,-12.5) -- (1.5,-11) (0.5,-11) -- (1.5,-12.5) (5,-10.5) -- (6,-9) (5,-9) -- (6,-10.5) (3,-11.5) -- (4,-10) (3,-10) -- (4,-11.5);
\draw (7.75,-1.25) node (0) {$\bm \mu=$};
\draw (15,-3.25) node (1) {$S_{\varphi_i}$} (15,-2.25) node (2) {$S_{\varphi_{i-1}}$};
\draw (8,2.5) -- (14.5,2.5) -- (14.5,3) -- (8,3) -- (8,2.5) (8.5,2.5) -- (8.5,3) (9,2.5) -- (9,3) (9.5,2.5) -- (9.5,3) (10,2.5) -- (10,3) (10.5,2.5) -- (10.5,3) (11,2.5) -- (11,3) (11.5,2.5) -- (11.5,3) (12,2.5) -- (12,3) (12.5,2.5) -- (12.5,3) (13,2.5) -- (13,3) (13.5,2.5) -- (13.5,3) (14,2.5) -- (14,3) (8,3.5) -- (14.5,3.5) -- (14.5,4) -- (8,4) -- (8,3.5) (8.5,3.5) -- (8.5,4) (9,3.5) -- (9,4) (9.5,3.5) -- (9.5,4) (10,3.5) -- (10,4) (10.5,3.5) -- (10.5,4) (11,3.5) -- (11,4) (11.5,3.5) -- (11.5,4) (12,3.5) -- (12,4) (12.5,3.5) -- (12.5,4) (13,3.5) -- (13,4) (13.5,3.5) -- (13.5,4) (14,3.5) -- (14,4) (8,4.5) -- (14.5,4.5) -- (14.5,5) -- (8,5) -- (8,4.5) (8.5,4.5) -- (8.5,5) (9,4.5) -- (9,5) (9.5,4.5) -- (9.5,5) (10,4.5) -- (10,5) (10.5,4.5) -- (10.5,5) (11,4.5) -- (11,5) (11.5,4.5) -- (11.5,5) (12,4.5) -- (12,5) (12.5,4.5) -- (12.5,5) (13,4.5) -- (13,5) (13.5,4.5) -- (13.5,5) (14,4.5) -- (14,5) (13.5,1.5) -- (14.5,1.5) -- (14.5,2) -- (13.5,2) -- (13.5,1.5) (14,1.5) -- (14,2) (13.5,0.5) -- (14.5,0.5) -- (14.5,1) -- (13.5,1) -- (13.5,0.5) (14,0.5) -- (14,1) (8,-6.5) -- (9,-6.5) -- (9,-6) -- (8,-6) -- (8,-6.5) (8.5,-6.5) -- (8.5,-6) (8,-7.5) -- (9,-7.5) -- (9,-7) -- (8,-7) -- (8,-7.5) (8.5,-7.5) -- (8.5,-7);
\draw (11,-0.5) -- (13,-0.5) -- (13,0) -- (11,0) -- (11,-0.5) (11.5,-0.5) -- (11.5,0) (12,-0.5) -- (12,0) (12.5,-0.5) -- (12.5,0) (11.5,-1.5) -- (13.5,-1.5) -- (13.5,-1) -- (11.5,-1) -- (11.5,-1.5) (12,-1.5) -- (12,-1) (12.5,-1.5) -- (12.5,-1) (13,-1.5) -- (13,-1) (11,-2.5) -- (13.5,-2.5) -- (13.5,-2) -- (11,-2) -- (11,-2.5) (11.5,-2.5) -- (11.5,-2) (12,-2.5) -- (12,-2) (12.5,-2.5) -- (12.5,-2) (13,-2.5) -- (13,-2) (9,-3.5) -- (11,-3.5) -- (11,-3) -- (9,-3) -- (9,-3.5) (9.5,-3.5) -- (9.5,-3) (10,-3.5) -- (10,-3) (10.5,-3.5) -- (10.5,-3) (9,-4.5) -- (10.5,-4.5) -- (10.5,-4) -- (9,-4) -- (9,-4.5) (9.5,-4.5) -- (9.5,-4) (10,-4.5) -- (10,-4) (10,-5.5) -- (11,-5.5) -- (11,-5) -- (10,-5) -- (10,-5.5) (10.5,-5.5) -- (10.5,-5);
\end{tikzpicture}$$
\end{figure}

\begin{proof}[Proof of Lemma \ref{lem:localrotation}. ]
Informally, we will first use commuting and cycling to bring the rows of $C$ together. To be specific, we first claim that if $t\in C_i$ and $t'\notin C_i$ with either $t>t'>i$, then $R_t\leftrightarrow R_{t'}$. Because $t\in C_i$, we have $R_t\prec R_i$ and $M_{i-1,t}=0$, so by Proposition \ref{prop:prec} we have $l(R_i)\leq l(R_t)\leq r(R_t)\leq r(R_i)$. Now if $t'\in A$, then $M_{i-1,t'}=M_{i,t}=M_{t,t'}=0$, so either $r(R_{t'})<l(R_{i-1})-1<l(R_t)-1$ and $R_t\leftrightarrow R_{t'}$ by Proposition \ref{prop:mrirj}, Parts 2 and 3, or $l(R_{t'})>r(R_i)\geq r(R_t)\geq l(R_t)$, but now we cannot have $R_t\nleftrightarrow R_{t'}$ by Corollary \ref{cor:mrirj}, Part 3. If $t'\in B$, then $R_{t'}\prec R_{i-1}$, so by Proposition \ref{prop:prec} we have $l(R_{t'})\leq l(R_{i-1})-1<l(R_t)$, but now we cannot have $R_t\nleftrightarrow R_{t'}$ by Corollary \ref{cor:mrirj}, Part 3. If $t'\in C_{i-1}$, then because $R_{t'}\prec R_{i-1}$ and $M_{i,t'}=0$, we have $r(R_{t'})<l(R_i)-1\leq l(R_t)-1$, so $R_t\leftrightarrow R_{t'}$ by Proposition \ref{prop:mrirj}, Part 1. This establishes our claim. \\

Therefore, by cycling and commuting, we may assume that $C_i=\{i+1,\ldots,y\}$ for some $y$ and similarly, by considering a rotation of $\bm\lambda$, we may assume that $C_{i-1}=\{x,\ldots,i-2\}$ for some $x$. In particular, we have $l(R_{i-1})\leq l(R_t)\leq r(R_t)\leq r(R_{i-1})$ for every $t\in C_{i-1}$ and $l(R_i)\leq l(R_t)\leq r(R_t)\leq r(R_i)$ for every $t\in C_i$. To summarize, we have \begin{equation}\bm\lambda=(R_1,\ldots,R_{x-1},R_x,\ldots,R_{i-2},R_{i-1},R_i,R_{i+1},\ldots,R_y,R_{y+1},\ldots,R_n),\end{equation}
where $A\cup B=\{1,\ldots,x-1\}\cup\{y+1,\ldots,n\}$, $C_{i-1}=\{x,\ldots,i-2\}$, and $C_i=\{i+1,\ldots,y\}$.\\

Now let $N=l(R_{i-1})+r(R_i)$ and define the horizontal-strip $\bm\mu=(S_1,\ldots,S_n)$ by $S_t=R_t$ if $t<x$ or $t>y$, and $S_t=N-R_{x+y-t}$ otherwise, that is 
\begin{equation}\bm\mu=(R_1,\ldots,R_{x-1},N-R_y,\ldots,N-R_i,N-R_{i-1},\ldots,N-R_x,R_{y+1},\ldots,R_n),\end{equation}
and define $\varphi:\{1,\ldots,n\}\to\{1,\ldots,n\}$ by $\varphi_t=t$ if $t<x$ or $t>y$ and $\varphi_t=x+y-t$ otherwise. Indeed, we have $l(S_{\varphi_t})=l(R_t)$ for all $t\in A\cup B$ and $l(S_{\varphi_i})=\ell(N-R_i)=l(R_{i-1})+r(R_i)-r(R_i)=l(R_{i-1})$. We claim that $\varphi:\Pi(\bm\lambda)\xrightarrow\sim\Pi(\bm\mu)$. We have $|R_t|=|S_{\varphi_t}|$ for every $1\leq t\leq n$, so it remains to check that the edge weights are preserved. If $t,t'\in A\cup B$, then the relative positions of $R_t$ and $R_{t'}$ have not changed, so indeed $M_{t,t'}(\bm\lambda)=M_{\varphi_t,\varphi_{t'}}(\bm\mu)$. If $t,t'\in C$, then this follows because $M(R_t,R_{t'})=M(N-R_{t'},N-R_t)$. Now suppose that $t\in A\cup B$ and $t'\in C$. We have either $t'<i-1$ and $R_{t'}\prec R_{i-1}$ or $t'>i$ and $R_{t'}\prec R_i$, so in either case we have \begin{equation}l(R_{i-1})\leq l(R_{t'})\leq r(R_{t'})\leq r(R_i), \text{ so }l(R_{i-1})\leq\ell(N-R_{t'})\leq r(N-R_{t'})\leq r(R_i).\end{equation}
Also note that $\chi(t>i)=\chi(t>i-1)=\chi(t>t')=\chi(t>x+y-t')$. Now if $t\in A$, we have either \begin{align}l(R_t)&>r(R_i)+\chi(i>t)\geq r(N-R_{t'})+\chi(i>x+y-t')\text{ or }\\r(R_t)&<l(R_{i-1})-\chi(t>i)\leq\ell(N-R_{t'})-\chi(t>x+y-t'),\end{align}
so in either case, we have $M_{\varphi_t,\varphi_{t'}}(\bm\mu)=0$. Similarly, if $t\in B$, we have 
\begin{align}
l(R_t)&\leq l(R_{i-1})+\chi(i-1>t)\leq\ell(N-R_{t'})+\chi(x+y-t'>t)\text{ and }\\
r(R_t)&\geq r(R_i)-\chi(t>i)\geq r(N-R_{t'})-\chi(t>x+y-t'),
\end{align}
so we have $N-R_{t'}\prec R_t$ and  $M_{\varphi_t,\varphi_{t'}}(\bm\mu)=|S_{\varphi_{t'}}|$.\\

Finally, we show that $G_{\bm\lambda}(\bm x;q)=G_{\bm\mu}(\bm x;q)$. Define the horizontal-strips
\begin{align}
\bm\lambda'&=(R_1,\ldots,R_i,R_{i-1},\ldots,R_n),\\
\bm\lambda''&=(R_1,\ldots,R_{i-1}\cup R_i,R_{i-1}\cap R_i,\ldots,R_n),\\
\bm\mu'&=(R_1,\ldots,N-R_y,\ldots,N-R_{i-1},N-R_i,\ldots,N-R_x,\ldots,R_n),\text{ and }\\
\bm\mu''&=(R_1,\ldots,(N-R_i)\cup(N-R_{i-1}),(N-R_i)\cap(N-R_{i-1}),\ldots,R_n).
\end{align}
Because Lemma \ref{lem:newgraphs} describes exactly how to derive the weighted graphs $\Pi(\bm\lambda')$ and $\Pi(\bm\lambda'')$ from $\Pi(\bm\lambda)$, we have that $\Pi(\bm\lambda')\cong\Pi(\bm\mu')$ and $\Pi(\bm\lambda'')\cong\Pi(\bm\mu'')$. We also have $n(\bm\lambda'')<n(\bm\lambda)$, $n(\bm\lambda')=n(\bm\lambda)$, and $M(\bm\lambda')>M(\bm\lambda)$, so because $\bm\lambda$ satisfies \eqref{eq:IH} by hypothesis, we have that $G_{\bm\lambda'}(\bm x;q)=G_{\bm\mu'}(\bm x;q)$ and $G_{\bm\lambda''}(\bm x;q)=G_{\bm\mu''}(\bm x;q)$. Finally, by Lemma \ref{lem:inductiverelation}, we have \begin{equation}G_{\bm\lambda}(\bm x;q)=\frac 1qG_{\bm\lambda'}(\bm x;q)+\frac{q-1}qG_{\bm\lambda''}(\bm x;q)=\frac 1qG_{\bm\mu'}(\bm x;q)+\frac{q-1}qG_{\bm\mu''}(\bm x;q)=G_{\bm\mu}(\bm x;q).\end{equation}
This completes the proof.
\end{proof}

The hypothesis of Lemma \ref{lem:localrotation} that $A\cup B\cup C=\{1,\ldots,n\}$ is a little technical so it will be convenient to rephrase it as follows.

\begin{proposition}\label{prop:localrotationhypothesis}
Let $\bm\lambda=(R_1,\ldots,R_n)$ be a horizontal-strip with $l(R_i)=r(R_{i-1})+1$ and define the sets $A$, $B$, $C$, and $I$ as in Lemma \ref{lem:localrotation}. If the following hold for every $t\in I$, then we have $A\cup B\cup C=\{1,\ldots,n\}$.\begin{enumerate}
\item If $M_{i-1,t}>0$ and $M_{i,t}>0$, then $R_{i-1}\prec R_t$ and $R_i\prec R_t$.
\item If $M_{i-1,t}>0$ and $M_{i,t}=0$, then $R_t\prec R_{i-1}$. 
\item If $M_{i-1,t}=0$ and $M_{i,t}>0$, then $R_t\prec R_i$. 
\item If $R_t\prec R_{i-1}$, then $M_{i,t}=0$, $M_{a,t}=0$ for all $a\in A$, and $R_t\prec R_b$ for all $b\in B$.
\item If $R_t\prec R_i$, then $M_{i-1,t}=0$, $M_{a,t}=0$ for all $a\in A$, and $R_t\prec R_b$ for all $b\in B$.
\end{enumerate}
\end{proposition}

\begin{proof}
Let $t\in I$. We need to show that $t\in A\cup B\cup C_{i-1}\cup C_i$. The integers $M_{i-1,t}$ and $M_{i,t}$ are either zero or nonzero. If $M_{i-1,t}=M_{i,t}=0$, then $t\in A$. If $M_{i-1,t}>0$ and $M_{i,t}>0$, then by (1) we have $t\in B$. If $M_{i-1,t}>0$ and $M_{i,t}=0$, then by (2) and (4) we have $t\in C_{i-1}$. If $M_{i-1,t}=0$ and $M_{i,t}>0$, then by (3) and (5) we have $t\in C_i$.
\end{proof}

The next Lemma is very technical but it is the key idea that uses local rotation to extend Corollary \ref{cor:adjacentstrictdone} to strict sequences.

\begin{lemma} \label{lem:strictsequencedone}
Let $\bm\lambda=(R_1,\ldots,R_n)$ be a horizontal-strip with a sequence $(R_{j_1},\ldots,R_{j_k})$ with $k\geq 2$, $j_1<\cdots<j_k$, $l(R_{j_{t+1}})=r(R_{j_t})+1$ for $1\leq t\leq k-1$, and suppose that there is no noncommuting path in $\bm\lambda$ from $R_{j_t}$ to $R_{j_{t+1}}$ for any $1\leq t\leq k-1$. Assume that $\bm\lambda$ satisfies \eqref{eq:IH}. Let $\bm\mu=(S_1,\ldots,S_n)$ and  $\varphi:\Pi(\bm\lambda)\xrightarrow\sim\Pi(\bm\mu)$ be such that \begin{equation}\label{eq:strictsequencehypothesis}\text{ for some permutation }\sigma:\{1,\ldots,k\}\to\{1,\ldots,k\}\text{ we have }l(S_{\varphi_{j_{\sigma_{t+1}}}})=r(S_{\varphi_{j_{\sigma_t}}})+1\end{equation}
for every $1\leq t\leq k-1$. Then there exists a good substitute for $(\bm\lambda,\bm\mu)$. 
\end{lemma}

\begin{example}
Lemma \ref{lem:strictsequencedone} applies in a situation like the one below. Informally, the condition \eqref{eq:strictsequencehypothesis} asks that these rows in $\bm\lambda$ still link end to end in $\bm\mu$, although they may be permuted. In this example, we have $\sigma_1=2$, $\sigma_2=4$, $\sigma_3=1$, and $\sigma_4=3$.

$$
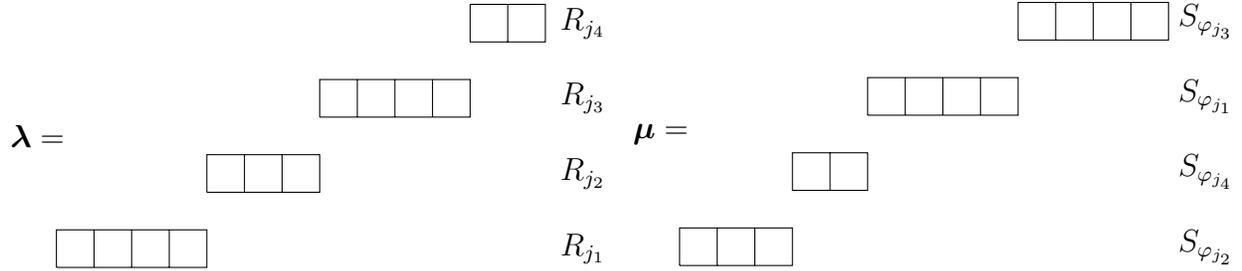
\begin{tikzpicture}
\draw (-0.25,-1.75) node (0) {$\bm \lambda=$};
\draw (7,-3.25) node (4) {$R_{j_1}$} (7,-2.25) node (5) {$R_{j_2}$} (7,-1.25) node (6) {$R_{j_3}$} (7,-0.25) node (7) {$R_{j_4}$};
\draw (0,-3.5) -- (2,-3.5) -- (2,-3) -- (0,-3) -- (0,-3.5) (0.5,-3.5) -- (0.5,-3) (1,-3.5) -- (1,-3) (1.5,-3.5) -- (1.5,-3) (2,-2.5) -- (3.5,-2.5) -- (3.5,-2) -- (2,-2) -- (2,-2.5) (2.5,-2.5) -- (2.5,-2) (3,-2.5) -- (3,-2) (3.5,-1.5) -- (5.5,-1.5) -- (5.5,-1) -- (3.5,-1) -- (3.5,-1.5) (4,-1.5) -- (4,-1) (4.5,-1.5) -- (4.5,-1) (5,-1.5) -- (5,-1) (5.5,-0.5) -- (6.5,-0.5) -- (6.5,0) -- (5.5,0) -- (5.5,-0.5) (6,-0.5) -- (6,0);
\end{tikzpicture} \ 
\begin{tikzpicture}
\draw (-0.25,-1.75) node (0) {$\bm \mu=$};
\draw (7,-3.25) node (4) {$S_{\varphi_{j_2}}$} (7,-2.25) node (5) {$S_{\varphi_{j_4}}$} (7,-1.25) node (6) {$S_{\varphi_{j_1}}$} (7,-0.25) node (7) {$S_{\varphi_{j_3}}$};
\draw (0,-3.5) -- (1.5,-3.5) -- (1.5,-3) -- (0,-3) -- (0,-3.5) (0.5,-3.5) -- (0.5,-3) (1,-3.5) -- (1,-3) (1.5,-2.5) -- (2.5,-2.5) -- (2.5,-2) -- (1.5,-2) -- (1.5,-2.5) (2,-2.5) -- (2,-2) (2.5,-1.5) -- (4.5,-1.5) -- (4.5,-1) -- (2.5,-1) -- (2.5,-1.5) (3,-1.5) -- (3,-1) (3.5,-1.5) -- (3.5,-1) (4,-1.5) -- (4,-1) (4.5,-0.5) -- (6.5,-0.5) -- (6.5,0) -- (4.5,0) -- (4.5,-0.5) (5,-0.5) -- (5,0) (5.5,-0.5) -- (5.5,0) (6,-0.5) -- (6,0);
\end{tikzpicture}$$
\end{example}

\begin{remark}
If $(R_{j_1},\ldots,R_{j_k})$ is a strict sequence of $\bm\lambda$, then $k\geq 2$, $j_1<\cdots<j_k$, and by Proposition \ref{prop:strictsequence}, $l(R_{j_{t+1}})=r(R_{j_t})+1$ for $1\leq t\leq k-1$. Moreover, if $\bm\mu=(S_1,\ldots,S_n)$ and $\varphi:\Pi(\bm\lambda)\xrightarrow\sim\Pi(\bm\mu)$, then by Remark \ref{rem:strictsequencetransfer} and by cycling $\bm\mu$ if necessary, there will be a permutation $\sigma:\{1,\ldots,k\}\to\{1,\ldots,k\}$ with $(S_{\varphi_{j_{\sigma_1}}},\ldots,S_{\varphi_{j_{\sigma_k}}})$ a strict sequence and therefore $l(S_{\varphi_{j_{\sigma_{t+1}}}})=r(S_{\varphi_{j_{\sigma_t}}})+1$ for every $1\leq t\leq k-1$. Therefore, a strict sequence satisfies the hypothesis \eqref{eq:strictsequencehypothesis} of Lemma \ref{lem:strictsequencedone}. 
\end{remark}

\begin{remark}
Informally, the strategy will be to apply Lemma \ref{lem:localrotation} to perform a series of local rotations to permute the rows of $\bm\lambda$ to match those of $\bm\mu$. We will be able to perform these local rotations unless some other row $R_t$ of $\bm\lambda$ violates some condition of Proposition \ref{prop:localrotationhypothesis}, forcing certain rows of $\bm\lambda$ to link end to end. However, in this case, these rows of $\bm\lambda$ will be a proper subset of rows that satisfies \eqref{eq:strictsequencehypothesis} and we can use induction to reason about these rows. 
\end{remark}

\begin{proof}[Proof of Lemma \ref{lem:strictsequencedone}. ]
We use induction on $k$. If $k=2$, then because $l(R_{j_2})=r(R_{j_1})+1$ and by hypothesis there is no noncommuting path in $\bm\lambda$ from $R_{j_1}$ to $R_{j_2}$, we can use Lemma \ref{lem:aoncp} to replace $\bm\lambda$ with a similar horizontal-strip as necessary to assume that $j_2=j_1+1$, and then by \eqref{eq:strictsequencehypothesis} we have either $l(S_{\varphi_{j_2}})=r(S_{\varphi_{j_1}})+1$ or $l(S_{\varphi_{j_1}})=r(S_{\varphi_{j_2}})+1$, so $S_{\varphi_{j_1}}\nleftrightarrow S_{\varphi_{j_2}}$ and the result follows from Corollary \ref{cor:adjacentnoncommutingdone}. So we now assume that $k\geq 3$ and that the result holds for $2\leq k'\leq k-1$. In particular, if $J\subseteq\{1,\ldots,k\}$ is an interval with $2\leq |J|\leq k-1$ and such that $\sigma^{-1}(J)=\{1\leq t'\leq k: \ \sigma_{t'}\in J\}\subseteq\{1,\ldots,k\}$ is an interval, then the sequence of rows $(R_{j_t}: \ t\in J)$ satisfies \eqref{eq:strictsequencehypothesis} and we are done by our induction hypothesis on $k$. In particular, if $\sigma_1=1$ or $\sigma_k=1$, then we can take $J=\{2,\ldots,k\}$, and if $\sigma_1=k$ or $\sigma_k=k$, then we can take $J=\{1,\ldots,k-1\}$, so we may assume that \begin{equation}\label{eq:strictsequencedone1} 2\leq \sigma_1,\sigma_k\leq k-1.\end{equation}

We now continue to use our induction hypothesis to make several additional simplifying assumptions. For $1\leq t\leq n$ such that $t\neq j_{t'}$ for any $1\leq t'\leq k$, consider the sets \begin{equation}E_t=\{1\leq t'\leq k: \ R_{j_{t'}}\prec R_t\}\text{ and }F_t=\{1\leq t\leq k: \ M_{j_{t'},t}(\bm\lambda)>0\}.\end{equation}
Note that $E_t\subseteq F_t$. Also, if $t_1<t_2<t_3$ and $t_1,t_3\in F_t$, then by Proposition \ref{prop:mrirj}, Parts 1 and 3, we have \begin{equation} l(R_t)\leq r(R_{j_{t_1}})+1\leq l(R_{j_{t_2}})\leq r(R_{j_{t_2}})\leq l(R_{j_{t_3}})-1\leq r(R_t),\end{equation}
so by Proposition \ref{prop:prec} we have $R_{t_2}\prec R_t$ and $t_2\in E_t$, so $E_t$ and $F_t$ are intervals in $\{1,\ldots,k\}$. Similarly, consider the sets \begin{equation}E_t'=\{1\leq t'\leq k: \ S_{\varphi_{j_{\sigma_{t'}}}}\prec S_{\varphi_t}\}\text{ and }F_t'=\{1\leq t'\leq k: \ M_{\varphi_{j_{\sigma_{t'}}},\varphi_t}(\bm\mu)>0\}.\end{equation}
Note that $E_t'\subseteq F_t'$, $E_t'=\sigma^{-1}(E_t)$, $F_t'=\sigma^{-1}(F_t)$ and as before, if $t_1<t_2<t_3$ and $t_1,t_3\in F_t'$, then by Proposition \ref{prop:mrirj}, Parts 1 and 3, we have \begin{equation} l(S_{\varphi_t})\leq r(S_{\varphi_{j_{\sigma_{t_1}}}})+1\leq l(S_{\varphi_{j_{\sigma_{t_2}}}})\leq r(S_{\varphi_{j_{\sigma_{t_2}}}})\leq l(S_{\varphi_{j_{\sigma_{t_3}}}})-1\leq r(S_{\varphi_t}),\end{equation}
so by Proposition \ref{prop:prec}, we have $S_{\varphi_{j_{\sigma_{t_2}}}}\prec S_{\varphi_t}$ and $t_2\in E_t'$, so $E_t'$ and $F_t'$ are intervals in $\{1,\ldots,k\}$. Therefore, if $2\leq |F_t|\leq k-1$, then taking $J=F_t$ above we are done by our induction hypothesis on $k$. Similarly, if $|F_t|=k$ and $t'\notin E_t$, then taking $J=E_t$ above we are done by our induction hypothesis on $k$. This means that we may assume that \begin{equation}\label{eq:strictsequencedone2}\text{ if }|F_t|\geq 2,\text{ then }E_t=F_t=\{1,\ldots,k\}.\end{equation}

Now suppose that $F_t=\{t'\}$. Our goal is to show that $R_t\prec R_{j_{t'}}$, that $M_{a,t}(\bm\lambda)=0$ for every $a$ with $M_{a,j_{t'}}(\bm\lambda)=0$, and $R_t\prec R_b$ for every $b$ with $R_{j_{t'}}\prec R_b$.\\

If $2\leq t'\leq k-1$, then $M_{j_{t'-1},t}(\bm\lambda)=0$ and $M_{j_{t'+1},t}(\bm\lambda)=0$ by definition of $F_t$ and by Proposition \ref{prop:mrirj}, Parts 1 and 3, we would have 
\begin{equation}l(R_{j_{t'}})=r(R_{j_{t'-1}})+1\leq l(R_t)\leq r(R_t)\leq l(R_{j_{t'+1}})-1=r(R_{j_{t'}})\end{equation}
and therefore $R_t\prec R_{j_{t'}}$ by Proposition \ref{prop:prec}. This means that if $R_t\nprec R_{j_{t'}}$, we must have $t'=1$ or $k$ and for the same reason, $\sigma_{t'}=1$ or $k$, but this contradicts our assumption \eqref{eq:strictsequencedone1}. Therefore, we must have $R_t\prec R_{j_{t'}}$. \\

Now suppose that there is some $a$ with $M_{a,j_{t'}}(\bm\lambda)=0$ but $M_{a,t}(\bm\lambda)>0$. Because $R_t\prec R_{j_{t'}}$, we have $M_{j_{t'},t}(\bm\lambda)+M_{a,t}(\bm\lambda)\geq |R_t|+1$, so either $j_{t'}<a$ and the pair $(R_{j_{t'}},R_a)$ is strict, or $a<j_{t'}$ and the pair $(R_a,R_{j_{t'}})$ is strict. In either case, by Proposition \ref{prop:strict} we have $R_{j_{t'}}\nleftrightarrow R_a$, and by Corollary \ref{cor:mrirj}, Part 3, we have either \begin{equation}\label{eq:strictsequencejta} l(R_a)=r(R_{j_{t'}})+1=l(R_{j_{t'+1}})\text{ or }l(R_{j_{t'}})=r(R_a)+1.\end{equation}
In particular, if $t'=1$, then either $F_a=\emptyset$ or $\{2\}$, if $2\leq t'\leq k-1$, then either $F_a=\{t'-1\}$ or $\{t'+1\}$, and if $t'=k$, then either $F_a=\emptyset$ or $\{k-1\}$. Similarly, $F_a'$ is either empty, in which case $\sigma_{t'}=1$ or $k$, or $F_a'$ and $F_t'$ consist of consecutive singletons. Therefore, if $F_a=\emptyset$, then this contradicts \eqref{eq:strictsequencedone1}, and otherwise, taking $J=F_t\cup F_a$ we are done by our induction hypothesis on $k$. \\

Next, let us suppose that for some $b$ we have $R_{j_{t'}}\prec R_b$ but $R_t\nprec R_b$. If $2\leq t'\leq k-1$, then $M_{j_{t'-1},t}(\bm\lambda)=0$ and $M_{j_{t'+1},t}(\bm\lambda)=0$ by definition of $F_t$ and by Proposition \ref{prop:mrirj}, Parts 2 and 3, we would have \begin{equation}l(R_{j_{t'}})=r(R_{j_{t'-1}})+1\leq l(R_t)\leq r(R_t)\leq l(R_{j_{t'+1}})-1=r(R_{j_{t'}}).\end{equation}
Now if $R_{j_{t'}}\subseteq R_b$, we would have $R_t\subseteq R_{j_{t'}}\subseteq R_b$ and $R_t\prec R_b$ by Proposition \ref{prop:prec}, so we must have either $R_{j_{t'}}\subseteq R_b^+$ and $l(R_b)\leq l(R_{j_{t'}})-1=r(R_{j_{t'-1}})$, or $R_{j_{t'}}\subseteq R_b^-$ and $r(R_b)\geq r(R_{j_{t'}})+1=l(R_{j_{t'+1}})$. However, this means that either $\{t'-1,t'\}\subseteq F_b$ or $\{t',t'+1\}\subseteq F_b$, so by \eqref{eq:strictsequencedone2} we must have $F_b=\{1,\ldots,k\}$, so $M_{j_{t'-1},t}(\bm\lambda)>0$ and $M_{j_{t'+1},t}(\bm\lambda)>0$, and by Proposition \ref{prop:prec}, we have \begin{equation}l(R_b)\leq r(R_{j_{t'-1}})+1=l(R_{j_{t'}})\leq l(R_t)\leq r(R_t)\leq r(R_{j_{t'}})= l(R_{j_{t'+1}})-1\leq r(R_b)\end{equation}
and $R_t\prec R_b$ by Proposition \ref{prop:prec} after all. Therefore, if $R_{j_{t'}}\prec R_b$ but $R_t\nprec R_b$, we must have $t'=1$ or $k$ and for the same reason, $\sigma_{t'}=1$ or $k$, but this contradicts our assumption \eqref{eq:strictsequencedone1}. To summarize, we may assume that \begin{align}\label{eq:strictsequencedone3} &\text{ if }|F_t|\geq 2,\text{ then }E_t=F_t=\{1,\ldots,k\},\text{ and if }F_t=\{t'\},\text{ then }R_t\prec R_{j_{t'}},\\\nonumber &M_{a,t}(\bm\lambda)=0\text{ whenever }M_{a,j_{t'}}(\bm\lambda)=0,\text{ and }R_t\prec R_b\text{ whenever }R_{j_{t'}}\prec R_b.\end{align}

Let $t_0$ and $t_1$ be such that $t_0<t_1$ and $\{t_0,t_1\}=\{\sigma_1,\sigma_2\}$, and let $x=j_{t_0}$ and $y=j_{t_1}$. Note that we have $x<y$, $l(R_y)\geq r(R_x)+1$, and $S_{\varphi_x}\nleftrightarrow S_{\varphi_y}$. We use induction on $l(R_y)-r(R_x)$. If $l(R_y)-r(R_x)=1$, then $t_1=t_0+1$ and because there is no noncommuting path in $\bm\lambda$ from $R_x$ to $R_y$, we may use Lemma \ref{lem:aoncp} to assume that $y=x+1$. In this case, it now follows from Corollary \ref{cor:adjacentnoncommutingdone} that there exists a good substitute for $(\bm\lambda,\bm\mu)$. So we now assume that $l(R_y)-r(R_x)\geq 2$, so in particular $t_1\geq t_0+2$. Because there is no noncommuting path in $\bm\lambda$ from $R_{j_{t_1-1}}$ to $R_y$, we may use Lemma \ref{lem:aoncp} to assume that $y=j_{t_1-1}+1$.\\

Our plan now is to apply Lemma \ref{lem:localrotation} to the rows $R_{y-1}$ and $R_y$ to replace $\bm\lambda$ with a similar horizontal-strip for which $l(R_y)$ has decreased and $r(R_x)$ has not changed, so that we will be done by our induction hypothesis on $l(R_y)-r(R_x)$. It remains to check the conditions of Proposition \ref{prop:localrotationhypothesis}. 
\begin{enumerate}
\item If $M_{y-1,t}(\bm\lambda)>0$ and $M_{y,t}(\bm\lambda)>0$, then $|F_t|\geq 2$, so by \eqref{eq:strictsequencedone3} we have $E_t=\{1,\ldots,k\}$, so $R_{y-1}\prec R_t$ and $R_y\prec R_t$.\\

\item If $M_{y-1,t}(\bm\lambda)>0$ and $M_{y,t}(\bm\lambda)=0$, then $t_1-1\in F_t$ and $t_1\notin F_t$, so by \eqref{eq:strictsequencedone3} we must have $F_t=\{t_1-1\}$ and then $R_t\prec R_{y-1}$.\\

\item If $M_{y-1,t}(\bm\lambda)=0$ and $M_{y,t}(\bm\lambda)>0$, then $t_1-1\notin F_t$ and $t_1\in F_t$, so by \eqref{eq:strictsequencedone3} we must have $F_t=\{t_1\}$ and $R_t\prec R_y$.\\

\item If $R_t\prec R_{y-1}$, then $t_1-1\in F_t$. By Proposition \ref{prop:mrirj}, Part 1, and because $x<y$, we cannot have both $M_{x,t}(\bm\lambda)>0$ and $M_{y,t}(\bm\lambda)>0$ because then \begin{align}l(R_t)&\leq r(R_x)+\chi(x>t)\text{ and }r(R_t)\geq l(R_y)-\chi(t>y),\text{ so } \\\nonumber|R_t|&=r(R_t)-l(R_t)+1\geq l(R_y)-r(R_x)+1-\chi(t>y)-\chi(x>t)\\\nonumber&=r(R_{y-1})+1-l(R_{y-1})+1+1-\chi(t>y)-\chi(x>t)>|R_{y-1}|,
\end{align}
contradicting $R_t\prec R_{y-1}$ by Proposition \ref{prop:prec}. Therefore, either $t_0\notin F_t$ or $t_1\notin F_t$, so by \eqref{eq:strictsequencedone3} we must have $F_t=\{t_1-1\}$ and $M_{y,t}(\bm\lambda)=0$. Moreover, if $a\in A$, then in particular $M_{a,y}(\bm\lambda)=0$, so by \eqref{eq:strictsequencedone3} we have $M_{a,t}(\bm\lambda)=0$, and if $b\in B$, then in particular we have $R_{y-1}\prec R_b$, so by \eqref{eq:strictsequencedone3} we have $R_t\prec R_b$ as well.\\

\item If $R_t\prec R_y$, then $t_1\in F_t$. By Proposition \ref{prop:prec} and because $x<y$, we cannot have $M_{x,t}(\bm\lambda)>0$ because then \begin{align}l(R_t)&\leq r(R_x)+\chi(x>t)\leq l(R_{y-1})-1+\chi(x>t)\leq r(R_{y-1})-1+\chi(x>t)\\\nonumber&=l(R_{y-1})-1-1+\chi(x>t)\leq l(R_y)-1-\chi(t>y)<l(R_y)-\chi(t>y),\end{align} 
contradicting $R_t\prec R_y$ by Proposition \ref{prop:prec}. Therefore, $t_0\notin F_t$, so by \eqref{eq:strictsequencedone3} we must have $F_t=\{t_1\}$ and $M_{y-1,t}(\bm\lambda)=0$. Moreover, if $a\in A$, then in particular $M_{a,y-1}(\bm\lambda)=0$, so by \eqref{eq:strictsequencedone3} we have $M_{a,t}(\bm\lambda)=0$, and if $b\in B$, then in particular we have $R_y\prec R_b$, so by \eqref{eq:strictsequencedone3} we have $R_t\prec R_b$ as well.\\
\end{enumerate}

This concludes our verification of the conditions of Proposition \ref{prop:localrotationhypothesis}. Therefore, the result follows by Lemma \ref{lem:localrotation} and our induction hypothesis on $l(R_y)- r(R_x)$.
\end{proof}

We now generalize Corollary \ref{cor:adjacentstrictdone} to the case where $\bm\lambda$ has a pair of strict rows that are not necessarily adjacent. 

\begin{corollary} \label{cor:strictdone} Let $\bm\lambda=(R_1,\ldots,R_n)$ be a horizontal-strip that satisfies \eqref{eq:IH}. If $\bm\lambda$ has a pair of strict rows $(R_i,R_j)$, then $\bm\lambda$ is good. \end{corollary}

\begin{proof}
Because $\bm\lambda$ has a pair of strict rows, we may assume that $(R_i,R_j)$ is a strict pair with $j-i$ minimal, in other words the pairs $(R_{i'},R_{j'})$ are not strict for $i\leq i'<j'\leq j$ with $j'-i'<j-i$. Also note that if there is no minimal noncommuting path in $\bm\lambda$ from $R_i$ to $R_j$, then by Lemma \ref{lem:aoncp} we may replace $\bm\lambda$ with a similar horizontal-strip as necessary to assume that $j=i+1$, in which case we are done by Corollary \ref{cor:adjacentstrictdone}. Therefore, we may assume that there is a minimal noncommuting path in $\bm\lambda$ from $R_i$ to $R_j$. By Proposition \ref{prop:mncpnostrict}, there is either a strict sequence $(R_{j_1},\ldots,R_{j_k})$ in $\bm\lambda$ such that the pairs $(R_{i'},R_{j'})$ are not strict for $j_1\leq i'<j'\leq j_k$, or we have $M_{i,j}=0$ and there is $i<x<j$ with $l(R_x)=l(R_j)$ and $R_j\precnsim R_x$. However, in the latter case, because the pair $(R_i,R_j)$ is strict, then since $M_{i,j}=0$ we must have $M_{i,k}+M_{j,k}\geq |R_k|+1$ for some $k$, but now $M_{i,k}+M_{x,k}\geq M_{i,k}+M_{j,k}\geq |R_k|+1$ so the pair $(R_i,R_x)$ is strict, contradicting minimality of $j-i$. Therefore, we can exclude the second possibility.\\

Because $\bm\lambda$ has a strict sequence $(R_{j_1},\ldots,R_{j_k})$ such that the pairs $(R_{i'},R_{j'})$ are not strict for $j_1\leq i'<j'\leq j_k$, we may assume that $(R_{j_1},\ldots,R_{j_k})$ is such a strict sequence with $j_k-j_1$ minimal, and among such strict sequences, with $k$ minimal. Now if there is a minimal noncommuting path in $\bm\lambda$ from $R_{j_t}$ to $R_{j_{t+1}}$ for any $1\leq t\leq k-1$, then by Proposition \ref{prop:mncpnostrict} again, there is either a strict sequence between $R_{j_t}$ and $R_{j_{t+1}}$, contradicting minimality of $j_k-j_1$, or there is some $j_t<x<j_{t+1}$ with $l(R_x)=l(R_{j_{t+1}})$ and $R_{j_{t+1}}\precnsim R_x$, but in that case by Proposition \ref{prop:strictsequence2} there is either a strict sequence between nearer rows, contradicting minimality of $j_k-j_1$, there is a shorter strict sequence from $R_{j_1}$ to $R_{j_k}$, contradicting minimality of $k$, or there is a strict pair, contradicting our hypothesis. Therefore, there is no minimal noncommuting path in $\bm\lambda$ from $R_{j_t}$ to $R_{j_{t+1}}$ for any $1\leq t\leq k-1$. \\

Finally, let $\bm\mu=(S_1,\ldots,S_n)$ and $\varphi:\Pi(\bm\lambda)\xrightarrow\sim\Pi(\bm\mu)$. By cycling, we may assume without loss of generality that $\varphi_h>\varphi_{j_t}$ for every $1\leq t\leq k$. Because the $M_{j_t,j_{t+1}}(\bm\lambda)=0$ for $1\leq t\leq k-1$, the integers $l(S_{\varphi_{j_t}})$ for $1\leq t\leq k$ must be distinct, so let $\sigma:\{1,\ldots,k\}\to\{1,\ldots,k\}$ be the permutation that sorts them in increasing order, in other words \begin{equation}l(S_{\varphi_{j_{\sigma_1}}})<l(S_{\varphi_{j_{\sigma_2}}})<\cdots<l(S_{\varphi_{j_{\sigma_k}}}).\end{equation}
Now the sequence $(S_{\varphi_{j_{\sigma_1}}},\ldots,S_{\varphi_{j_{\sigma_k}}})$ is strict, so by Proposition \ref{prop:strict} we have $l(S_{\varphi_{j_{\sigma_{t+1}}}})=r(S_{\varphi_{j_{\sigma_t}}})+1$ for every $1\leq t\leq k-1$. Therefore, by Lemma \ref{lem:strictsequencedone} there exists a good substitute for $(\bm\lambda,\bm\mu)$. This completes the proof.
\end{proof}

\begin{corollary}\label{cor:strictsequencedone}
Let $\bm\lambda=(R_1,\ldots,R_n)$ be a horizontal-strip that satisfies \eqref{eq:IH}. If $\bm\lambda$ has a strict sequence $(R_{j_1},\ldots,R_{j_k})$, then $\bm\lambda$ is good. \end{corollary}

\begin{proof}
By Corollary \ref{cor:strictdone}, we may assume that $\bm\lambda$ has no strict pairs, after which the result follows by the argument above. 
\end{proof}

By Corollary \ref{cor:strictdone} and Corollary \ref{cor:strictsequencedone}, we may assume in completing the proof of Lemma \ref{lem:key} that there are no strict pairs or strict sequences in $\bm\lambda$ or any similar horizontal-strip. It will be convenient to make the following definition.

\begin{definition} Let $\bm\lambda=(R_1,\ldots,R_n)$ be a horizontal-strip. We say that $\bm\lambda$ is \emph{nesting} if for every $1\leq i<j\leq n$ we have either $M_{i,j}=0$, $R_i\prec R_j$, or $R_j\prec R_i$, and if $M_{i,j}=0$, then $M_{i,k}+M_{j,k}\leq |R_k|$ for every $k$. \end{definition}

\begin{example}
An example of a nesting horizontal-strip is given in Figure \ref{fig:nesting}. Informally, every pair of rows is either disjoint or one is contained in the other. 
\begin{figure}
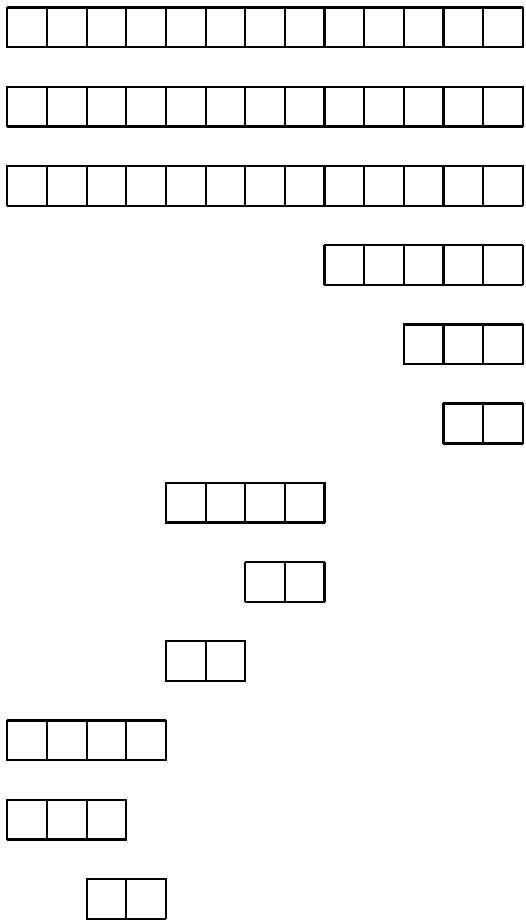
\caption{A nesting horizontal-strip} \label{fig:nesting}
$$\tableau{\ & \ & \ & \ & \ & \ & \ & \ & \ & \ & \ & \ & \ \\ \\ \ & \ & \ & \ & \ & \ & \ & \ & \ & \ & \ & \ & \ \\ \\ \ & \ & \ & \ & \ & \ & \ & \ & \ & \ & \  & \ & \ \\ \\ &&&&&&&& \ & \ & \ & \ & \ \\ \\ &&&&&&&&&& \ & \ & \ \\ \\ &&&&&&&&&&& \ & \ \\ \\ &&&& \ & \ & \ & \ \\ \\ &&&&&& \ & \ \\ \\ &&&& \ & \ \\ \\ \ & \ & \ & \ \\ \\ \ & \ & \ \\ \\ && \ & \  }$$
\end{figure}
\end{example}

\begin{corollary} \label{cor:notnestingdone}
Let $\bm\lambda=(R_1,\ldots,R_n)$ be a horizontal-strip that satisfies \eqref{eq:IH}. If $\bm\lambda$ is not nesting, then $\bm\lambda$ is good. \end{corollary}

\begin{proof}
If $\bm\lambda$ is not nesting, then there is some $1\leq i<j\leq n$ with either $0<M_{i,j}<\min\{|R_i|,|R_j|\}$ or $M_{i,j}=0$ and $M_{i,k}+M_{j,k}\geq |R_k|+1$ for some $k$. By rotating, we may assume that $l(R_i)<l(R_j)$ so that the pair $(R_i,R_j)$ is strict, and then the result follows from Corollary \ref{cor:strictdone}. 
\end{proof}

We now explore some properties of a nesting horizontal-strip. 

\begin{proposition} \label{prop:nesting}
Let $\bm\lambda=(R_1,\ldots,R_n)$ be a nesting horizontal-strip. 
\begin{enumerate}
\item The pairs $(R_i,R_j)$ are not strict for $1\leq i<j\leq n$.\\

\item If $R_i\prec R_j$ and $R_j\prec R_k$, then $R_i\prec R_k$. In other words, the relation $\prec$ is transitive on the rows of $\bm\lambda$.\\

\item If $R_i\prec R_j$ and $M_{j,k}=0$, then $M_{i,k}=0$. \\

\item We cannot have $i<x<y<j$ with $l(R_x)=l(R_j)=r(R_i)-1=r(R_y)-1$, $R_j\precnsim R_x$, and $R_1\precnsim R_y$.
\end{enumerate}
\end{proposition}

\begin{proof}\hspace{2pt}
\begin{enumerate}
\item This follows directly from the definitions of strictness and nesting.\\

\item Because $M_{i,j}+M_{j,k}=|R_i|+|R_j|\geq|R_j|+1$, then by definition of nesting we cannot have $M_{i,k}=0$, so we must have $R_i\prec R_k$ or $R_k\prec R_i$. If $R_k\prec R_i$, then by Proposition \ref{prop:prec} we have $|R_k|\leq |R_i|\leq |R_j|\leq |R_k|$, so $|R_i|=|R_k|$ and in fact $R_i\prec R_k$ as well.\\

\item Because $M_{j,k}=0$, by definition of nesting we must have $M_{i,j}+M_{i,k}=|R_i|+M_{i,k}\leq |R_i|$, so we must have $M_{i,k}=0$. \\

\item Because the conditions $R_j\precnsim R_x$ and $R_1\precnsim R_y$ imply that $|R_x|\geq 2$ and $|R_y|\geq 2$, then we would have $M_{x,y}=1<\min\{|R_x|,|R_y|\}$, contradicting that $\bm\lambda$ is nesting.
\end{enumerate}
\end{proof}

\begin{proposition}\label{prop:nestingmncp}
Let $\bm\lambda=(R_1,\ldots,R_n)$ be a nesting minimal noncommuting path with $l(R_1)<l(R_n)$, $M_{1,n}=0$, and $R_1\leftrightarrow R_n$. Then we have $l(R_{t+1})=r(R_t)+1$ for every $1\leq t\leq n-1$. 
\end{proposition}

\begin{example}
Informally, Proposition \ref{prop:nestingmncp} states that if $\bm\lambda$ is nesting, a minimal noncommuting path must look like the example below.

$$\begin{tikzpicture}
\draw (-0.25,-1.75) node (0) {$\bm \lambda=$};
\draw (7,-3.25) node (4) {$R_1$} (7,-0.25) node (7) {$R_n$};
\draw (0,-3.5) -- (2,-3.5) -- (2,-3) -- (0,-3) -- (0,-3.5) (0.5,-3.5) -- (0.5,-3) (1,-3.5) -- (1,-3) (1.5,-3.5) -- (1.5,-3) (2,-2.5) -- (3.5,-2.5) -- (3.5,-2) -- (2,-2) -- (2,-2.5) (2.5,-2.5) -- (2.5,-2) (3,-2.5) -- (3,-2) (3.5,-1.5) -- (5.5,-1.5) -- (5.5,-1) -- (3.5,-1) -- (3.5,-1.5) (4,-1.5) -- (4,-1) (4.5,-1.5) -- (4.5,-1) (5,-1.5) -- (5,-1) (5.5,-0.5) -- (6.5,-0.5) -- (6.5,0) -- (5.5,0) -- (5.5,-0.5) (6,-0.5) -- (6,0);
\end{tikzpicture}$$
\end{example}

\begin{proof} [Proof of Proposition \ref{prop:nestingmncp}. ]
If $l(R_{n-1})>l(R_n)$, then because $l(R_1)<l(R_n)$, we must have $l(R_{t-1})<l(R_t)$ for some maximal $2\leq t\leq n-1$. But then we must have $R_{t-1}\precnsim R_{t+1}$ by Proposition \ref{prop:strict3} and therefore $R_1\prec R_{t+1}$ by Proposition \ref{prop:mncp}, Part 1, which is impossible by Proposition \ref{prop:prec} because $l(R_1)<l(R_n)\leq l(R_{t+1})$. Therefore, we must have $l(R_{n-1})<l(R_n)$, and more specifically, because $M_{n-1,n}<\min\{|R_{n-1}|,|R_n|\}$ by Corollary \ref{cor:mrirj}, Part 3 and because $\bm\lambda$ is nesting, we must have $M_{n-1,n}=0$ and $r(R_{n-1})=l(R_n)-1>r(R_1)$. Similarly, by rotating, we must have $l(R_2)=r(R_1)+1<l(R_n)$. Note that if $n=3$, then this proves our claim, so we now use induction on $n\geq 4$.\\

If $l(R_1)<l(R_{n-1})$, then because $R_1\leftrightarrow R_{n-1}$ by minimality of the noncommuting path and therefore $M_{1,n-1}=0$ by Proposition \ref{prop:mrirj}, we have by our induction hypothesis that $l(R_{t+1})=r(R_t)+1$ for every $1\leq t\leq n-2$, which proves our claim. Similarly, by rotating, we are done if $r(R_2)<r(R_n)$. Therefore, the only remaining case to consider is when $l(R_{n-1})\leq l(R_1)$ and $r(R_2)\geq r(R_n)$. However, we would have $R_1\prec R_{n-1}$ and $R_n\prec R_2$ by Proposition \ref{prop:prec}, and \begin{equation}
l(R_{n-1})\leq l(R_1)\leq r(R_1)+1=l(R_2)\leq r(R_{n-1})=l(R_n)-1\leq r(R_n)\leq r(R_2),
\end{equation}
and therefore $M_{2,n-1}>0$ by Proposition \ref{prop:mrirj} Parts 2 and 3. Because $\bm\lambda$ is nesting, we must have either $R_2\prec R_{n-1}$ or $R_{n-1}\prec R_2$, but if $R_2\prec R_{n-1}$, we have $M_{2,n-1}+M_{2,n}=|R_2|+|R_n|\geq |R_2|+1$, and if $R_{n-1}\prec R_2$, we have $M_{1,n-1}+M_{2,n-1}=|R_1|+|R_{n-1}|\geq |R_{n-1}|+1$, a contradiction in either case. This completes the proof.
\end{proof}

We are at long last ready to prove Lemma \ref{lem:key}, which implies Theorem \ref{thm:main}. The strategy is as follows. We will use commuting and cycling to rewrite $\bm\lambda$ so that the bottom two rows are in the desired form. If the corresponding two rows $S_i$ and $S_j$ of $\bm\mu$ do not commute, then we are done by Corollary \ref{cor:adjacentnoncommutingdone}, and if there is no minimal noncommuting path between them, then we can again use commuting and cycling to bring them closer together until they do not commute. If there is a minimal noncommuting path $(S_i=S_{i_1},\ldots,S_{i_k}=S_j),$ then because we may assume that $\bm\mu$ is nesting, Proposition \ref{prop:nestingmncp} specifies the structure of these rows. In this case, we will use an argument similar to that of the proof of Lemma \ref{lem:strictsequencedone} to locally rotate pairs of rows $(S_{i_{t-1}},S_{i_t})$ of $\bm\mu$ to again bring $S_j$ toward $S_i$ until they do not commute. 

\begin{proof}[Proof of Lemma \ref{lem:key}. ]
We first note that by Corollary \ref{cor:strictsequencedone} and Corollary \ref{cor:notnestingdone}, we may assume that $\bm\lambda$ is nesting and has no strict pairs or strict sequences. We will first replace $\bm\lambda$ by a similar horizontal-strip $\bm\lambda'=(R'_1,\ldots,R'_n)\in\mathcal S(\bm\lambda)$ with $l(R'_1)<l(R'_2)$ and $R'_1\nleftrightarrow R'_2$. Because $n(\bm\lambda)-M(\bm\lambda)\geq 1$, by \eqref{eq:Mn} we have $M_{i,j}(\bm\lambda)<\min\{|R_i|,|R_j|\}$ for some $1\leq i<j\leq n$, and because $\bm\lambda$ is nesting, we must in fact have $M_{i,j}(\bm\lambda)=0$. By rotating and cycling, we may assume without loss of generality that $i=1$ and $l(R_1)<l(R_j)$, and then $l(R_j)\geq r(R_1)+1$ by Proposition \ref{prop:mrirj}, Parts 2 and 3. \\

Suppose that $l(R_j)-r(R_1)=1$. If there is a minimal noncommuting path in $\bm\lambda$ from $R_1$ to $R_j$, then by Proposition \ref{prop:mncpnostrict}, either $\bm\lambda$ has a strict sequence, contradicting our assumption, or there is $1<x<y<j$ with $l(R_x)=l(R_j)=r(R_i)-1=r(R_y)-1$, $R_j\precnsim R_x$, and $R_1\precnsim R_y$, contradicting that $\bm\lambda$ is nesting by Proposition \ref{prop:nesting}, Part 4. Therefore, there is no minimal noncommuting path in $\bm\lambda$ from $R_1$ to $R_j$, so by Lemma \ref{lem:aoncp} we can find our desired horizontal-strip $\bm\lambda'$. We now suppose that $l(R_j)-r(R_1)\geq 2$, which means that $R_1\leftrightarrow R_j$ by Proposition \ref{prop:mrirj}, Part 1, and we use induction on $l(R_j)-r(R_i)$. \\

If there is a minimal noncommuting path $(R_1=R_{i_1},\ldots,R_{i_k}=R_j)$ in $\bm\lambda$ from $R_1$ to $R_j$, then because $R_1\leftrightarrow R_j$ we have by Proposition \ref{prop:nestingmncp} that $l(R_{i_2})-r(R_1)=1$ and we can repeat the above argument with the rows $R_1$ and $R_{i_2}$. If there is no minimal noncommuting path in $\bm\lambda$ from $R_1$ to $R_j$, then by commuting and cycling we have $(R_1,\ldots,R_{j-1},R_{j+1},\ldots,R_n,R_j^-)\in\mathcal S(\bm\lambda)$ and $l(R_j^-)-r(R_1)<l(R_j)-r(R_1)$, so we are done by our induction hypothesis on $l(R_j)-r(R_1)$. Therefore, there is indeed a horizontal-strip $\bm\lambda'=(R'_1,\ldots,R'_n)\in\mathcal S(\bm\lambda)$ with $l(R'_1)<l(R'_2)$ and $R'_1\nleftrightarrow R'_2$. \\

We will now strengthen the conditions on our choice of $\bm\lambda'$. Consider the set 
\begin{equation}\mathcal S^*(\bm\lambda)=\{\bm\lambda'=(R'_1,\ldots,R'_n)\in\mathcal S(\bm\lambda): \ l(R'_1)<l(R'_2), \ R'_1\nleftrightarrow R'_2\}\end{equation}
and for $\bm\lambda'=(R'_1,\ldots,R'_n)\in\mathcal S^*(\bm\lambda)$, define the integer \begin{equation}h(\bm\lambda')=|\{3\leq t\leq n: \ R'_1\prec R'_t, \ R'_2\prec R'_t\}|.\end{equation}
Because we have shown that the set $\mathcal S^*(\bm\lambda)$ is nonempty and because we have a uniform bound $h(\bm\lambda')\leq n-2$, we may let $\bm\lambda'=(R'_1,\ldots,R'_n)\in\mathcal S^*(\bm\lambda)$ be such that $h(\bm\lambda')$ is maximal, and among those, with $|R'_2|$ maximal. Let $\bm\mu=(S_1,\ldots,S_n)$ and $\varphi:\Pi(\bm\lambda')\xrightarrow\sim\Pi(\bm\mu)$, and note that again by Corollary \ref{cor:strictsequencedone} and Corollary \ref{cor:notnestingdone} we may assume that $\bm\mu$ is nesting and has no strict pairs or strict sequences. \\

Let $i=\varphi_1$ and $j=\varphi_2$, and note that by cycling and rotating we may assume that $i<j$ and $l(S_i)<l(S_j)$, and then because $M_{i,j}(\bm\mu)=0$, we have $l(S_j)\geq r(S_i)+1$ by Proposition \ref{prop:mrirj}, Parts 2 and 3. If $l(S_j)-r(S_i)=1$, then $S_i\nleftrightarrow S_j$ by Proposition \ref{prop:mrirj}, Part 3, so there exists a good substitute for $(\bm\lambda,\bm\mu)$ by Corollary \ref{cor:adjacentnoncommutingdone} and we would be done. We now suppose that $l(S_j)-r(S_i)\geq 2$, which means that $S_i\leftrightarrow S_j$ by Proposition \ref{prop:mrirj}, Part 1, and we use induction on $l(S_j)-r(S_i)$. If there is no minimal noncommuting path in $\bm\mu$ from $S_i$ to $S_j$, then by commuting and cycling we have $(S_i,\ldots,S_{j-1},S_{j+1},\ldots,S_n,S_1^-,\ldots,S_j^-)\in\mathcal S(\bm\mu)$ and $l(S_j^-)-r(S_i)<l(S_j)-r(S_i)$, so we are done by our induction hypothesis on $l(S_j)-r(S_i)$. Therefore, we may assume that there is a minimal noncommuting path $(S_i=S_{i_1},\ldots,S_{i_k}=S_j)$ in $\bm\mu$ from $S_i$ to $S_j$. Because $\bm\mu$ is nesting and $S_i\leftrightarrow S_j$, we have that $l(S_{i_{t+1}})=r(S_{i_t})+1$ for every $1\leq t\leq k-1$ by Proposition \ref{prop:nestingmncp}. Also note that if there is a minimal noncommuting path in $\bm\mu$ from $S_{i_t}$ to $S_{i_{t+1}}$ for any $1\leq t\leq k-1$, then by Proposition \ref{prop:mncpnostrict} either $\bm\mu$ has a strict sequence, contradicting our assumption, or there is $1<x<y<j$ with $l(S_x)=l(S_j)=r(S_1)-1=r(S_y)-1$, $S_j\precnsim S_x$, and $S_1\precnsim S_y$, contradicting that $\bm\mu$ is nesting by Proposition \ref{prop:nesting}, Part 4. Therefore, there is no minimal noncommuting path in $\bm\mu$ from $S_{i_t}$ to $S_{i_{t+1}}$ for any $1\leq t\leq k-1$,\\

We now make the following useful observation. For every row $R'_t$ of $\bm\lambda$ with $R'_1\prec R'_t$ and $R'_2\prec R'_t$, we have $S_i\prec S_{\varphi_t}$ and $S_j\prec S_{\varphi_t}$ and therefore by Proposition \ref{prop:prec} we have 
\begin{equation}l(S_{\varphi_t})\leq l(S_1)+1\leq r(S_1)+1=l(S_{i_2})\leq r(S_{i_{k-1}})=l(S_j)-1\leq r(S_j)-1\leq r(S_{\varphi_t})\end{equation}
and therefore $S_{i_{t'}}\prec S_{\varphi_t}$ for every $1\leq t'\leq k$. \\

Now suppose that there is some $t$ with $S_{i_{k-1}}\precnsim S_t$ and $M_{j,t}(\bm\mu)=0$. By Proposition \ref{prop:prec}, we must have $r(S_t)\leq r(S_{i_{k-1}})$ and $l(S_t)\leq l(S_{i_{k-1}})-1=r(S_{i_{k-2}})$, so $M_{i_{k-2},t}(\bm\mu)>0$. Because $\bm\lambda$ is nesting, this means that either $S_{i_{k-2}}\prec S_t$ or $S_t\prec S_{i_{k-2}}$, but if $S_t\prec S_{i_{k-2}}$ we would have $M_{i_{k-2},t}(\bm\mu)+M_{i_{k-1},t}(\bm\mu)=|S_t|+|S_{i_{k-1}}|\geq|S_t|+1$ and the pair $(S_{i_{k-2}},S_{i_{k-1}})$ would be strict, a contradiction by Proposition \ref{prop:nesting}, Part 1. Therefore, we must have $S_{i_{k-2}}\prec S_t$. Because there is no minimal noncommuting path in $\bm\mu$ from $S_{i_{k-2}}$ to $S_{i_{k-1}}$, we may cycle and use Lemma \ref{lem:aoncp} to replace $\bm\mu$ with a similar horizontal-strip to assume that $i_{k-2}=1$ and $i_{k-1}=2$, but now we would have $h(\bm\mu)>h(\bm\lambda')$ because $S_t$ is counted only by $h(\bm\mu)$, contradicting maximality of $h(\bm\lambda')$. Therefore, we may assume that \begin{equation}\label{eq:keydone1} \text{ there is no }t \text{ with }S_{i_{k-1}}\precnsim S_t\text{ and } M_{j,t}(\bm\mu)=0.\end{equation}
Similarly, now suppose that there is some $t$ with $S_j\precnsim S_t$ and $M_{i_{k-1},t}(\bm\mu)=0$. Because there is no minimal noncommuting path in $\bm\mu$ from $S_{i_{k-1}}$ to $S_j$, we may cycle and use Lemma \ref{lem:aoncp} to replace $\bm\mu$ with a similar horizontal-strip to assume that $i_{k-1}=1$ and $j=2$. Because $S_j\precnsim S_t$ and $M_{i_{k-1},t}(\bm\mu)=0$, by Proposition \ref{prop:prec}, we have $l(S_t)=l(S_j)$ and $|S_t|>|S_j|$. Now if there is a minimal noncommuting path in $\bm\mu$ from $S_1$ to $S_t$, then as before, by Proposition \ref{prop:mncpnostrict} either $\bm\mu$ has a strict sequence, contradicting our assumption, or we contradict that $\bm\mu$ is nesting. Therefore, there is no minimal noncommuting path in $\bm\mu$ from $S_1$ to $S_t$, so by Lemma \ref{lem:aoncp} we may replace $\bm\mu$ by a similar horizontal-strip to instead assume that $t=2$. We also note that for every row $S_{t'}$ of $\bm\mu$ with $S_1\prec S_{t'}$ and $S_j\prec S_{t'}$, we must have $M_{t,t'}(\bm\mu)>0$ and therefore either $S_t\prec S_{t'}$ or $S_{t'}\prec S_t$. However, if $S_{t'}\prec S_t$, then we would have $M_{1,t'}(\bm\mu)+M_{t,t'}(\bm\mu)=|S_1|+|S_{t'}|\geq|S_{t'}|+1$ and the pair $(S_1,S_t)$ would be strict, a contradiction, and therefore $S_t\prec S_{t'}$. However, we now have $h(\bm\mu)\geq h(\bm\lambda')$ and $|S_2|>|S_j|=|R'_2|$, contradicting either the maximality of $h(\bm\lambda')$ or the maximality of $|R'_2|$. Therefore, we may assume that \begin{equation}\label{eq:keydone2} \text{ there is no }t\text{ with }S_j\precnsim S_t\text{ and }M_{i_{k-1},t}(\bm\mu)=0.\end{equation}

Because there is no minimal noncommuting path in $\bm\mu$ from $S_{i_{k-1}}$ to $S_j$, we may use Lemma \ref{lem:aoncp} to replace $\bm\mu$ by a similar horizontal-strip to assume that $j=i_{k-1}+1$. Our plan is now to apply Lemma \ref{lem:localrotation} to the rows $S_{j-1}$ and $S_j$ to replace $\bm\mu$ by a similar horizontal-strip for which $l(S_j)$ has decreased and $r(S_i)$ has not changed, so that we will be done by our induction hypothesis on $l(S_j)-r(R_i)$. It remains to check the conditions of Proposition \ref{prop:localrotationhypothesis}.
\begin{enumerate}
\item If $M_{j-1,t}(\bm\mu)>0$ and $M_{j,t}(\bm\mu)>0$, then because $\bm\mu$ is nesting we must have $S_{j-1}\prec S_t$ or $S_t\prec S_{j-1}$. If $S_t\prec S_{j-1}$, then $M_{j-1,t}(\bm\mu)+M_{j,t}(\bm\mu)=|S_t|+M_{j,t}(\bm\mu)\geq |S_t|+1$ and the pair $(S_{j-1},S_j)$ would be strict, contradicting that $\bm\mu$ is nesting. Therefore, we must have $S_{j-1}\prec S_t$ and similarly we must have $S_j\prec S_t$. \\

\item If $M_{j-1,t}(\bm\mu)>0$ and $M_{j,t}(\bm\mu)=0$, then because $\bm\mu$ is nesting we must have $S_{j-1}\prec S_t$ or $S_t\prec S_{j-1}$, but by \eqref{eq:keydone1} we cannot have $S_{j-1}\precnsim S_t$, so we must have $S_t\prec S_{j-1}$. \\

\item If $M_{j-1,t}(\bm\mu)=0$ and $M_{j,t}(\bm\mu)>0$, then because $\bm\mu$ is nesting we must have $S_j\prec S_t$ or $S_t\prec S_j$, but by \eqref{eq:keydone2} we cannot have $S_j\precnsim S_t$, so we must have $S_t\prec S_j$.\\

\item If $S_t\prec S_{j-1}$, then by Proposition \ref{prop:nesting}, Part 3, we must have $M_{j,t}(\bm\mu)=0$ and $M_{a,t}(\bm\mu)=0$ for every $a\in A$. Moreover, if $b\in B$, then $S_{j-1}\prec S_b$ and therefore $S_t\prec S_b$ because by Proposition \ref{prop:nesting}, Part 2, the relation $\prec$ is transitive.\\

\item If $S_t\prec S_j$, then by Proposition \ref{prop:nesting}, Part 3, we must have $M_{j-1,t}(\bm\mu)=0$ and $M_{a,t}(\bm\mu)=0$ for every $a\in A$. Moreover, if $b\in B$, then $S_j\prec S_b$ and therefore $S_t\prec S_b$ because by Proposition \ref{prop:nesting}, Part 2, the relation $\prec$ is transitive. \\
\end{enumerate}

This concludes our verification of the conditions of Proposition \ref{prop:localrotationhypothesis}. Therefore, the result follows by Lemma \ref{lem:localrotation} and our induction hypothesis on $l(S_j)-r(S_i)$. This completes the proof of Lemma \ref{lem:key}, and therefore our proof of Theorem \ref{thm:main}. 
\end{proof}

\section{Acknowledgements}
The author would like to thank Mark Haiman for his helpful comments.

\end{document}